\documentclass[11pt]{amsart}
\usepackage{latexsym,amsmath,amsfonts,amscd,amssymb}
\usepackage{stmaryrd}
\usepackage{hyperref}
\usepackage{array}   
\usepackage[mathscr]{eucal} 
\usepackage{mathrsfs}
\usepackage{graphicx}
\usepackage{nicematrix}
\usepackage{xy}
\usepackage{calligra}
\usepackage[T1]{fontenc}
\usepackage[nocompress]{cite}
\DeclareFontFamily{OT1}{pzc}{}
\DeclareFontShape{OT1}{pzc}{m}{it}{<-> s * [1.10] pzcmi7t}{}
\DeclareMathAlphabet{\mathpzc}{OT1}{pzc}{m}{it}
\DeclareFontFamily{OT1}{rsfs}{}
\DeclareFontShape{OT1}{rsfs}{n}{it}{<->rsfs10}{}
\DeclareMathAlphabet{\curly}{OT1}{rsfs}{n}{it}
\theoremstyle{plain}
\newtheorem{theorem}{Theorem}[section]
\newtheorem*{theorem*}{Theorem}

\newtheorem{lemma}[theorem]{Lemma}
\newtheorem{proposition}[theorem]{Proposition}
\newtheorem{conjecture}[theorem]{Conjecture}
\newtheorem*{conjecture*}{Conjecture}
\newtheorem*{proposition*}{Proposition}
\theoremstyle{definition}
\newtheorem{definition}[theorem]{Definition}
\theoremstyle{remark}
\newtheorem{example}[theorem]{Example}
\newtheorem{remark}[theorem]{Remark}

\newtheorem*{claim*}{Claim}
\numberwithin{equation}{section}

\renewcommand{\le}{\leqslant}

\renewcommand{\ge}{\geqslant}
\renewcommand{\setminus}{\smallsetminus}

\newcommand{\R}{\mathbb{R}}

\newcommand{\Z}{\mathbb{Z}}
\newcommand{\C}{\mathbb{C}}

\DeclareMathOperator*{\heightr}{ht}

\newcommand{\Waff}[1][]{
	W_{\mathrm{aff}
		\if\relax\detokenize{#1}\relax
		\else
		,#1
		\fi
	}
}
\newcommand{\gen}[1]{\left< #1 \right>}

\newcommand{\lie}{\mathfrak}

\newcommand{\PGL}{\mathrm{PGL}}

\DeclareMathOperator{\ord}{ord}

\DeclareMathOperator{\Spec}{Spec}
\DeclareMathOperator{\ad}{ad}
\DeclareMathOperator{\Ad}{Ad}

\DeclareMathOperator{\rank}{rank}

\DeclareMathOperator{\Hom}{Hom}

\DeclareMathOperator{\Id}{Id}

\DeclareMathOperator{\Sym}{Sym}

\DeclareMathOperator{\Aut}{Aut}

\DeclareMathOperator{\Fl}{Fl}

\DeclareMathOperator*{\GL}{GL}
\DeclareMathOperator*{\SL}{SL}

\DeclareMathOperator*{\Prym}{Prym}
\DeclareMathOperator*{\Pic}{Pic}

\newcommand{\pdeg}{\operatorname{pardeg}}

\newcommand{\Gr}{\operatorname{Gr}}
\renewcommand{\phi}{\varphi}
\newcommand{\git}{\mathbin{/\mkern-6mu/}}

\newcommand{\liep}{\mathfrak{p}}
\newcommand{\liet}{\mathfrak{t}}

\newcommand{\lieg}{\mathfrak{g}}

\newcommand{\liel}{\mathfrak{l}}

\newcommand{\liesl}{\mathfrak{sl}}

\renewcommand{\phi}{\varphi}

\usepackage{tikz-cd}
\usetikzlibrary{calc}
\tikzset{curve/.style={settings={#1},to path={(\tikztostart)
			.. controls ($(\tikztostart)!\pv{pos}!(\tikztotarget)!\pv{height}!270:(\tikztotarget)$)
			and ($(\tikztostart)!1-\pv{pos}!(\tikztotarget)!\pv{height}!270:(\tikztotarget)$)
			.. (\tikztotarget)\tikztonodes}},
	settings/.code={\tikzset{quiver/.cd,#1}
		\def\pv##1{\pgfkeysvalueof{/tikz/quiver/##1}}},
	quiver/.cd,pos/.initial=0.35,height/.initial=0}
\tikzset{tail reversed/.code={\pgfsetarrowsstart{tikzcd to}}}
\tikzset{2tail/.code={\pgfsetarrowsstart{Implies[reversed]}}}
\tikzset{2tail reversed/.code={\pgfsetarrowsstart{Implies}}}
\tikzset{no body/.style={/tikz/dash pattern=on 0 off 1mm}}

\begin{document}
	\title[Very stable parabolic $G$-Higgs bundles and affine flag varieties]{Very stable parabolic $G$-Higgs bundles \\ and affine flag varieties}
	\author[Miguel González]{Miguel González}
	\address{Instituto de Ciencias Matem\'aticas \\
		CSIC-UAM-UC3M-UCM \\ Nicol\'as Cabrera, 13--15 \\ 28049 Madrid \\ Spain}
	\email{miguel.gonzalez@icmat.es}
	\noindent
	\thanks{ \noindent The project that gave rise to these results received the support of a fellowship from ``la Caixa'' Foundation (ID 100010434). The fellowship code is LCF/BQ/DR23/12000030. The work was also partially supported by the Spanish Ministry of Science, Innovation, and Universities through the ``Severo Ochoa Programme for Centres of Excellence in R{\&}D (CEX2023-001347-S)'' and grant PID2022-141387NB-C21, as well as ILINK25127 CSIC grant. This work was also partially supported by ERC grant (ViaFiPos, 101199663). Views and opinions expressed are however those of the author only and do not necessarily reflect those of the European Union or the European Research Council Executive Agency. Neither the European Union nor the granting authority can be held responsible for them.
	}
	\begin{abstract}
	Let $G$ be a semisimple complex algebraic group. We study the $\C^\times$-action on the moduli space of strongly parabolic $G$-Higgs bundles over a punctured curve with weights in the interior of the standard Weyl alcove. We describe its fixed points and, for the components with generically regular Higgs field, we classify the fixed points that have closed Bia\l ynicki-Birula upward flows, called \textit{very stable}. The classification is stated in terms of a naturally associated divisor valued in the extended affine Weyl group of $G$ which encodes the information of the fixed point. Under this point of view, very stable fixed points naturally coincide with divisors where the coefficients are minimal under the Bruhat order. Furthermore, when the divisor is supported at the parabolic punctures, the corresponding very stable upward flow is a section of the Hitchin map. For these upward flows, we propose a conjectural mirror line bundle over the dual moduli space, and verify its validity for $G=\PGL_n(\C)$.
	\end{abstract}
	
	\maketitle
	
	\section{Introduction}\label{secintro}
	
	The theory of parabolic Higgs bundles over a smooth projective complex curve $C$ was initiated by Simpson \cite{simpsonHarmonicBundlesNoncompact1990} motivated by finding a generalisation of the non-abelian Hodge correspondence of Corlette, Donaldson, Hitchin and Simpson \cite{corlette_flat_1988, donaldson_twisted_1987,hitchin_self-duality_1987,simpson_higgs_1992} to the setting with punctures $D = c_1+\dots+c_s$ in $C$, in a similar way to how the theory of parabolic bundles of Mehta--Seshadri \cite{mehtaModuliVectorBundles1980} generalises the theorem of Narasimhan--Seshadri \cite{narasimhan_stable_1965} in the presence of parabolic punctures. Roughly speaking, given a choice of weights $\alpha = (\alpha_1,\dots,\alpha_s)$, each of which is a vector of $n$ real numbers, there are notions of stability and a quasiprojective moduli space $\mathcal M(\GL_n(\C),\alpha)$, constructed by Biquard \cite{biquardFibresParaboliquesStables1991}, Konno \cite{konnoConstructionModuliSpace1993} and Yokogawa \cite{yokogawaCompactificationModuliParabolic1993}, satisfying the existence of a homeomorphism
	$$\mathcal M(\GL\nolimits_n(\C),\alpha) \simeq \left(\Hom(\pi_1(C \setminus D), \GL\nolimits_n(\C))\git \GL\nolimits_n(\C)\right)_\alpha,$$
	\noindent where the subscript denotes the subvariety where the compact part (i.e. the factor belonging to $U(n)$ in the Cartan decomposition of $\GL_n(\C)$) of the monodromies at the punctures is determined by $\alpha$. Since then, the theory has been generalised to accommodate for arbitrary complex or real reductive structure groups \cite{biquard_parabolic_2019}. 
	
	We consider $G$ to be complex semisimple and work with the corresponding moduli space $\mathcal M(G,\alpha)$, related via non-abelian Hodge theory to representations into $G$. The space $\mathcal M(G,\alpha)$ presents a very rich geometry \cite[Proposition 7.1]{biquard_parabolic_2019}, \cite{bottacinSymplecticGeometryModuli1995,markmanSpectralCurvesIntegrable1994,logaresModuliParabolicHiggs2010}: it is endowed with a Poisson structure foliated by symplectic leaves corresponding to prescribing certain extra information related to the non-compact part of the monodromies around the punctures of the corresponding representation. Each symplectic leaf is furthermore endowed with a hyperkähler structure. In this paper we focus on the leaf of \textbf{strongly parabolic} Higgs bundles
	$$\mathcal M_{sp}(G,\alpha) \subseteq \mathcal M(G,\alpha).$$
	It is equipped with an algebraically completely integrable system
	$$h_G : \mathcal M_{sp}(G,\alpha) \to \mathcal A(G)$$
	\noindent onto a vector space, known as the \textbf{Hitchin system} (originally introduced by Hitchin \cite{hitchin_stable_1987} in the non-parabolic setting), whose geometry was studied in detail in \cite{bottacinSymplecticGeometryModuli1995,markmanSpectralCurvesIntegrable1994}.  
	
	As it is the case for moduli spaces of usual (non-parabolic) $G$-Higgs bundles, the space $\mathcal M_{sp}(G,\alpha)$ plays an important role in the setting of mirror symmetry and Langlands duality for Higgs bundles initiated by Hausel--Thaddeus \cite{hausel_mirror_2003}, who already considered the parabolic setting in their original announcement. The physical aspects of this duality and their relation to the geometric Langlands program were studied in \cite{kapustin_electric-magnetic_2007} (for the non-parabolic setting) and \cite{gukovGaugeTheoryRamification2008} (for the parabolic setting, related to the ramified geometric Langlands program). It is a duality between $\mathcal M_{sp}(G,\alpha)$ and $\mathcal M_{sp}(G^\vee, \alpha^\vee)$, where $G^\vee$ denotes the Langlands dual group, which manifests in several different ways. One is \textbf{topological mirror symmetry}, i.e. the agreement of certain invariants (\textit{stringy} $E$-polynomials), established for $G=\SL_2(\C)$ and $G=\SL_3(\C)$ by Gothen and Oliveira \cite{gothenTopologicalMirrorSymmetry2019}, and later for $G=\SL_n(\C)$ by Shen \cite{shenMirrorSymmetryParabolic2024} and Su--Wang--Wen \cite{suTopologicalMirrorSymmetry2022} using the $p$-adic integration techniques \cite{groechenigMirrorSymmetryModuli2020} that were originally applied to settle this form of the duality in the non-parabolic setting. Another form is \textbf{Strominger--Yau--Zaslow (SYZ) mirror symmetry}, i.e. a duality of abelian fibrations
	\[\begin{tikzcd}[ampersand replacement=\&]
		{\mathcal M_{sp}(G,\alpha)} \&\& {\mathcal M_{sp}(G^\vee,\alpha^\vee)} \\
		\\
		\& {\mathcal A(G) \simeq \mathcal A(G^\vee)}
		\arrow["{h_G}", from=1-1, to=3-2]
		\arrow["{h_{G^\vee}}"', from=1-3, to=3-2]
	\end{tikzcd}\]
	
	\noindent such that the generic fibres are dual abelian varieties. Originally considered in the non-parabolic setting for $G=\SL_n(\C)$ by Hausel--Thaddeus \cite{hausel_mirror_2003} and general $G$ by Donagi--Pantev \cite{donagi_langlands_2012}, it has been proven in the parabolic setting for $G=\SL_n(\C)$ by Biswas--Dey \cite{biswasSYZDualityParabolic2012} and there has also been recent work for type $B$ and $C$ by Wang--Wen--Wen \cite{wangSpringerCorrespondenceMirror2025}.
	
	In the non-parabolic setting, the work of Donagi--Pantev \cite{donagi_langlands_2012} establishes a \textit{classical limit} of the geometric Langlands correspondence, i.e. an equivalence of categories
	$$D_c^b(\mathcal M(G)) \simeq D_c^b(\mathcal M(G^\vee))$$
	\noindent generically over $\mathcal A(G)$, obtained by a fibrewise Fourier--Mukai transform. Taking into account the hyperkähler structure, this duality should \cite[Section 12]{kapustin_electric-magnetic_2007} exchange \textbf{BAA-branes} (which include local systems supported on holomorphic Lagrangian subvarieties) on $\mathcal M(G)$ with \textbf{BBB-branes} (hyperholomorphic connections) on $\mathcal M(G^\vee)$. Motivated by obtaining tractable examples of this phenomenon, Hausel and Hitchin \cite{hauselEnhancedMirrorSymmetry2022,hausel_very_2022} introduced the notion of a \textbf{very stable} $\GL_n(\C)$-Higgs bundle, extending the notion of a very stable vector bundle of Drinfeld and Laumon \cite{laumon_analogue_1988}. Such a very stable Higgs bundle is associated to a particularly simple BAA-brane supported on a closed holomorphic Lagrangian and closely related to the geometry of the nilpotent cone $h_G^{-1}(0)$. The \textbf{main purpose} of the current paper is to study and classify the same very stable Lagrangians in the setting of $\mathcal M_{sp}(G,\alpha)$.
	
	Now we introduce the setting and problem more precisely. Let $\lieg$ denote the Lie algebra of $G$ and $K_C$ the canonical line bundle of $C$. Fix a maximal torus and Borel subgroup $T \subseteq B \subseteq G$. For simplicity, we choose parabolic weights $\alpha_i$ in the interior of the standard Weyl alcove in the Cartan subalgebra $\liet_{\R} \subseteq \liet$ of a split real form of $G$, although much of the work here generalises to arbitrary weights in $\liet_\R$ by working with parahoric torsors and logahoric Higgs bundles \cite{balajiModuliParahorictorsors2015, kydonakisLogahoricHiggsTorsors2024}, as we indicate throughout the paper. Then, a strongly parabolic $G$-Higgs bundle is a triple $(E,\varphi,Q)$ where $E$ is a principal $G$-bundle over $C$, $Q=(Q_1,\dots,Q_s)$ is a choice of a parabolic structure $Q_i \in E|_{c_i}/B$ at each parabolic puncture, and $\varphi \in H^0(E(\lieg) \otimes K_C(D))$ is a Higgs field whose restriction to the punctures is nilpotent and compatible with $Q$.
	
	The holomorphic Lagrangian subvarieties in question are obtained from the Bia\l ynicki-Birula cells of the natural $\C^\times$-action, where $\lambda \in \C^\times$ maps
	$$(E,\varphi,Q) \mapsto (E,\lambda \varphi, Q).$$
	Given a fixed point $(E,\varphi,Q) \in \mathcal M_{sp}(G,\alpha)^{\C^\times}$, its \textbf{upward flow}
	$$W^+_{(E,\varphi,Q)} := \left\{(E',\varphi',Q') \in \mathcal M_{sp}(G,\alpha) : \lim\limits_{\lambda \to 0}(E',\lambda \varphi', Q') = (E,\varphi,Q)\right\} \subseteq \mathcal M_{sp}(G,\alpha)$$
	\noindent turns out to be a holomorphic Lagrangian subvariety whenever $(E,\varphi,Q)$ is a smooth point. Following Hausel and Hitchin \cite{hausel_very_2022}, we say that a smooth fixed point $(E,\varphi,Q)$ is \textbf{very stable} if $W^+_{(E,\varphi,Q)}$ contains no nilpotent parabolic $G$-Higgs bundles other than $(E,\varphi,Q)$ itself, which is equivalent to $W^+_{(E,\varphi,Q)}$ being closed. A smooth fixed point which is not very stable is termed \textbf{wobbly}.
	
	Very stable and wobbly Higgs bundles have been studied extensively in the non-parabolic setting in multitude of works. For $G=\GL_n(\C)$, the original work of Laumon \cite{laumon_analogue_1988} considers the case of very stable vector bundles, that is, when $\varphi = 0$. In this work, it is shown that the very stable locus is open and dense, and recently Pal \cite{palConjectureDrinfeld2026} showed that the wobbly locus is of pure codimension one. At the other extreme, the work of Hausel and Hitchin \cite{hausel_very_2022} considers fixed points where $\varphi$ is generically regular and fully classifies them in terms of a natural divisor over $C$ associated with such a fixed point. In the same setting, \cite{gonzalez_even_2024} gives a similar result restricted to the fixed point locus of the involution given by the action of $\lambda = -1$, which is related \cite{garcia-prada_involutions_2019} to Higgs bundles for the real structure group $U(n,n)$ and $U(n,n+1)$. For non-regular, non-zero Higgs fields there is work of Peón-Nieto \cite{peon-nieto_wobbly_2023} establishing the wobbliness of many of the components where $E=E_0\oplus E_1$ and $\varphi : E_0 \to E_1\otimes K_C$. Finally, for general $G$, the case with zero Higgs field was studied by Zelaci \cite{zelaci_very_2018} and the case with regular Higgs field was carried out in \cite{gonzalez_very_2025}, providing a full classification in terms of a coweight-valued divisor on $C$ in the spirit of the work of Hausel and Hitchin.
	
	Meanwhile, in the setting of strongly parabolic Higgs bundles, for $G=\GL_n(\C)$ the case of very stable parabolic bundles (i.e. with zero Higgs field) was analysed in work of Peón-Nieto \cite{peon-nietoEqualityWobblyShaky2024}, and the generically regular $\varphi$ case has been considered in the PhD thesis of Henke \cite{henkeQuantumGeometryParabolic2024}. The goal of the present work is to establish a simple classification criterion for general $G$ and generically regular $\varphi$ in terms of a divisor valued in the extended affine Weyl group $\widetilde{W}$ for $\lieg$ by exploiting the Bruhat decomposition of affine flag varieties.
	
	Classifying very stable parabolic $G$-Higgs bundles requires understanding the fixed point locus $\mathcal M_{sp}(G,\alpha)^{\C^\times}$. We provide a description in terms of \textbf{Vinberg $\C^\times$-pairs} similar to that of \cite{biquard_arakelov-milnor_2021}. Given a $\Z$-grading of the Lie algebra
	$$\lieg = \bigoplus_{j \in \Z}\lieg_j,$$
	\noindent with $\liet \subseteq \lieg_0$ and denoting by $G_0 \subseteq G$ the connected subgroup corresponding to the subalgebra $\lieg_0$, the natural representations $(G_0,\lieg_k)$ for $k \neq 0$ are called Vinberg $\C^\times$-pairs (since these representations were studied by Vinberg \cite{vinberg_graded}). Then, we have the following result (see \cite[Proposition 4.7]{biquard_arakelov-milnor_2021} for the non-parabolic setting):
	\begin{theorem*}[Theorem \ref{fixedpointsparab}]
		A strongly parabolic $G$-Higgs bundle $(E,\varphi,Q) \in \mathcal M_{sp}(G,\alpha)$ is fixed by the $\C^\times$-action if and only if there is a Vinberg $\C^\times$-pair $(G_0,\lieg_k)$ such that
		\begin{itemize}
			\item There is a reduction of structure group $\sigma \in H^0(E(G/G_0))$.
			\item The Higgs field satisfies $\varphi \in H^0(E_\sigma(\lieg_k) \otimes K_C(D)) \subseteq H^0(E(\lieg) \otimes K_C(D))$.
			\item The parabolic structures satisfy
			$$Q_i \in \bigsqcup_{w \in W_{G_0} \backslash W}\left(\left(E_\sigma\right)|_{c_i} \cdot w \cdot B\right)/B \subseteq E|_{c_i}/B.$$
		\end{itemize}
	\end{theorem*}
	
	In the theorem, $W:=N_G(T)/T \supseteq W_{G_0} := N_{G_0}(T)/T$ denote the respective Weyl groups.
	
	Our classification of very stable Higgs bundles occurs in the setting where $\varphi$ is generically regular. This corresponds to the Vinberg pair $(T,\lieg_1)$ where $\lieg_1$ is the sum of the simple root spaces corresponding to $B$. This can be exploited to attach, for each point $c \in C$, an element $w_c(E,\varphi,Q) \in \widetilde{W}$ of the extended affine Weyl group of $\lieg$ to the fixed point $(E,\varphi,Q)$. These elements can be arranged in the \textbf{twisted multiplicity divisor} $w(E,\varphi,Q)$, a $\widetilde{W}$-valued divisor on $C$. The extended affine Weyl group $\widetilde{W}$ is endowed (after choosing the Borel subgroup $B$) with a Bruhat order coming from its finite-index normal subgroup $\Waff$, the affine Weyl group, which is a Coxeter group. Therefore, we can consider the notion of the divisor $w(E,\varphi,Q)$ being \textbf{reduced} whenever every coefficient is minimal under that order (for $c \notin D$, minimality has to be considered after projecting to $W\backslash \widetilde{W}/W$). With this, we can state our main result.
	
	\begin{theorem*}[Theorem \ref{main}]
		Let $(E,\varphi,Q)$ be a smooth fixed point of Borel type. Then, $(E,\varphi,Q)$ is very stable if and only if $w(E,\varphi,Q)$ is reduced.
	\end{theorem*}
	
	The main technique used in the proof is, as in the original work of Hausel and Hitchin \cite{hausel_very_2022} (and in \cite{gonzalez_very_2025}), that of \textbf{Hecke transformations}, i.e. modifications of the fixed point $(E,\varphi,Q)$ locally around a point $c \in C$, in such a way that the result is isomorphic over $C \setminus \{c\}$. Studying the space of Hecke transformations at a given fixed point $(E,\varphi,Q)$ and $c \in C$ leads naturally to the consideration of affine Springer fibres of nilpotent elements inside affine Flag varieties $\Fl_G$, and the problem can be attacked by relating the $\C^\times$-action in $\mathcal M_{sp}(G,\alpha)$ with an induced $\C^\times$-action on these affine Springer fibres. It turns out that, specifically for fixed points with generically regular $\varphi$, these local modifications are sufficient to understand properties of their upward flows.
	
	Whenever the twisted multiplicity divisor $w(E,\varphi,Q)$ is reduced and supported at $D$, the corresponding very stable upward flow turns out to be a section of $h_G$. Recalling that mirror symmetry is realised by a Fourier--Mukai transform along the fibres of $h_G$, this renders such upward flows as the simplest kind of BAA-branes for which the duality can be tested: since they intersect each fibre at a single point, the dual should be realised by a line bundle on $\mathcal M_{sp}(G^\vee,\alpha)$. We propose a construction of such a line bundle in terms of a universal $G^\vee$-bundle
	$$\mathbb E \to \mathcal M_{sp}(G^\vee,\alpha) \times C.$$
	
	As a universal parabolic bundle, it comes with reductions of structure group $$\mathbb Q_i \in H^0(\mathcal M_{sp}(G^\vee,\alpha) \times \{c_i\}, \mathbb E|_{\mathcal M_{sp}(G^\vee,\alpha) \times \{c_i\}}/B^\vee)$$
	\noindent at each $c_i \in D$. Denote by $\mathbb E_i$ the corresponding $B^\vee$-bundle over $\mathcal M_{sp}(G^\vee,\alpha)$. Since $X^*(B^\vee) = X^*(T^\vee) = X_*(T)$, we may evaluate it at cocharacters of $T$. Then, we propose the following mirror line bundle.
	
	\begin{conjecture*}[Conjecture \ref{mirrorconjecture}]
		Let $(E,\varphi,Q)$ be a smooth fixed point of Borel type with twisted multiplicity divisor $w := w(E,\varphi,Q)$ supported at $D$. The line bundle on $\mathcal M_{sp}(G^\vee,\alpha)$ mirror to the structure sheaf $\mathcal O_{W^+_{(E,\varphi,Q)}}$ on $\mathcal M_{sp}(G,\alpha)$ is given by
		$$\mathcal L_w := \bigotimes_{c_i \in D}\mu_{c_i}(E,\varphi)^{-1}(\mathbb E_i)$$
		\noindent up to tensoring by a fixed line bundle given by the chosen normalisation for $\mathbb E$.
	\end{conjecture*}
	
	There are several implicit assumptions for this conjecture to make sense, such as the existence of a globally defined universal bundle in the moduli space or the lifting of each $\mu_{c_i}(E,\varphi)$ to $X_*(T)$. Another important difficulty is the precise statement of SYZ mirror symmetry and the duality via Fourier--Mukai transforms in the general setting of strongly parabolic $G$-Higgs for any choice of semisimple $G$: as discussed above, this program is still in progress even for classical groups.
	
	However, for $G=\PGL_n(\C)$, these issues can be avoided as the universal bundle descends to the moduli space (cf. \cite{biswasBrauerGroupModuli2011}), every coweight lifts to the group (since $\PGL_n(\C)$ is of adjoint type) and, as previously mentioned, mirror symmetry for type $A$ has been significantly developed in several of its incarnations. Thus, in this setting we can verify that the proposed duals are correct.
	
	\begin{proposition*}[Proposition \ref{conjecturecheck}]
		For $G=\PGL_n(\C)$ and generic weights $\alpha$, the dual of the structure sheaf $\mathcal O_{W^+_{(E,\varphi,Q)}}$ for a fixed point $(E,\varphi,Q)$ of Borel type with twisted multiplicity divisor $w$ supported at $D$ equals the line bundle $\mathcal L_w$ proposed in Conjecture \ref{mirrorconjecture} over the locus $\mathcal A^s \subseteq \mathcal A(\SL_n(\C))$ of smooth spectral curves.
	\end{proposition*}
	
	The paper is structured as follows. In Section \ref{sechiggs} we start by recalling and fixing notation for the relevant Lie-theoretic data of $G$ that will be needed, with special emphasis on the different variants of Weyl groups and their Coxeter theory properties that play a central role in our study. Then we define parabolic Higgs bundles, their stability conditions and we recall relevant geometric aspects of their moduli spaces. We proceed to study the $\C^\times$-action and classify its fixed points in Theorem \ref{fixedpointsparab}. We conclude by introducing relevant properties of the Bia\l ynicki--Birula decomposition for this action. Then, in Section \ref{sechecke} we describe the spaces of Hecke transformations of parabolic $G$-Higgs bundles as affine Springer fibres in affine flag varieties, their Bruhat decomposition and we study explicit curves inside these spaces which realise the closure relations of the Bruhat cells. Section \ref{sectheorem} is devoted to the proof of the main Theorem \ref{main}. For this, we compute the induced $\C^\times$-action on Hecke transformation spaces of a smooth fixed point, and we show how this $\C^\times$-action preserves the explicit curves mentioned before. These also preserve stability and thus give paths inside $\mathcal M_{sp}(G,\alpha)$, relating the different upward flows. We also study how Hecke transformations fixed under this induced $\C^\times$-action, which naturally result in different $\C^\times$-fixed parabolic $G$-Higgs bundles with generically regular $\varphi$, transform the twisted multiplicity coweight. Then we show that, in the case where $w_c(E,\varphi,Q)$ is reduced, the corresponding upward flow is very stable. We finish by combining everything in the proof of the main theorem. Finally, Section \ref{secmirror} studies the mirror symmetry aspects. For the very stable upward flows constituting sections of the Hitchin map, we compute their mirror line bundles in order to check our conjectural proposal. This is done in the case of $G=\PGL_n(\C)$, for which we first briefly recall the necessary aspects of the Beauville--Narasimhan--Ramanan correspondence in the parabolic setting, then we introduce the necessary universal bundles and their properties, and we conclude with the verification of the expected duality.

	\textbf{Acknowledgements.} I would like to thank Oscar García-Prada and Tamás Hausel for their guidance and encouragement. I would also like to thank Olivier Biquard, Guillermo Gallego and Olga Trapeznikova for helpful discussions.   

	\section{Moduli space of parabolic and strongly parabolic \texorpdfstring{$G$}{G}-Higgs bundles}\label{sechiggs}
	
	\subsection{Lie theory preliminaries and notation.}\label{subseclietheory} In this section we define parabolic and strongly parabolic Higgs bundles and recall properties about their moduli spaces. First, we need to review some Lie theory about their structure group that will be used throughout the paper.
	
	Let $G$ be a connected, semisimple complex algebraic group with Lie algebra $\lieg$. Fix a maximal torus $T \subseteq G$ and a Borel subgroup $B \subseteq G$ containing $T$. This choice determines a Cartan subalgebra $\liet := \mathrm{Lie}(T)$ and a Borel subalgebra $\lie b := \mathrm{Lie}(B)$ satisfying $\liet \subseteq \lie b \subseteq \lieg$. Denote by $\Delta := \Delta(\lieg,\liet) \subseteq \liet^*$ the corresponding root system and $\Pi = \{\beta_1,\dots,\beta_r\} \subseteq \Delta$ the choice of simple roots corresponding to $\lie b$, where $r = \rank(\lieg)$. This determines a choice of positive roots $\Delta^+ \subseteq \Delta$. Consider also its dual basis $\{\omega_1^\vee,\dots,\omega_r^\vee\} \subseteq \liet$, which is the basis of \textbf{fundamental coweights}. 
	
	The \textbf{root lattice}, which is the lattice in $\liet^*$ consisting of linear combinations with integer coefficients of the simple roots, is denoted by $Q_\Delta \subseteq \liet^*$. On the other hand, the \textbf{coweight lattice}, consisting of linear combinations with integer coefficients of the fundamental coweights, is denoted by $P_{\Delta^\vee} \subseteq \liet$. If $(\cdot,\cdot)$ denotes the Killing form (either of $\lieg$ or $\lieg^*$), for any element $\beta \in \liet^*$ we define
	$$\beta^\vee := 2\frac{(\beta,\cdot)}{(\beta,\beta)} \in \liet,$$
	\noindent and similarly from $\liet$ to $\liet^*$. We define the \textbf{coroot lattice} $Q_{\Delta^\vee} \subseteq \liet$ and the \textbf{weight lattice} $P_\Delta \subseteq \liet^*$ by applying this operation to the previous lattices.  
	
	The coweights whose pairing with any positive root is non-negative are called \textbf{dominant}, defining the subset $P_{\Delta^\vee}^+ \subseteq P_{\Delta^\vee}$. Similarly, if the pairing with any positive root is non-positive, they are called \textbf{antidominant} and denoted by $P_{\Delta^\vee}^- \subseteq P_{\Delta^\vee}$. Dually one defines notions of dominant and antidominant weights $P^+_\Delta, P^-_\Delta \subseteq P_\Delta$.
	
	We also define the set of \textbf{characters} and \textbf{cocharacters} of $T$, which are the algebraic group homomorphisms
	$$T \to \mathbb G_m,$$
	\noindent and
	$$\mathbb G_m \to T,$$
	\noindent respectively. The groups formed by these are denoted by $X^*(T)$ and $X_*(T)$ respectively. Naturally, they induce linear morphisms $\liet \to \C$ and $\C \to \liet$, therefore they can be viewed in $\liet^*$ and $\liet$, respectively. We have inclusions
	$$Q_\Delta \subseteq X^*(T) \subseteq P_\Delta$$
	\noindent and
	$$Q_{\Delta^\vee}\subseteq X_*(T) \subseteq P_{\Delta^\vee},$$
	\noindent therefore we can talk about dominant and antidominant characters and cocharacters, defining $X^*_+(T), X^*_-(T) \subseteq X^*(T)$, $X_*^+(T), X_*^-(T) \subseteq X_*(T)$.
	
	Now let $C$ be a smooth projective complex curve of genus $g$ with canonical line bundle $K_C$. Choose a reduced effective divisor $D = c_1+\dots + c_s$ of $s$ punctures in $C$. We assume that $\deg K_C(D) = 2g-2+s > 0$ to ensure the existence of an interesting moduli space. 
	
	Let $\liet_\R \subseteq \liet$ denote the $\R$-span of $P_{\Delta^\vee}$ (in other words, it is the Cartan subalgebra for the split real form of $\lieg$). At each puncture $c_i$, we choose \textbf{parabolic weights} $\alpha_i \in \liet_\R$, and denote by $\alpha := (\alpha_1,\dots,\alpha_s)$ the vector of parabolic weights.
	
	For simplicity in the definitions, we \textbf{assume} that the parabolic weights lie in the interior of the \textbf{standard Weyl alcove} of $\liet$. This means that, for each $\alpha_i$, we assume:
	\begin{itemize}
		\item $\beta_j(\alpha_i) > 0$, for every $\beta_j \in \Pi$.
		\item $\delta(\alpha_i) < 1$, where $\delta \in \Delta$ denotes the highest root.
	\end{itemize}
	
	\begin{remark}\label{parahoric}
		It is possible to consider more general weights, in the closure of the standard Weyl alcove (which, in the context of non-abelian Hodge theory, covers all the possible conjugacy classes of the compact part of the monodromy, see \cite[Appendix A]{biquard_parabolic_2019}) or even in any alcove of $\liet_\R$. This introduces more subtleties in the definition of parabolic $G$-Higgs bundles \cite[Sections 2--4]{biquard_parabolic_2019}.
		
		The alternative, which is more convenient for our purposes, would be to work with \textbf{parahoric} torsors \cite{balajiModuliParahorictorsors2015} and \textbf{logahoric} Higgs torsors \cite{kydonakisLogahoricHiggsTorsors2024} for the corresponding weight. For our assumption of weights in the interior of the standard alcove, this notion is equivalent to the parabolic one presented in this paper (see \cite[Section 3]{biquard_parabolic_2019}).
	\end{remark}
	
	Let $W = N_G(T)/T$ denote the \textbf{Weyl group} of $\lieg$, it is a finite reflection group acting on $T$ and on $\liet$ as well as $\liet^*$. It is generated by the set $S_W := \{s_1 := s_{\beta_1},\dots,s_r := s_{\beta_r}\}$ of reflections about the simple root hyperplanes $\{\beta_j = 0\} \subseteq \liet$. 
	
	Similarly, we have the \textbf{affine Weyl group}
	$$\Waff := Q_{\Delta^\vee} \rtimes W$$ which is another (non-finite) reflection group acting on $\liet$: the coroot lattice $Q_{\Delta^\vee}$ acts by translations, and we adopt the convention that $\mu \in Q_{\Delta^\vee}$ maps $X \in \liet$ to $X-\mu$; this convention shall become clearer in Section \ref{sechecke} when this group is regarded inside of $G((z))$. It is generated by the set $S_{\Waff} = S_W \cup \{s_0 := s_{\delta,1}\}$ where $\delta$ denotes the highest root and $s_{\delta,1}$ is the reflection by the affine root hyperplane $\{\delta = 1\} \subseteq \liet$. 
	
	We have two slightly larger groups that will be crucial in our study. The first one is the \textbf{Iwahori--Weyl group}
	$$\widetilde{W}_G := X_*(T) \rtimes W$$
	\noindent and the second one is the \textbf{extended affine Weyl group}
	$$\widetilde{W} := P_{\Delta^\vee} \rtimes W.$$
	
	The inclusions of lattices described above imply inclusions
	$$\Waff \subseteq \widetilde{W}_G \subseteq \widetilde{W}.$$

	The last two groups are not reflection groups in general: the quotient 
	$$\widetilde{W}/\Waff \simeq P_{\Delta^\vee}/Q_{\Delta^\vee}$$ 
	\noindent is finite and elements in the nontrivial classes cannot be expressed as a product of root hyperplane reflections. A set of representatives of the quotient is that of \textbf{dominant minuscule coweights}, i.e. elements $\mu \in P_{\Delta^\vee}$ such that $\beta(\mu) \in \{0,1\}$ for every $\beta \in \Delta^+$. The quotient $\widetilde{W}_G/\Waff \subseteq \widetilde{W}/\Waff$ can be identified with the set of dominant minuscule coweights contained in $X_*(T)$.   

	All these Weyl groups fit (up to finite index in the last two cases) within the theory of Coxeter groups. In what follows, we recall the necessary aspects of this theory for later use, a good reference is \cite{bjorner2005combinatorics}. We denote by $(\widehat{W},\widehat{S})$ a Coxeter system, for our purposes this is either $(W,S_W)$ or $(\Waff, S_{\Waff})$. Similarly, we denote by $(\widehat{W}',\widehat{S})$ a quasi-Coxeter system, meaning that $\widehat{W}'$ is a group containing some Coxeter group $\widehat{W}$ (which is part of a Coxeter system ($\widehat{W}, \widehat{S}$)) as a normal subgroup such that the quotient $\widehat{W}'/\widehat{W}$ is finite abelian. For our purposes this is either $(\widetilde{W}_G,S_{\Waff})$ or $(\widetilde{W},S_{\Waff})$, in both cases containing $\Waff$. We have a natural splitting $\widehat{W}' = \widehat{W} \rtimes \Omega$, where $\Omega \simeq \widehat{W}'/\widehat{W}$ is defined to be the normaliser of $\widehat{S}$.
	
	\begin{definition}
		The \textbf{length} of an element $w' = w \cdot \omega \in \widehat{W}' = \widehat{W} \rtimes \Omega$ is the minimal number $k \in \Z_{\ge 0}$ such that there is an expression
		$$w = s_{i_1} \cdot \dots \cdot s_{i_k}$$
		\noindent for $s_{i_j} \in \widehat{S}$. We denote it by $l(w')$.
	\end{definition}
	
	\begin{definition}\label{bruhatorder}
		The \textbf{Bruhat order} in $\widehat{W}'$ is defined by the following rule: for $w\cdot \omega, w'\cdot \omega' \in \widehat{W}'$, we have $w \cdot \omega \le w' \cdot \omega'$ if and only if both of these hold:
		\begin{itemize}
			\item The length zero parts agree: $\omega = \omega'$.
			\item Given an expression $w' = s_{i_1}\cdot \dots \cdot s_{i_k}$ where $s_{i_j} \in \widehat{S}$, there exists a subexpression, obtained by omitting some entries, whose product equals $w$.
		\end{itemize}
	\end{definition}
	
	Some basic properties that we use frequently are that length is order-preserving and that the Bruhat order is generated by relations of the form $w \le wt$ for $w \in \widehat{W}$, $t=usu^{-1}$ with $u \in \widehat{W}$, $s \in \widehat{S}$ and $l(w) \le l(wt)$. Moreover, inverting elements preserves length.
	
	\begin{definition}\label{parabolicweyl}
		Given a subset $J \subseteq \widehat{S}$, we define the \textbf{parabolic subgroup} $\widehat{W}_J \subseteq \widehat{W}$ as the subgroup generated by $J$. The pair $(\widehat{W}_J,J)$ is itself a Coxeter system. The quotient $\widehat{W}'/\widehat{W}_J$ has a natural set of representatives given by those with minimal length:
		$$(\widehat{W}')^J := \{w \in \widehat{W}' : l(ws) > l(w) \text{ for all } s \in J\}.$$
		This set is called the \textbf{left parabolic quotient} for $J$. Analogously, the quotient $\widehat{W}_J \backslash \widehat{W}'$ has minimal length representatives:
		$${}^{J}\widehat{W}' := \{w \in \widehat{W}' : l(sw) > l(w) \text{ for all } s \in J\} = \left((\widehat{W}')^J\right)^{-1}.$$
		This set is called \textbf{right parabolic quotient} for $J$. Similarly, for subsets $J,J' \subseteq \widehat{S}$ the double quotient $\widehat{W}_{J'}\backslash \widehat{W}'/\widehat{W}_J$ has a natural set of representatives:
		$${}_{J'}\left(\widehat{W}'\right)^J =  \left\{w \in \widehat{W}' : l(w) = \max_{w_1 \in \widehat{W}_{J'}}\min_{w_2 \in \widehat{W}_J}l(w_1ww_2)\right\}.$$
		This set is called the \textbf{double parabolic quotient} for $J,J'$.
	\end{definition}
	
	The Bruhat order and length restrict to parabolic quotients and double parabolic quotients in such a way that the quotient morphisms are order preserving. Note that for double quotients it is also standard to take ${}^{J'}(\widehat{W}')^J$ the set of minimum length representatives (i.e. not mixing minimum and maximum), but we will see later in Section \ref{sechecke} that geometrically the previous choice is cleaner.
	
	\begin{remark}
		In the case of the finite Weyl group $(W,S_W)$, parabolic quotients have a unique longest element which is of special importance to us. It suffices to recall that the length of an element in $W$ can be regarded as the number of positive roots which are sent to negative roots. Since elements $w \in W^J$ are characterised by $l(ws) > l(w)$ for $s \in J$ or, equivalently, by sending the simple roots $\Pi_J \subseteq \Pi$ whose associated reflections are in $J$ to positive roots, the unique longest element is the one sending every positive root outside the span of $\Pi_J$ to a negative root. Inverting, the same is true for ${}^JW$ by considering $w^{-1}$.
	\end{remark}
	
	\begin{remark}\label{Jmu}
		Given an element $\mu \in \liet$, we often associate parabolic subgroups and quotients by taking $J_\mu \subseteq \widehat{S}$ to be the stabiliser of $\mu$.
	\end{remark}
	
	We conclude this section by identifying precisely the length zero elements (i.e. the set $\Omega$) for the extended affine Weyl group $\widetilde{W}$. It turns out that this does not simply agree with the set of minuscule representatives we gave earlier. In our convention (where $P_{\Delta^\vee}$ translates with a negative sign), we have the length formula:
	$$l(\mu \cdot u) = \sum_{\beta, u^{-1}(\beta) \in \Delta^+}|\mu(\beta)| + \sum_{\beta, -u^{-1}(\beta) \in \Delta^+}|\mu(\beta)+1|.$$
	This was given in \cite[Proposition 1.23]{iwahori_bruhat_1965}, noting that they provide the formula in the case that $P_{\Delta^\vee}$ translates by adding.
	
	\begin{proposition}\label{lengthzero}
		Let $\mu \cdot u \in \widetilde{W} = P_{\Delta^\vee} \rtimes W$. Then, $l(\mu \cdot u) = 0$ if and only if
		\begin{itemize}
			\item $\mu$ is antidominant minuscule, i.e. $\mu(\beta) \in \{0,-1\}$ for all $\beta \in \Delta^+$.
			\item $u$ is the longest element in the parabolic quotient (of $W$)
			$${}^{J_\mu}W = \{w \in W : w^{-1}(\beta) \in \Delta^+ \text{ for all } \beta \in \Pi \text{ such that } \mu(\beta) = 0\}.$$
		\end{itemize}
	\end{proposition}
	
	\begin{proof}
		It suffices to use the length formula given immediately above. For the length to be zero, it is clear that $\mu(\beta) \in \{0,-1\}$ for positive roots $\beta \in \Delta^+$. Moreover, we need that $\mu(\beta) = 0$ if and only if $u^{-1}(\beta) \in \Delta^+$. This characterises $u$ as the longest element of ${}^{J_\mu}W$. Both conditions are clearly sufficient.
	\end{proof}

	\subsection{Parabolic Higgs bundles and their moduli spaces.}We start by defining parabolic bundles over the punctured curve $(C,D)$.
	
	\begin{definition}
		A \textbf{parabolic $G$-bundle} over $(C,D)$ with weights $\alpha = (\alpha_1,\dots,\alpha_s)$ is a pair $(E,Q)$ where
		\begin{itemize}
			\item $E$ is a principal $G$-bundle over $C$.
			\item $Q = (Q_1,\dots,Q_s)$ is a choice of a \textbf{parabolic structure} for each $c_i \in D$, that is, a $B$-orbit $Q_i \in E|_{c_i}/B$.
		\end{itemize}
	\end{definition}
	
	The choice of a parabolic structure at $c_i$ is equivalent to choosing a subgroup $Q'_i \subseteq E(G)|_{c_i}$ isomorphic to $B$ in the fibre of the adjoint bundle $E(G)$. Indeed, said fibre can be described as the space of $G$-equivariant maps
	$$f : E|_{c_i} \to G$$
	\noindent so given a $B$-orbit $eB \in E|_{c_i}/B$ we get a subgroup 
	$$Q_i' = \{f \in E(G)|_{c_i} : f(e) \in B\}.$$
	Conversely, given such a subgroup, there is an element $e \in E|_{c_i}$ such that $f(e) \in B$ for all $f \in Q'_i$, determining the orbit $eB \in E|_{c_i}$ (which is well-defined as $B$ is self-normalising). 
	
	Therefore, we get an induced Borel subalgebra $\lie q_i \subseteq E(\lieg)|_{c_i}$. Let $\lie n_i \subseteq \lie q_i$ denote its nilpotent radical. Define the subsheaves $SPE(\lieg) \subseteq PE(\lieg) \subseteq E(\lieg)$ via the exact sequences
	$$0 \to PE(\lieg) \to E(\lieg) \to \bigoplus_{i=1}^sE(\lieg)|_{c_i}/\lie q_i \to 0$$
	$$0 \to SPE(\lieg) \to E(\lieg) \to \bigoplus_{i=1}^sE(\lieg)|_{c_i}/\lie n_i \to 0.$$
	These are called the sheaves of \textbf{parabolic} and \textbf{strongly parabolic} sections of the adjoint bundle, respectively.
	
	\begin{definition}
		A \textbf{parabolic $G$-Higgs bundle} with weights $\alpha$ over $(C,D)$ is a triple $(E,\varphi,Q)$ where:
		\begin{itemize}
			\item $(E,Q)$ is a parabolic $G$-bundle with weights $\alpha$ over $(C,D)$. 
			\item $\varphi \in H^0(PE(\lieg) \otimes K_C(D))$ is a section called \textbf{Higgs field}.
		\end{itemize}
		If $\varphi \in H^0(SPE(\lieg) \otimes K_C(D))$, we say that the parabolic $G$-Higgs bundle is \textbf{strongly parabolic}.
	\end{definition}
	
	We want to study the moduli space $\mathcal M(G,\alpha)$ parametrising isomorphism classes of polystable parabolic $G$-Higgs bundles. For this, we first recall the suitable notions of stability \cite[Section 4.2]{biquard_parabolic_2019} \cite[Definition 2.2]{telemanParabolicBundlesProducts2003a}.
	
	\begin{definition}
		Let $\hat{G} \le G$ be a closed subgroup and $E$ a principal $G$-bundle over a variety $X$. A \textbf{reduction of structure group} of $E$ to $\hat{G}$ is a section $\sigma \in H^0(X,E(G/\hat{G}))$.
	\end{definition}
	
	Such a section $\sigma : X \to E(G/\hat{G})$ allows to obtain a principal $\hat{G}$-bundle $E_\sigma := \sigma^*E$ on $X$ by pulling back the $\hat{G}$-bundle $E \to E(G/\hat{G})$. It is a subvariety $E_\sigma \subseteq E$ and $E_\sigma(G) \simeq E$.
	
	The stability notions for parabolic $G$-Higgs bundles are defined via reductions to parabolic subgroups of $G$. Let $P \subseteq G$ be a standard parabolic subgroup, that is, any closed subgroup containing $B$. Let $L \subseteq P$ be its standard Levi subgroup, by which we mean the Levi subgroup that contains $T$. Choose a dominant character $\chi \in X^*_+(P)$, i.e. a character $P \to \mathbb G_m$ such that $\chi|_T$ is dominant. 
	
	Given a principal $G$-bundle $E$ and a reduction $\sigma \in H^0(E(G/P))$ we have a notion of \textbf{degree} of $\sigma$ with respect to the dominant character $\chi$, given by
	$$\deg(E,\sigma,\chi) := \deg \chi(E_\sigma),$$
	\noindent where we use that $\chi(E_\sigma)$ is a $\C^\times$-bundle, i.e. a line bundle. 
	
	Moreover, taking into account the parabolic structure $Q$ on $E$, we may define a \textbf{parabolic degree}. For this we add an extra term depending on the parabolic structures and weights. Notice that in $E|_{c_i}$ we have both a $P$-orbit $\sigma(c_i)$ and a $B$-orbit $Q_i$, therefore there is a well-defined \textbf{relative position} by writing $\sigma(c_i) = gP$ and $Q_i = g'B$ in any trivialisation, and taking the element 
	$$(g,g')\in G \backslash (G \times G)/(P \times B) \simeq P \backslash G/B \simeq W_P \backslash W,$$
	\noindent where $W_P \subseteq W$ denotes the \textit{parabolic subgroup} corresponding to the Weyl group of $L$ (note that this is indeed a parabolic subgroup in the sense of Definition \ref{parabolicweyl}). This does not depend on the chosen trivialisation. We denote this relative position by $\deg(\sigma,Q_i)$.
	
	\begin{definition}
		Let $(E,Q)$ be a parabolic $G$-bundle over $(C,D)$ with weights $\alpha$, denote by $\sigma \in H^0(E(G/P))$ a reduction of structure group and let $\chi \in X^*_+(P)$ a dominant character. We define its \textbf{parabolic degree} as
		$$\pdeg(E,Q,\sigma,\chi) := \deg(E,\sigma,\chi) + \sum_{i=1}^s\chi(\deg(\sigma,Q_i)\alpha_i).$$
	\end{definition}
	
	Notice that the rightmost part of the expression is well defined: since $\chi$ is a character of $P$, the expression does not depend on the chosen representative in $W_P\backslash W$ acting on $\alpha$, as elements in $W_P$ fix $\chi$.
	
	\begin{remark}
		The previous quantity can also be defined using analytic techniques. On one hand, the degree can be defined via Chern--Weil theory by reducing $E_\sigma$ further to the maximal compact subgroup of $L$, yielding $E_{\sigma'}$, choosing a connection $A$ on $E_{\sigma'}$ with curvature $F_A$ and setting
		$$\deg(E,\sigma,\chi) := \frac{i}{2\pi}\int_C\chi(F_A).$$
		
		On the other hand, the local part of the parabolic degree can be defined using the Tits distance on the \textit{visual boundary} of a symmetric space associated to the local data, see \cite[Appendix B]{biquard_parabolic_2019}.
	\end{remark}
	
	Now we define the stability notions for parabolic $G$-Higgs bundles.
	
	\begin{definition}
		A parabolic $G$-Higgs bundle $(E,\varphi,Q)$ with weights $\alpha$ over $(C,D)$ is
		\begin{itemize}
			\item \textbf{semistable}, if for any standard parabolic subgroup $P \subseteq G$, dominant character $\chi \in X^*_+(P)$ and reduction $\sigma \in H^0(C,E(G/P))$ such that $\varphi|_{C \setminus D} \in H^0(C \setminus D, E_\sigma(\liep) \otimes K_C|_{C\setminus D})$, we have
			$$\pdeg(E,Q,\sigma,\chi) \le 0.$$
			\item \textbf{stable}, if for any proper standard parabolic subgroup $P \subsetneq G$, dominant character $\chi \in X^*_+(P)$ and reduction $\sigma \in H^0(C,E(G/P))$ such that $\varphi|_{C \setminus D} \in H^0(C \setminus D, E_\sigma(\liep) \otimes K_C|_{C\setminus D})$, we have
			$$\pdeg(E,Q,\sigma,\chi) < 0.$$
			\item \textbf{polystable}, if it is semistable and, for any standard parabolic subgroup $P \subseteq G$, dominant character $\chi \in X^*_+(P)$ and reduction $\sigma \in H^0(C,E(G/P))$ such that $\varphi|_{C \setminus D} \in H^0(C \setminus D, E_\sigma(\liep) \otimes K_C|_{C\setminus D})$ and
			$$\pdeg(E,Q,\sigma,\chi) = 0,$$
			\noindent there exists a further reduction $\sigma' \in H^0(C,E_\sigma(P/L))$ such that $\varphi \in H^0(C \setminus D, E_{\sigma'}(\liel) \otimes K_C|_{C\setminus D})$ and a parabolic structure on the $L$-bundle $E_{\sigma'}$ inducing that of $E$. 
		\end{itemize}
	\end{definition}
	
	Using these notions, a \textbf{moduli space} of isomorphism classes of polystable parabolic $G$-Higgs bundles with weight $\alpha$ can be constructed. We denote it by $\mathcal M(G,\alpha)$. It is a complex, quasiprojective variety with a Poisson structure, foliated by symplectic leaves corresponding to fixing the extra data of the orbit for the associated graded to the residues at the parabolic punctures \cite[Proposition 7.1]{biquard_parabolic_2019}, \cite{bottacinSymplecticGeometryModuli1995,markmanSpectralCurvesIntegrable1994,logaresModuliParabolicHiggs2010}. In particular, the locus of strongly parabolic Higgs bundles defines a symplectic leaf
	$$\mathcal M_{sp}(G,\alpha) \subseteq \mathcal M(G,\alpha)$$
	\noindent which has a hyperkähler structure, and thus it is equipped with a holomorphic symplectic form $\omega \in H^0(\Omega^2_{\mathcal M_{sp}(G,\alpha)})$. This is the space where we focus or work for the remainder of the paper. It has dimension
	$$\dim \mathcal M_{sp}(G,\alpha) = 2\dim(G)(g-1) + 2s\dim(G/B).$$
	
	\subsection{Deformation theory, the symplectic form and the Hitchin system.}\label{sympform} We now give a more precise definition of the holomorphic symplectic form $\omega$ mentioned above. For this, we need to understand the deformation theory of strongly parabolic $G$-Higgs bundles.
	
	\begin{definition}\label{deformationcomplex}
		Let $(E,\varphi,Q) \in \mathcal M_{sp}(G,\alpha)$ be a strongly parabolic $G$-Higgs bundle. Its \textbf{deformation complex} is the following complex of sheaves over $C$:
		$$C^\bullet(E,\varphi,Q) : PE(\lieg) \xrightarrow{[\varphi,-]} SPE(\lieg) \otimes K_C(D).$$
	\end{definition} 
	
	Notice that indeed $PE(\lieg)$ maps to $SPE(\lieg) \otimes K_C(D)$ because $\varphi \in H^0(SPE(\lieg) \otimes K_C(D))$ and $[\lie q_i, \lie n_i] \subseteq \lie n_i$. This complex governs the deformations of $(E,\varphi,Q)$ in the moduli space.
	
	\begin{proposition}[{\cite[Remark 2.1, Proposition 4.2]{yokogawaInfinitesimalDeformationParabolic1995}, \cite[Section 6]{markmanSpectralCurvesIntegrable1994}}]\label{deformationtheory}
		The space of deformations of a strongly parabolic $G$-Higgs bundle $(E,\varphi,Q)$ is naturally isomorphic to $\mathbb H^1(C^\bullet(E,\varphi,Q))$. In particular, if $(E,\varphi,Q) \in \mathcal M_{sp}^s(G,\alpha)$ is a smooth point, we have a description of its tangent space
		$$T_{(E,\varphi,Q)}\mathcal M(G,\alpha) \simeq \mathbb H^1(C^\bullet(E,\varphi,Q)).$$
	\end{proposition}

	\begin{remark}\label{smoothness}
		For a point $(E,\varphi,Q) \in \mathcal M_{sp}(G,\alpha)$ to be smooth, it has to be stable. Due to the definition of stability and of the parabolic degree, a generic choice of the weights $\alpha$ implies that every point in the moduli space is stable. However, the resulting space may still have singularities of orbifold type at those points whose automorphism group $\Aut(E,\varphi,Q)$ is strictly larger than the centre $Z(G)$ (see \cite[Section 7.1]{biquard_parabolic_2019}). Although this does not pose significant problems, we will typically work over the smooth locus $\mathcal M_{sp}^s(G,\alpha)$. 
	\end{remark}
	
	The previous deformation theory can be used to describe the symplectic form $\omega$ using the standard Serre duality argument. Indeed, the Killing form induces an isomorphism
	$$SPE(\lieg)^* \simeq PE(\lieg) \otimes \mathcal O(D)$$
	\noindent so that
	$$\mathbb H^1(C^\bullet(E,\varphi,Q))^* \simeq \mathbb H^1(C^\bullet(E,\varphi,Q)^* \otimes K_C) \simeq \mathbb H^1(C^\bullet(E,\varphi,Q))$$
	\noindent where the first isomorphism is Serre duality and the second is the Killing form pairing explained above. This identification defines $\omega$.
	
	The moduli space $\mathcal M_{sp}(G,\alpha)$ is furthermore endowed with an algebraically completely integrable system, the \textbf{Hitchin system} (originally \cite{hitchin_stable_1987} in the non parabolic setting, see \cite{bottacinSymplecticGeometryModuli1995,markmanSpectralCurvesIntegrable1994}), defined as follows: consider the algebra of invariant polynomial functions $\C[\lieg]^G = \C[p_1,\dots,p_r]$, which is a polynomial ring on $r$ generators. Set $d_i := \deg(p_i)$. Then, given a section $\varphi \in H^0(E(\lieg)\otimes K_C(D))$, the evaluation $p_i(\varphi) \in H^0(C,K_C^{d_i}(d_i\cdot D))$ is well defined and, moreover, if $\varphi$ is strongly parabolic, it lies in $H^0(C,K_C^{d_i}((d_i-1) \cdot D))$ since $\varphi|_{c_i}$ is nilpotent. 
	
	\begin{definition}
		The \textbf{Hitchin map} is defined by
	\begin{align*}
		h_G : \mathcal M_{sp}(G,\alpha) &\to \mathcal A(G) := \bigoplus_{i=1}^rH^0(C,K_C^{d_i}((d_i-1) \cdot D))\\
		(E,\varphi,Q) &\mapsto (p_1(\varphi),\dots,p_r(\varphi)).
	\end{align*}
	\end{definition}
	
	It is a proper map of algebraic varieties and the base has half the dimension of the moduli space.
	
	\subsection{The \texorpdfstring{$\C^\times$}{C*}-action and fixed points.} Given a parabolic $G$-Higgs bundle $(E,\varphi,Q)$, it can be scaled by a nonzero complex number $\lambda \in \C^\times$:
	$$(E,\varphi,Q) \mapsto (E,\lambda \varphi, Q).$$
	This scaling results in a well-defined parabolic $G$-Higgs bundle (since the sheaf of parabolic sections is invariant by scaling) and the stability notions are preserved. Moreover, some of the symplectic leaves are preserved, including $\mathcal M_{sp}(G,\alpha)$. This results in a well defined $\C^\times$-action
	\begin{align*}
	\mathcal M_{sp}(G,\alpha) &\to \mathcal M_{sp}(G,\alpha)\\
(E,\varphi,Q) &\mapsto (E,\lambda \varphi, Q).
	\end{align*}
	
	The Hitchin map $h_G$ is $\C^\times$-equivariant for the $\C^\times$-action on $\mathcal A(G)$ scaling each of the direct summands with weight $d_i$. Also, it is not hard to see from this description that the $\C^\times$-action maps the holomorphic symplectic form $\omega$ defined above to $\lambda \omega$.
	
	We are interested in the study of this action: first, we describe the classification of its fixed points. As in the case of usual (non-parabolic) $G$-Higgs bundles, this description is related to $\Z$-gradings of the Lie algebra $\lieg$ \cite[Section 4.2]{biquard_arakelov-milnor_2021}, \cite{simpson_constructing_1988}.
	
	\begin{definition}
		A \textbf{$\Z$-grading} of $\lieg$ is a direct sum decomposition of the underlying vector space
		$$\lieg = \bigoplus_{j \in \Z}\lieg_j$$
		\noindent such that $[\lieg_j,\lieg_k] \subseteq \lieg_{j+k}$.
	\end{definition}
	
	Since $\lieg_0$ is preserved by the bracket, there is a corresponding connected subgroup $G_0 \subseteq G$. Moreover, since $[\lieg_0,\lieg_k] \subseteq \lieg_k$, the restriction of the adjoint action of $G$ to $G_0$ preserves each $\lieg_k$ resulting in an action of $G_0$ on $\lieg_k$. These representations of $G_0$ (whenever $k\neq 0$) are called \textbf{Vinberg $\C^\times$-pairs} \cite{garcia-prada_cyclic_2024, garcia-pradaVinbergPairsHiggs2024}.
	
	As $\lieg$ is semisimple, there exists a \textbf{grading element} $\zeta \in \lieg_0$ such that $\lieg_k$ is the eigenspace with eigenvalue $k$ for $\ad(\zeta)$. We have $G_0 = C_G(\zeta)$, that is, the subgroup $G_0$ can be described as the centraliser of this grading element. Without loss of generality (by conjugating if necessary) we may assume that $\zeta \in \liet$. In particular $T \subseteq G_0$ is still a maximal torus.
	
	In order to classify the fixed points by the $\C^\times$-action, we first set the parabolic structure aside. 
	
	\begin{definition}
		A \textbf{$K_C(D)$-twisted} $G$-Higgs bundle over $C$ is a pair $(E,\varphi)$ where $E$ is a principal $G$-bundle over $C$ and $\varphi \in H^0(E(\lieg) \otimes K_C(D))$.
	\end{definition}
	
	Note that, since for our choice of parabolic weights we have $PE(\lieg) \subseteq E(\lieg)$, forgetting the parabolic structure on a parabolic $G$-Higgs bundle results in a $K_C(D)$-twisted $G$-Higgs bundle. Being fixed by the $\C^\times$-action as a $K_C(D)$-twisted $G$-Higgs bundle (i.e. having a family of automorphisms $(E,\varphi) \simeq (E,\lambda \varphi)$) is weaker than being fixed as a parabolic $G$-Higgs bundle, since the automorphisms need not be compatible with the parabolic structure. Thus, our strategy to identify the fixed points in $\mathcal M(G,\alpha)$ is to first do so for $K_C(D)$-twisted Higgs bundles and then study which parabolic structures can we equip in a compatible way with the corresponding $\C^\times$-family of automorphisms.
	
	\begin{proposition}\label{cstarfixedmeromorphic}
		Let $(E,\varphi)$ be a $K_C(D)$-twisted $G$-Higgs bundle. There is a family of automorphisms
		$$f_\lambda : (E,\varphi) \xrightarrow{\sim} (E,\lambda \varphi)$$
		\noindent if and only if there is a Vinberg $\C^\times$-pair $(G_0,\lieg_k)$ and a reduction $\sigma \in H^0(C,E(G/G_0))$ such that $\varphi \in H^0(E_\sigma(\lieg_k) \otimes K_C(D))$.
	\end{proposition}
	
	\begin{proof}
		The proof is analogous to usual $K_C$-twisted Higgs bundles, see \cite[Proposition 4.15]{biquard_arakelov-milnor_2021}. For convenience later, we recall why having such a reduction implies the existence of the $f_\lambda$. Introduce the cocharacter $\xi_\bullet : \mathbb G_m \to \Ad(G)$ defined by exponentiating the grading element $\zeta$. Since $\zeta$ is centralised by $G_0$, this cocharacter lies in the centre of $G_0$ and defines an automorphism of $E_\sigma$, which acts on the Higgs field as $\Ad_{\xi_\lambda}(\varphi) = \lambda^k\varphi$ because $\varphi$ takes values in $\lieg_k$. Since $k \neq 0$, we get that $(E,\varphi)$ is isomorphic to $(E,\lambda^k \varphi)$ and hence to $(E,\lambda \varphi)$ for any $\lambda \in \C^\times$. 
	\end{proof}
 		
 	Now we investigate which parabolic structures $Q$ on a $K_C(D)$-twisted $G$-Higgs bundle fixed by the $\C^\times$-action are compatible with the family of automorphisms $f_\lambda$, resulting in the identification of the $\C^\times$-fixed parabolic $G$-Higgs bundles.
 	
 	\begin{proposition}\label{localparab}
 	Let $(E,\varphi)$ be a $K(D)$-twisted Higgs bundle fixed by the $\C^\times$-action. Let $f_\lambda : (E,\varphi) \xrightarrow{\sim} (E,\lambda \varphi)$ denote the family of automorphisms given by Proposition \ref{cstarfixedmeromorphic}, i.e. obtained by using $\xi_\lambda \in \Ad(G)$ the cocharacter given by exponentiating the grading element $\zeta$. Let $c \in D$ be one of the parabolic punctures. 
 	
 	There is an identification between the space of parabolic structures $E|_c/B$ and $G/B$ such that the following are equivalent:
 	\begin{itemize}
 		\item Every $f_\lambda$ is a parabolic automorphism with respect to the parabolic structure $gB \in G/B$.
 		\item The parabolic structure $gB$ is fixed by the $\C^\times$-action
 		$$gB \mapsto \xi_\lambda gB$$
 		\noindent on $G/B$.
 	\end{itemize}
    \end{proposition}
 
    \begin{proof}
	 	Note that $\xi_\lambda \in \Ad(G)$ acts on the flag variety because $Z(G) \subseteq B$. As explained in the proof of Proposition \ref{cstarfixedmeromorphic}, we can find $e_\lambda \in E|_c$ a trivialisation via which $f'_\lambda(e_\lambda) = \tilde{\xi}_\lambda$ where $\tilde{\xi}_\lambda \in G$ lifts $\xi_\lambda$ and $f'_\lambda : E|_c \to G$ is the $G$-equivariant section corresponding to the bundle automorphism $f_\lambda: E \to E$. This element is the desired trivialisation $E|_c \simeq G$ that identifies $E|_c/B$ with $G/B$: let $Q_g$ denote the parabolic structure corresponding to $gB \in G/B$.
	 	
	 	Recall that $f_\lambda$ and $f_\lambda'$ are related as follows: $f_\lambda(e) = ef'_\lambda(e)$. Thus 
	 	$$f_\lambda(e_g) = e_gf'_\lambda(e_g) = e_gg^{-1}\tilde{\xi}_\lambda g = e_\lambda \tilde{\xi}_\lambda g.$$
	 	
	 	Therefore, we have the following diagram:
	 	\[\begin{tikzcd}[ampersand replacement=\&]
	 		{E|_c/B} \&\& {E|_c/B} \\
	 		\\
	 		{G/B} \&\& {G/B}
	 		\arrow["{\bar{f}_\lambda|_c}", from=1-1, to=1-3]
	 		\arrow["{\bar{e}_\lambda}"', from=1-1, to=3-1]
	 		\arrow["{\bar{e}_\lambda}"', from=1-3, to=3-3]
	 		\arrow["{\tilde{\xi}_\lambda \cdot -}", from=3-1, to=3-3]
	 	\end{tikzcd}.\]
	 	
	 	\noindent and a parabolic structure $e_g$ will be fixed by $f_\lambda$ if and only if it is fixed by the proposed action.
    \end{proof}
 
	\begin{remark}
		We can arrive at the same conclusion applying the definition of parabolic automorphism directly: $f_\lambda$ is a parabolic automorphism with respect to the parabolic structure in the $B$-orbit of $e_g \in E|_c$ if and only if $f'_\lambda(e_g) \in B$. But we have
		$$f'_\lambda(e_g) = f_\lambda(e_\lambda g) = g^{-1}\tilde{\xi}_\lambda g,$$
		\noindent therefore the condition is equivalent to
		$$g^{-1}\tilde{\xi}_\lambda g \in B.$$
		In other words, it is equivalent to $\tilde{\xi}_\lambda g B = g B$.
	\end{remark}
 
	\begin{remark}\label{cstarparabremark}
		From the proof above we deduce that the constructed identification has the following property: let $Q_g = (Q_{g_1},\dots, Q_{g_s})$ be the parabolic structure corresponding to elements $g_iB \in G/B$ under the previous identification at each of the punctures $c_i \in D$. Then
		$$(E,\varphi, Q_g) \simeq (E,\lambda \varphi, Q_{\xi_\lambda g}),$$
	 \noindent via $f_\lambda$, where $\xi_\lambda g$ denotes the diagonal action.
	\end{remark}	
 
 	Having understood the action on the space of parabolic structures, we can identify the fixed points. For this we also introduce the Weyl group $W_{G_0} \subseteq W_G := W$ given by $N_{G_0}(T)/T$.
 	
 	\begin{lemma}\label{localparabfixed}
 		The fixed point locus under the $\C^\times$-action in Proposition \ref{localparab} is a disjoint union of subvarieties 
 		$$G_0wB/B \subseteq G/B$$
 		\noindent indexed by $w \in W_{G_0} \backslash W_G$.
 	\end{lemma}
 	\begin{proof}
 		This is a standard result, see e.g. \cite[Section 3.1.7]{chriss2010representation} for the case where $\xi_\bullet$ is a regular cocharacter. In general, we can use the Bruhat decomposition
 		$$G/B = \bigsqcup_{w \in W}BwB/B$$
 		\noindent deduced in \textit{loc.cit}. This decomposition shows that an element in $G/B$ can be written as $uwB/B$ where $u \in U$ is an element in the unipotent radical $U \subseteq B$. This representation is unique if we select the representative $u \in U \cap wU^{opp}w^{-1}$. From this uniqueness it is easy to see that $\xi_\lambda uwB/B = uw/B$ if and only if $u$ is centralised by $\xi_\lambda$, that is, if and only if $u \in G_0$, as desired. Of course, the components $G_0wB/B$ and $G_0w'B/B$ are the same if and only if $w \equiv w' \mod W_{G_0}$.
 	\end{proof}
 	
 	We are now in position of describing the fixed point locus in the moduli space of parabolic $G$-Higgs bundles. Given a Vinberg $\C^\times$-pair $(G_0,\lieg_k)$ and a reduction of structure group $\sigma \in H^0(E(G/G_0))$ of a principal $G$-bundle $E$ with parabolic structure $Q$, let 
 	$$PE_\sigma(\lieg_k) = E_\sigma(\lieg_k) \cap PE(\lieg)$$ 
 	\noindent denote the sheaf of parabolic sections passing through $E_\sigma(\lieg_k) \subseteq E(\lieg)$. Define $SPE_\sigma(\lieg_k)$ similarly.
 	
 	\begin{theorem}\label{fixedpointsparab}
 		A parabolic $G$-Higgs bundle $(E,\varphi,Q) \in \mathcal M(G,\alpha)$ is fixed by the $\C^\times$-action if and only if there is a Vinberg $\C^\times$-pair $(G_0,\lieg_k)$ such that
 		\begin{itemize}
 			\item There is a reduction of structure group $\sigma \in H^0(E(G/G_0))$.
 			\item The Higgs field satisfies $\varphi \in H^0(PE_\sigma(\lieg_k) \otimes K_C(D)) \subseteq H^0(PE(\lieg) \otimes K_C(D))$.
 			\item The parabolic structures satisfy
 			 $$Q_i \in \bigsqcup_{w \in W_{G_0} \backslash W}\left(\left(E_\sigma\right)|_{c_i} \cdot w \cdot B\right)/B \subseteq E|_{c_i}/B.$$
 		\end{itemize}
 	\end{theorem}
 	\begin{proof}
 		Follows directly from Proposition \ref{cstarfixedmeromorphic}, Proposition \ref{localparab} and Lemma \ref{localparabfixed}.
 	\end{proof}
 	
 	\begin{remark}\label{nilpotentissp}
 		Note that for $k \neq 0$ we have $SPE_\sigma(\lieg_k) = PE_\sigma(\lieg_k)$. Indeed, elements $X \in \lieg_k$ are nilpotent because $\ad^m(X)$ increases the degree by $mk$, and degrees are bounded since $\lieg$ is finite dimensional. Moreover, since our parabolic weights are in the interior of the standard Weyl alcove, nilpotent elements in $\mathfrak q_i$ are automatically in $\mathfrak n_i$ (this would not be true for other choices of weights). In particular, all fixed points of the $\C^\times$-action lie in $\mathcal M_{sp}(G,\alpha)$. 
 	\end{remark}
 	
 	\begin{example}
 		Let us illustrate the above theorem with a simple example. Let $G=\PGL_2(\C)$ and choose any weight in the interior of the standard Weyl alcove. The Lie algebra is $\lieg = \liesl_2(\C)$. Letting $(e,h,f)$ be a triple of generators satisfying the usual relations, we can grade it by $\zeta = \frac{h}{2}$. Then, we have
 		$$G_0 = (\C^\times \times \C^\times)/\C^\times = T$$
 		\noindent (where the quotient is by the diagonal action) embedded diagonally in $\PGL_2(\C)$ and acting on $\lieg_1 = \C$ by $(\lambda, \beta) \cdot z = \lambda \beta^{-1}z$. A $G_0$ bundle is equivalent via the standard representation to a vector bundle
 		$$E = L_0 \oplus L_1$$
 		\noindent splitting as a direct sum of two line bundles and up to tensoring by any line bundle. Therefore we can normalise it to
 		$$E = \mathcal O_C \oplus L$$
 		\noindent where $L$ is any line bundle (although stability requires $\deg L \le 0$).
 		
 		A Higgs field taking values in $\lieg_1$ is the same as a section $\varphi \in H^0(LK_C(D))$. Finally, a parabolic structure at a point $c \in C$ is equivalent to the choice of a line in the fibre $E|_c$. According to the theorem, the resulting Higgs bundle is fixed if and only if the line belongs to either $(E_\sigma)|_{c_i}B/B$ or $(E_\sigma)|_{c_i}wB/B$ where $w \in W$ is the nontrivial element and $B$ has the Lie algebra spanned by $(e,h)$. The first choice corresponds to the line $L|_c \subseteq E|_c$ and is always compatible with the Higgs field. The second choice corresponds to the line $\mathcal O_C|_c \subseteq E|_c$ and requires the section $\varphi$ to vanish at $c$.
 		
 		In any case, it is clear that the resulting Higgs bundle is fixed by the $\C^\times$-action, since $E$ has an automorphism given by scaling $L$ by a factor of $\lambda$: this maps $\varphi$ to $\lambda \varphi$ and preserves any of those two lines in $E|_c$. Note that no other lines are preserved.      
 	\end{example}
 	
 	\begin{example}\label{typeafixed}
 		Let us consider more generally what is the description of fixed points for $G=\PGL_n(\C)$. The $\Z$-gradings of semisimple Lie algebras are classified, see \cite[Chapter 3.3]{onivin}. For $\lieg = \liesl_n(\C)$, they can be described as follows. Set $V$ a $n$-dimensional vector space over $\C$, so that $\lieg = \liesl(V)$. Gradings are determined by ordered partitions $(l_k)_k$ of $n$. Once such a partition is selected, fix a splitting
 		$$V = \bigoplus_{k=0}^{m-1}V_k$$
 		\noindent in vector subspaces of dimensions $l_k$. This determines the grading
 		$$\lieg_p := \bigoplus_{k \in \Z}\Hom(V_k,V_{k+p})_0$$
 		\noindent where the subscript $0$ denotes morphisms that are traceless inside of $\lieg$ and is only meaningful when $p=0$. We have
 		$$G_0 = \left(\prod_{k=0}^{m-1}\GL(V_k)\right)/\C^\times$$
 		\noindent where $\C^\times$ acts diagonally by scalars.
 		
 		A strongly parabolic $\PGL_n(\C)$-Higgs bundle (for any choice of weights in the interior of the Weyl alcove, which in this case corresponds to choosing real numbers $\alpha_i^1 > \alpha_i^2 > \dots \alpha_i^n$ such that $\sum_j\alpha_i^j = 0$ and $\alpha_i^1 - \alpha_i^n < 1$ at each $c_i \in D$) is given by
 		\begin{itemize}
 			\item A vector bundle $E$ of rank $n$ (up to line bundle tensoring).
 			\item A choice of a full flag
 			$$Q_i^0 = 0 \subsetneq Q_i^1 \subsetneq \dots \subsetneq Q_i^n = E|_{c_i}$$
 			\noindent at each $c_i \in D$.
 			\item A bundle morphism $\varphi : E \to E\otimes K_C(D)$ such that
 			$$\varphi(Q_i^j)\subseteq Q_i^{j-1}\otimes K_C(D).$$
 		\end{itemize}
 		
 		For such a strongly parabolic $G$-Higgs bundle to be a fixed point associated to the previous grading (say, for the pair $(G_0,\lieg_1)$, the others being analogous), the vector bundle has to split
 		$$E = \bigoplus_{k=0}^{m-1}E_k$$
 		\noindent in vector subbundles, the Higgs field has to map
 		$$\varphi(E_k) \subseteq E_{k+1}\otimes K_C(D)$$
 		\noindent and the parabolic structures need to be \textit{graded} with respect to the splitting, in other words:
 		$$Q_i^j = \bigoplus_{k=0}^{m-1}\left(E_k|_{c_i} \cap Q_i^j\right)$$
 		\noindent has to hold for every $i \in \{1,\dots,s\}$ and $j \in \{1,\dots,n\}$.
 	\end{example}
 	
 	\begin{remark}
 		Theorem \ref{fixedpointsparab} could be rephrased in a similar fashion to \cite[Proposition 4.15]{biquard_arakelov-milnor_2021} by defining \textit{moduli spaces of parabolic $(G_0,\lieg_k)$-Higgs pairs}, $\mathcal M(G_0,\lieg_k,\alpha)$, similar to how moduli spaces for $(\hat{G},\hat{V})$-Higgs pairs are defined for any representation $\hat{V}$ of a complex reductive group $\hat{G}$ (e.g. \cite{schmitt2008geometric}) in the non-parabolic setting.
 		
 		These parametrise triples $(E,\varphi,Q)$ where $E$ is a principal $G_0$-bundle, $Q_i \in E|_{c_i}/(G_0\cap B)$ is a choice of parabolic structure at $c_i \in D$ and $\varphi \in H^0(PE(\lieg_k) \otimes K_C(D))$ (here we define $PE(\lieg)$ by including $E|_{c_i}/(G_0 \cap B)$ in $E(G)|_{c_i}/B$ and then define $PE(\lieg_k)$ by intersecting with $E(\lieg_k)$).
 		
 		Then, for any choice of $w_i \in W_{G_0} \backslash W_G$ for each $c_i \in D$, resulting in $w=(w_1,\dots,w_s)$, we have the map
 		$$\mathcal M(G_0,\lieg_k,w(\alpha)) \to \mathcal M(G,\alpha)$$
 		\noindent by 
 		$$E \mapsto E(G)$$
 		$$\varphi \mapsto \varphi$$
 		$$Q_i \mapsto Q_i \cdot w.$$
 		
 		By the previous Theorem \ref{fixedpointsparab}, the union of the images of all these maps covers the fixed point locus of the $\C^\times$-action. The agreement of both stability conditions is checked as in the discussion preceding \cite[Proposition 4.15]{biquard_arakelov-milnor_2021} by noticing the equivalence of the gauge-theoretic equations involved in both moduli spaces. 
 	\end{remark}
 	
 	\begin{remark}
		For more general parabolic weights, giving parabolic subgroups $P_\alpha \subseteq G$ other than the Borel, it is possible to carry the same analysis to conclude that if the parabolic structure belongs (under the identification in Proposition \ref{cstarfixedmeromorphic}) to $G_0wP_\alpha/P_\alpha$ for $w \in W_{G_0}\backslash W_G / W_{P_\alpha}$, the resulting point is fixed. However, as observed in \cite[Section 4]{biquard_parabolic_2019}, these more general weights require considering the notion of \textit{meromorphic equivalence} of the principal $G$-bundles instead of isomorphism, so a priori not every fixed point is obtained in this way. As stated in Remark \ref{parahoric}, in this situation we would consider more natural to work with parahoric torsors and logahoric Higgs torsors for the corresponding weights $\alpha$. In that case, the fixed point locus is described by reductions (as a logahoric Higgs torsor!) to \textit{Vinberg pairs} ($\mathcal G_{\alpha,0},\underline{\lieg_k})$ determined by gradings of the scheme of Lie algebras corresponding to the Bruhat--Tits group scheme $\mathcal G_\alpha$ for weights $\alpha$.
 	\end{remark}
 	
 	To conclude the discussion about fixed points, we highlight a distinguished class which we term of \textbf{Borel type} (see \cite[Definition 3.17]{gonzalez_very_2025}). These correspond to components of the fixed point locus where the Higgs field is generically a regular (nilpotent) element of the Lie algebra. For this to happen, the corresponding Vinberg pair $(G_0,\lieg_k)$ needs to satisfy that $\lieg_k$ contains regular nilpotent elements, and this restricts the grading to the following.
 	
 	\begin{definition}
 		The \textbf{Borel grading} of $\lieg$ is given by the grading element $\zeta = \sum_{j=1}^r\omega_j^\vee$. In other words, it satisfies $\lieg_0 = \liet$ and
 		$$\lieg_k := \bigoplus_{\beta \in \Delta : \heightr(\beta) = k}\lieg_\beta,$$
 		\noindent where $\heightr(\beta)$ denotes the height of $\beta$ with respect to $\Pi$.
 	\end{definition}
 	
 	The terminology refers to the associated parabolic $\bigoplus_{k \ge 0}\lieg_k$ being precisely the Borel subalgebra  $\lie b$. Notice that for the Borel grading we have $G_0 = T$ and its action on $\lieg_1 = \bigoplus_{j=1}^r\lieg_{\beta_j}$ has an open orbit $\Omega \subseteq \lieg_1$ consisting of regular nilpotent elements. This orbit is given by the elements with nonzero projection to each $\lieg_{\beta_j}$.
 		
 	\begin{definition}
 		A fixed point $(E,\varphi,Q) \in \mathcal M_{sp}(G,\alpha)^{\C^\times}$ is of \textbf{Borel type} if it verifies the condition of Theorem \ref{fixedpointsparab} for the Vinberg $\C^\times$-pair $(T,\lieg_1)$ corresponding to the Borel grading, and $\varphi$ takes values in $\Omega \subseteq \lieg_1$ generically. 
 	\end{definition}
 	
 	Given a fixed point of Borel type $(E,\varphi,Q)$ we can associate certain Lie-theoretic data encoding the vanishing multiplicities of $\varphi$ and the parabolic structures $Q$. Denote by $E_T$ the reduction of $E$ to $T$. First, since $T$ acts on each simple root space $\lieg_{\beta_j}$, we may extract sections of line bundles by projecting:
 	$$\varphi_j := \pi_j(\varphi) \in H^0(E_T(\lieg_{\beta_j}) \otimes K_C(D)),$$
 	\noindent where we used
 	$$\pi_j : H^0(E(\lieg) \otimes K_C(D)) \to H^0(E_T(\lieg_{\beta_j}) \otimes K_C(D)).$$
 	These are the \textbf{components} of $\varphi$ along the direction of the simple root $\beta_j$. They are non-vanishing sections of the line bundle $E_T(\lieg_{\beta_j}) \otimes K_C(D)$: denote by $\ord_{c}(\varphi_j)$ the order of vanishing.
 	
 	\begin{definition}
 		Given a fixed point $(E,\varphi,Q) \in \mathcal M_{sp}(G,\alpha)^{\C^\times}$ of Borel type, we define its \textbf{multiplicity coweight at $c \in C$} by the expression
 		$$\mu_c(E,\varphi) := \sum_{j=1}^r\ord_c(\varphi_j) \cdot \omega_j^\vee \in P_{\Delta^\vee}^+.$$
 	\end{definition}
 	
 	It is an element of the space of dominant coweights and it does not depend on $Q$.
 	
 	Additionally, Theorem \ref{fixedpointsparab} indicates that the parabolic structure comes from
 	$$E_T|_{c_i} \cdot w_i \cdot B/B \subseteq E|_{c_i}/B$$
 	\noindent for some $w_i \in W$. Notice that the space on the left is a single point, so we can naturally associate the Weyl group element $w_i \in W$ to $(E,\varphi,Q)$ at each $c_i \in D$ given by the parabolic structure. We introduce the following notation:
 	$$u_c(E,\varphi,Q) := \begin{cases}
 		w_i & c = c_i,\\
 		1 & c \notin D
 	\end{cases} \in W.$$
 	
 	There is a natural way to package all the information above using the extended affine Weyl group.
 	
 	\begin{definition}\label{twistmultdiv}
 		Given a smooth fixed point $(E,\varphi,Q) \in \mathcal M_{sp}(G,\alpha)^{\C^\times}$ of Borel type, we define its \textbf{twisted multiplicity divisor} by
 		$$w(E,\varphi,Q) := \sum_{c \in C}w_c(E,\varphi,Q) \cdot c$$
 		\noindent where
 		$$w_c(E,\varphi,Q) := \left(-\mu_c(E,\varphi)\right) \cdot u_c(E,\varphi,Q) \in \widetilde{W} = P_{\Delta^\vee} \rtimes W.$$
 	\end{definition}
 	
 	Notice that it is a divisor on $C$ whose coefficients lie in $\widetilde{W}$. Only finitely many points in $c$ have nontrivial coefficient (they are contained in the union of $D$ with the locus where $\varphi$ is non-regular). If $c \in C \setminus D$, the Weyl group element is trivial and it is more natural to think of the coefficient as an element in $P_{\Delta^\vee}^- \simeq W\backslash \widetilde{W}/W$.
 	
 	\begin{remark}\label{signconvetion}
 		Note that, in the definition of the twisted multiplicity divisor, the coweight part given by $\mu_c(E,\varphi)$ is inverted. This is related to the convention explained in Section \ref{subseclietheory} where $P_{\Delta^\vee}$ acts on $\liet$ by subtracting. This is merely a choice since there is an automorphism of $\widetilde{W}$ given by inverting on the $P_{\Delta^\vee}$ factor as the latter is abelian. The reason for this choice will become clearer in Section \ref{sechecke} when $\widetilde{W}$ is regarded inside the loop group $G((z))$.
 	\end{remark}
 	
 	\begin{remark}\label{sphericalquotient}
 		Since $\mu_c(E,\varphi)$ is dominant, the part of $w_c(E,\varphi,Q)$ in $P_{\Delta^\vee}$ is always antidominant. Similarly, recalling the definition of $u_{c_i}(E,\varphi,Q)$, the compatibility condition between $\varphi|_{c_i}$ and $Q_i$ can be cleanly restated in terms of combinatorics (recall the notation of Section \ref{subseclietheory}) as
 		$$u_{c_i}(E,\varphi,Q) \in {}^{J_{\mu_{c_i}(E,\varphi)}}W.$$
 		These two conditions determine a subset of $\widetilde{W}$ where $w_{c_i}(E,\varphi,Q)$ must lie. This subset is precisely the right parabolic quotient ${}^{S_W}\widetilde{W}$, of minimal length representatives of $W \backslash \widetilde{W}$, sometimes referred to as \textbf{spherical quotient}.
 	\end{remark}
 	
	\begin{remark}\label{toptypeborelparabolic}
		Recall that there is a \textbf{topological type} $d_E \in \pi_1(T) \simeq X_*(T)$ associated to a $T$-bundle $E$, uniquely determined by the condition
		$$\deg \chi(E) = \chi(d_E)$$
		\noindent for every character $\chi \in X^*(T)$. For a fixed point $(E,\varphi,Q)$ of Borel type, it can be recovered from $\deg (\mu(E,\varphi)) := \sum_{c \in C}\mu_c(E,\varphi) \in P_{\Delta^\vee}$. Indeed, we have $$E(\lieg_{\beta_j}) \otimes K_C(D) \simeq \mathcal O_C(\sum_{c \in C}\beta_j(\mu_c(E,\varphi))),$$ \noindent which automatically forces
		$$d_E = \deg(\mu(E,\varphi)) + \mu_{s}$$
		\noindent where $\mu_{s} = -(2g-2+s)\zeta = -(2g-2+s)\sum_{j=1}^r\omega_j^\vee$.
	\end{remark}
 	
 	\begin{remark}\label{recoverfromtwistedmult}
 		More strongly, knowing $w(E,\varphi,Q)$ it is possible to recover $(E/Z(G),\varphi,Q)$, i.e recover $(E,\varphi,Q)$ up to a twist by a $Z(G)$-bundle (therefore uniquely if $G$ is of adjoint type). From the data of $\mu_c(E,\varphi)$ one recovers the line bundles $E_T(\lieg_{\beta_j})$ and their sections $\pi_j(\varphi)$. Since the simple roots span the root lattice, knowing $E_T(\lieg_{\beta_j})$ recovers $E_T/Z(G)$ and therefore $E_T$ is recovered up to a twist of a $Z(G)$-bundle, and so is $E = E_T(G)$. The Higgs field is completely recovered from the $\pi_j(\varphi)$, and the parabolic structures are completely recovered from the $u_c(E,\varphi,Q)$. On the other hand, it is not always possible to construct $(E,\varphi,Q)$ with arbitrarily prescribed $w(E,\varphi,Q)$. One immediate constraint is that the \textit{degree} of the divisor (i.e. the sum of the coefficients) needs to lie in the Iwahori--Weyl group $\widetilde{W}_G$ (this is always satisfied if $G=G_{\ad}$ is of adjoint type).
 		Even then, the resulting bundle may not be stable. Another constraint is provided by Remark \ref{sphericalquotient} involving the compatibility with the Higgs field.
 	\end{remark}
 	
 	\begin{example}\label{typeafixedborel}
 		Consider the fixed points for $G=\PGL_n(\C)$ as described in Examples \ref{typeafixed} and \ref{typeafixedborel}. The Borel grading corresponds to the decomposition where each piece has degree one. Therefore, fixed points of Borel type are given by sums of $n$ line bundles, which we may write in the form
 		$$E = \mathcal O_C \oplus K_C^{-1}(-D+D_1) \oplus K_C^{-2}(-2D+D_1+D_2) \oplus \dots \oplus K_C^{1-n}((1-n)D+D_1+\dots+D_{n-1})$$
 		\noindent where $D_i \subseteq C$ are effective divisors. The Higgs field takes the form
 		$$\varphi = \begin{pmatrix}
		0 & 0 & \dots & 0 & 0 \\
		\varphi_1 & 0 & \dots & 0 & 0 \\
		0 & \varphi_2 & \dots & 0 & 0 \\
		\vdots & \vdots & \ddots & \vdots & \vdots \\
		0 & 0 & \dots & \varphi_{n-1} & 0\end{pmatrix}$$
		\noindent where $\varphi_j = \pi_j(\varphi) \in H^0(\mathcal O_C(D_j))$ has divisor $D_j$.
		The parabolic structure at $c_i \in D$ consists of a full flag obtained by taking at each step one extra term in the decomposition of $E|_{c_i}$ in line bundles given above, subject to the compatibility condition with $\varphi$ (i.e. not every flag is allowed). The only flag which is always permitted, no matter the Higgs field, corresponds to $u_{c_i}(E,\varphi,Q) = 1$ and is given by
		\begin{equation*}
			\begin{split}
				0 &\subsetneq K_C^{1-n}((1-n)D+D_1+\dots+D_{n-1})|_{c_i} \\
				&\subsetneq K_C^{2-n}((2-n)D+D_1+\dots+D_{n-2})|_{c_i} \oplus  K_C^{1-n}((1-n)D+D_1+\dots+D_{n-1})|_{c_i} \\
				&\subsetneq \dots \subsetneq E|_{c_i}.
			\end{split}
		\end{equation*}
		
		The twisted multiplicity divisor in this case is obtained by noticing that $u_{c_i}(E,\varphi,Q) \in W \simeq S_n$ corresponds to the permutation of the $n$ factors that brings the previous flag to the parabolic structure over $c_i$. Then, we have
		$$w(E,\varphi,Q) = \left(\sum_{j=1}^r-D_j\omega_j^\vee\right) \cdot \sum_{c \in C}u_c(E,\varphi,Q)$$
		\noindent where the product is taken coefficient-wise.
		
		For a more concrete example, take $n=2$, consider $D=c_1$ a single point, let $c \neq c_1$ be another point and take
		$$E = \mathcal O_C \oplus K_C^{-1}(3c+c_1)$$
		$$\varphi = \begin{pmatrix}
			0 & 0 \\
			s_c^3s_{c_1}^2 & 0
		\end{pmatrix}$$
		\noindent where $s_c \in H^0(\mathcal O_C(c))$ is the canonical section with divisor $c$, and similarly for $s_{c_1}$. Take the parabolic structure $Q_1$ to be given by the line $\mathcal O_C|_{c_1}$, which is compatible since $\varphi|_{c_1} = 0$. 
		
		The twisted multiplicity divisor for this example is
		$$w(E,\varphi,Q) = (-2\omega_1^\vee \cdot s_{\beta_1})c_1 + (-3\omega_1^\vee)c$$
		\noindent where $s_{\beta_1} \in W$ is the non-trivial element.
 	\end{example}
 		
 			We are mainly interested in fixed points $(E,\varphi,Q) \in \mathcal M_{sp}(G,\alpha)^{\C^\times}$ of Borel type that define smooth points in $\mathcal M_{sp}(G,\alpha)$. Therefore, by Remark \ref{smoothness}, we need to analyse their stability as well as their automorphism group. It turns out that stability is sufficient due to the following result.
 		
 		\begin{proposition}\label{simplicityparab}
 			Let $(E,\varphi,Q)$ be a stable parabolic $\C^\times$-fixed point of Borel type. Then, $\Aut(E,\varphi,Q) = Z(G)$. In particular, $(E,\varphi,Q) \in \mathcal M_{sp}(G,\alpha)^s$ defines a smooth point of the moduli space.
 		\end{proposition}
 		
 		\begin{proof}
 			Since $(E,\varphi,Q)$ is of Borel type, the Higgs field $\varphi$ takes values in the regular nilpotent orbit generically over $C$.  Now, if $f \in \Aut(E,\varphi,Q) \simeq H^0(C, E(G))$ is an automorphism, which we regard as a $G$-equivariant section $f: E \to G$ (with respect to the left conjugation action of $G$ on $G$) compatible with $\varphi$ and $Q$, the compatibility with the Higgs field is given by $\Ad_f(\varphi) = \varphi$. In other words, generically over $C$, the automorphism $f$ takes values in the centraliser of a regular nilpotent element. Such a centraliser is of the form $Z(G) \cdot U$ where $U \subseteq G$ is a connected abelian unipotent subgroup with $\dim(U)=\rank(\lieg)$ (see e.g. \cite[Proposition 14 and Remark 15]{kostantLieGroupRepresentations1963}). Therefore, its only finite order elements are those of $Z(G)$. By a direct adaptation of \cite[Proposition 3.11]{garcia-prada_hitchin-kobayashi_2012} to the parabolic setting, the stability of $(E,\varphi,Q)$ implies that $\Aut(E,\varphi,Q)$ is finite. Therefore, $f$ is a finite order element and must take values generically in $Z(G)$. This implies that $f$ is constant with values in $Z(G)$ over the whole of $C$, as $Z(G)$ is discrete.
 		\end{proof}
 		
 		Now, we characterise the stability of a fixed point of Borel type in terms of its twisted multiplicity divisor.
 		
 		\begin{proposition}\label{boreltypestabilityparab}
 			Let $(E,\varphi,Q)$ be a strongly parabolic $G$-Higgs bundle satisfying the condition of Theorem \ref{fixedpointsparab} for the Borel grading and such that $\varphi$ takes values in the open orbit $\Omega \subseteq \lieg_1$ generically. Let $d_E = \deg(\mu(E,\varphi)) + \mu_{\textup{s}} \in X_*(T)$ denote the topological type of the $T$-reduction of $E$ from Remark \ref{toptypeborelparabolic}. The following are equivalent:
 			\begin{enumerate}
 				\item $(E,\varphi,Q)$ is stable as a parabolic $G$-Higgs bundle.
 				\item For all $\chi \in X^*_+(T)$ we have $\chi(d_E) + \sum_{i=j}^s\chi(u_{c_j}(E,\varphi,Q)(\alpha_j))< 0.$
 			\end{enumerate}
 			The previous equivalences also hold replacing stable by semistable and strict inequalities by non-strict inequalities.
 		\end{proposition}
 		
 		\begin{proof}
 			Let $P \subseteq G$ be a standard parabolic subgroup and $\chi \in X_+^*(P)$ a dominant character. Let $\sigma \in H^0(C,E(G/P))$ be a reduction of structure group such that $\varphi|_{C\setminus D} \in H^0(C \setminus D, E_\sigma(\liep) \otimes K_C(D))$. Since $\varphi$ is generically a regular nilpotent element of $\lieg$, $\sigma$ is uniquely determined over the dense open subset of $C \setminus D$ where $\varphi$ is regular. Indeed, a regular nilpotent element $e \in \lieg$ only belongs to a unique conjugate of $\liep \subseteq \lieg$. This is well known when $\liep$ is Borel (e.g. \cite[Proposition 3.2.14]{chriss2010representation}), and, if $e \in \Ad_{g}(\liep)$ for some $g \in G$, the nilpotency of $e$ implies that there is some $g' \in G$ with $e \in \Ad_{g'}(\lie b) \subseteq \Ad_{g}(\liep)$ (see e.g. \cite[Lemma in Section 30.3]{humphreys2012linear}), from which the claim follows by reducing to the case of a Borel subalgebra.
 			
 			Then, properness of $G/P$ implies that $\sigma$ is uniquely determined everywhere. We can construct such $\sigma$ by extending the structure group via $T \subseteq P$ and $\lieg_1 \subseteq \lie b \subseteq \liep$ from the reduction to the Vinberg $\C^\times$-pair $(T,\lieg_1)$ that $(E,\varphi,Q)$ possesses. To conclude, it suffices to notice that, due to the previous construction, we have
 			$$\deg(E,\sigma,\chi) = \chi(d_E)$$
 			\noindent and $\deg(\sigma, Q_i) = u_{c_i}(E,\varphi,Q)$ by definition, so that
 			$$\pdeg(E,\sigma,\chi) = \chi(d_E) + \sum_{i=j}^s\chi(u_{c_j}(E,\varphi,Q)(\alpha_j)).$$
 		\end{proof}
 		
	\subsection{The Bia\l ynicki-Birula decomposition and very stable upward flows.}	
	Now we recall some aspects of the Bia\l ynicki-Birula decomposition \cite{bialynicki-birula_theorems_1973} via $\C^\times$-actions on semiprojective varieties that are relevant for us; see \cite[Sections 2 and 3]{hausel_very_2022} for a more detailed account.
	
	Even if the moduli space is only quasiprojective, the $\C^\times$-action on $\mathcal M_{sp}(G,\alpha)$ is \textbf{semiprojective}, that is, the fixed point locus is projective and, for any $(E,\varphi,Q) \in \mathcal M_{sp}(G,\alpha)$, the limit $\lim\limits_{\lambda \to 0}(E,\lambda \varphi,Q)$ always exists.
	
	Indeed, since the Hitchin map $h_G$ is $\C^\times$-equivariant and the only fixed point in $\mathcal A(G)$ is zero, the fixed point locus defines a closed subvariety of $h_G^{-1}(0)$, and the latter is projective because $h_G$ is proper. Therefore, the fixed point locus $\mathcal M_{sp}^{\C^\times}(G,\alpha)$ is projective. Moreover, since limits when $\lambda \to 0$ exist on $\mathcal A(G)$ and $h_G$ is proper, the valuative criterion of properness implies that limits when $\lambda \to 0$ exist in $\mathcal M_{sp}(G,\alpha)$.
	
	Therefore, we get a partition of the moduli space according to the limit point.
	
	\begin{definition}
		Let $(E,\varphi,Q) \in \mathcal M^{\C^\times}_{sp}(G,\alpha)$. The \textbf{upward flow} from $(E,\varphi,Q)$ is the subvariety
		$$W^+_{(E,\varphi,Q)} := \left\{(E',\varphi',Q') \in \mathcal M_{sp}(G,\alpha) : \lim\limits_{\lambda \to 0}(E',\lambda \varphi', Q') = (E,\varphi,Q)\right\} \subseteq \mathcal M_{sp}(G,\alpha).$$
	\end{definition} 
	
	Since $\mathcal M_{sp}(G,\alpha)$ is semiprojective, upward flows define a partition
	$$\mathcal M_{sp}(G,\alpha) = \bigsqcup_{(E,\varphi,Q) \in \mathcal M_{sp}^{\C^\times}(G,\alpha)}W^+_{(E,\varphi,Q)}$$
	\noindent known as the \textbf{Bia\l ynicki-Birula partition}.
	
	Considering limits in the opposite direction, we have the following.
	\begin{definition}
		Let $(E,\varphi,Q) \in \mathcal M^{\C^\times}_{sp}(G,\alpha)$. The \textbf{downward flow} from $(E,\varphi,Q)$ is the subvariety
		$$W^-_{(E,\varphi,Q)} := \left\{(E',\varphi',Q') \in \mathcal M_{sp}(G,\alpha) : \lim\limits_{\lambda \to \infty}(E',\lambda \varphi', Q') = (E,\varphi,Q)\right\} \subseteq \mathcal M_{sp}(G,\alpha).$$
	\end{definition} 
	
	Since the moduli space is not projective, these do not partition the whole moduli space and they determine its \textbf{core}:
	$$\mathcal C:= \bigsqcup_{(E,\varphi,Q) \in \mathcal M_{sp}^{\C^\times}(G,\alpha)}W^-_{(E,\varphi,Q)}.$$
	As before, $\C^\times$-equivariance and properness of $h_G$ allow to deduce that (the underlying reduced schemes for) $\mathcal C$ and $h^{-1}_G(0)$ coincide \cite[Section 3.3]{hausel_very_2022}.
	
	Now we recall some of the main properties of upward and downward flows.
	
	\begin{proposition}[{\cite{bialynicki-birula_theorems_1973}, \cite[Propositions 2.1 and 2.10]{hausel_very_2022}}]
		Let $\mathcal E:=(E,\varphi,Q) \in \mathcal M_{sp}^{s\C^\times}(G,\alpha)$ be a smooth fixed point of the $\C^\times$-action. Consider the weight decomposition for the induced $\C^\times$-action on the tangent space 
		$$T_{\mathcal E}\mathcal M_{sp}(G,\alpha) = \bigoplus_{k \in \Z}T_{\mathcal E}^k.$$
		Let 
		$$T_{\mathcal E}^+ := \bigoplus_{k > 0}T_{\mathcal E}^k \subseteq T_{\mathcal E}\mathcal M_{sp}(G,\alpha)$$
		\noindent denote the subspace of positive weights and $T_{\mathcal E}^- := \bigoplus_{k < 0}T_{\mathcal E}^k$ the subspace of negative weights. Then
		\begin{itemize}
			\item Upward and downward flows $W^+_{\mathcal E}, W^-_{\mathcal E} \subseteq \mathcal M_{sp}(G,\alpha)$ are locally closed subvarieties of $\mathcal M_{sp}(G,\alpha)$.
			\item There is an isomorphism of $\C^\times$-varieties $W_{\mathcal E}^+ \simeq T^+_{\mathcal E}$. In particular, upward flows are isomorphic to affine spaces. 
			\item Similarly, there is an isomorphism of $\C^\times$-varieties $W^-_{\mathcal E} \simeq T^-_{\mathcal E}$.
			\item With respect to the holomorphic symplectic form $\omega$ from Section \ref{sympform},
			the subvarieties $W^+_{\mathcal E}$ are Lagrangian.
		\end{itemize}
	\end{proposition}
	
	Therefore, upward flows from smooth fixed points provide Lagrangian subvarieties with very simple geometry and a $\C^\times$-action. It is then reasonable to ask whether these are closed Lagrangians (since otherwise the process of closing them down by recursively attaching the upward flows of limit points results in a more complicated geometry). 
	
	\begin{definition}[{\cite[Definition 2.12, Proposition 2.14, Lemma 4.6]{hausel_very_2022}}]
		A smooth fixed point $\mathcal E:=(E,\varphi,Q) \in \mathcal M^{s\C^\times}_{sp}(G,\alpha)$ is \textbf{very stable} if any of the following equivalent conditions is satisfied:
		\begin{itemize}
			\item We have $W^+_{\mathcal E} \cap h^{-1}_G(0) = \{\mathcal E\}$.
			\item The subvariety $W^+_{\mathcal E} \subseteq \mathcal M_{sp}(G,\alpha)$ is closed.
			\item The restriction of the Hitchin map $h|_{W^+_{\mathcal E}} : W_{\mathcal E}^+ \to \mathcal A(G)$ is a finite, flat surjective morphism.
		\end{itemize}
	\end{definition}
	
	We can describe upward flows locally (i.e. the deformations of a fixed point that lie in the upward flow) by taking advantage of the previous identification with the space of positive weights of the tangent space. We record here this description for fixed points of Borel type for later use.
	
	\begin{proposition}\label{parabolicupwardflow}
		Let $(E,\varphi,Q)$ be a smooth fixed point of Borel type and $(E',\varphi',Q')$ be a point on its upward flow. Let $\{(U_j,t_j)\}$ be a trivialisation of $E$ where the $U_j \subseteq C$ are affine opens, the transition functions $f_{ij} : U_i \cap U_j \to T \subseteq G$ take values in $T$, the local expressions of the Higgs field $\varphi_i : U_i \to \lieg_1 \subseteq \lieg$ take values in $\lieg_1$ and the local expressions of the parabolic structures take values in $W \subseteq G/B$. Then, $(E',\varphi',Q')$ is isomorphic to the bundle obtained by deforming
		$$f_{ij}' = (1+\varepsilon_{ij})f_{ij}$$
		$$\varphi_i' = \varphi_i + \delta_i$$
		\noindent (and the local expression for the parabolic structures unchanged) for a \v Cech cocycle 
		$$\varepsilon_{ij} \in H^0(U_{ij},PE(\lieg_{<0})),$$ $$\delta_i \in H^0(U_i, SPE(\lie b^{opp}) \otimes K_C(D)),$$ \noindent where $\lieg_{<0} := \bigoplus_{k < 0}\lieg_k$, representing an element in
		$$\mathbb H^1(PE(\lieg) \to SPE(\lieg) \otimes K_C(D)).$$ 
	\end{proposition}
	
	\begin{proof}
		Recall that the tangent space at $(E,\varphi,Q)$ is given in Proposition \ref{deformationtheory} by
		$$\mathbb H^1(PE(\lieg) \to SPE(\lieg) \otimes K_C(D)),$$
		\noindent where the arrow is induced by $\ad(\varphi)$. Since $\varphi \in H^0(SPE(\lieg_1) \otimes K_C(D))$, this can be decomposed as
		$$\bigoplus_{k \in \Z}\mathbb H^1(PE(\lieg_k) \to SPE(\lieg_{k+1}) \otimes K_C(D)).$$
	
		In our fixed trivialisation, an element in the $k$-th piece of this hypercohomology corresponds to a \v{C}ech cocycle $\varepsilon_{ij} \in H^0(U_{ij},PE(\lieg_k))$, $\delta_i \in H^0(U_i, SPE(\lieg_{k+1}))$ satisfying the cocycle condition described in detail in \cite{biswas_infinitesimal_1994}. They correspond to the deformation 
		$$f_{ij}' = (1+z\varepsilon_{ij})f_{ij},$$
		$$\varphi_i' = \varphi_i + z\delta_i.$$
		\noindent (and unchanged local expression for the parabolic structure) where $z \in \C[z]/z^2$ is the deformation parameter.
		
		We may find the induced weight for this piece of this deformation space by applying the $\C^\times$-action, which corresponds to scaling the Higgs field by $\lambda$ and hence the local data 
		$$f_{ij}' = (1+z\varepsilon_{ij})f_{ij},$$
		$$\lambda\varphi_i' = \lambda\varphi_i + \lambda z\delta_i.$$
		However this no longer has the form of a deformation as above. We can recover that form by applying the local automorphism given by $\xi_\lambda^{-1}$, getting:
		$$\xi_\lambda^{-1}f_{ij}'\xi_\lambda = (1+z\lambda^{-k}\varepsilon_{ij})f_{ij}$$
		$$\Ad_{\xi_\lambda^{-1}}\lambda \varphi_i' = \varphi_i + z\lambda^{-k}\delta_i$$
		\noindent and the same local expression for the parabolic structure since $\xi_\lambda^{-1}$ fixes $W \subseteq G/B$.
		
		Therefore, we see that the weight on this piece is $-k$ and therefore the deformations corresponding to the upward flow are those with $k < 0$, so the $\varepsilon_{ij}$ take values in $\bigoplus_{j < 0}\lieg_j = \lieg_{<0}$ and the $t_i$ take values in $\bigoplus_{j \le 0}\lieg_j = \lie b^{opp}$.
	\end{proof}
	
	We conclude this section by describing the most important example of very stable upward flow.
	
	\begin{example}\label{paraboliccanonical}
		We now explain the construction of the fixed point $(E^0,\varphi^0,Q^0)$ of Borel type that satisfies $w(E^0,\varphi^0,Q^0) = 1$, i.e. that whose Higgs field $\varphi^0$ is everywhere regular and $Q^0$ is always the \textit{trivial} parabolic structure (i.e. the one preserved by $\lieg_1$). This is called the \textbf{canonical uniformising Higgs bundle}. We do so assuming that $G$ is of adjoint type, since otherwise they cannot be constructed in general; this will be enough for our purposes. 
		
		Let $E_T$ denote the $T$-bundle obtained from the reduction of $E^0$. The fact that $\varphi^0$ is everywhere regular implies that all the sections $\pi_j(\varphi^0) \in H^0(E_T(\lieg_{\beta_j}) \otimes K_C(D))$ are nowhere vanishing. Therefore, the line bundles $E_T(\lieg_{\beta_j})$ are all isomorphic to $(K_C(D))^{-1}$. Since $G$ is of adjoint type, the previous condition determines a unique $T$-bundle $E_T$. 
		
		Moreover, since all the line bundles are equal, this $E_T$ admits a reduction of structure group to $T' \subseteq T$, the one-dimensional subtorus spanned by the principal cocharacter corresponding to the grading element $\zeta = \sum_{j=1}^r\omega_j^\vee$. Consider the regular nilpotent element $e:= \sum_{j=1}^rX_{\beta_j} \in \lieg_1$ and the associated principal $\mathfrak {sl}_2$-triple $(2\zeta, e, f)$ obtained by the Jacobson--Morozov theorem.
		
		By construction, since $[\zeta,e] = e$, we have $E_{T'}(\left< e \right>) \simeq (K_C(D))^{-1}$. Therefore, we have a constant section $e \in H^0(E_{T'}(\left< e \right>) \otimes K_C(D))$ and, extending the structure group of the pair $(E_{T'},e)$ to $G$, we get a $K_C(D)$-twisted $G$-Higgs bundle $(E^0,\varphi^0)$ which, by construction, reduces to a $K_C(D)$-twisted $(T,\lieg_1)$-Higgs pair (in fact, to a $(T',\left< e \right>)$-Higgs pair).
		
		Finally, we endow it with the unique parabolic structure $Q_{c_i}$ at each $c_i \in D$ such that $u_{c_i}(E^0,\varphi^0,Q^0) = 1$. The result is a parabolic Higgs bundle (i.e. $\varphi^0$ is compatible with this parabolic structure) since $\lieg_1 \subseteq \mathfrak b$. Note that this is the only possible one, since any non-trivial Weyl group element $u \in W$ verifies $u(e) \notin \mathfrak b$.
		
		Now we explicitly construct its upward flow as a section of the Hitchin map, as was done originally in \cite{hitchin_lie-groups_1992}. Given a point $a = (a_1,\dots,a_r) \in \mathcal A(G) =  \bigoplus_{i=1}^rH^0(C,K^{\deg(p_i)}_C((\deg(p_i)-1)D))$ in the Hitchin base, decompose $\lieg$ into irreducible representations of the $\liesl_2(\C)$-representation induced by $(2\zeta,e,f)$, getting
		$$\lieg = \bigoplus_{i=1}^r V_i.$$
		There are $r = \rank \lieg$ summands, and they can be ordered so that $\dim V_i = 2\deg p_i - 1$ where the $p_i \in \C[\lieg]^G$ are the chosen basis elements of invariant polynomials for the construction of $h_G$. Thus, if $f_i \in V_i$ is a lowest weight vector, we have that $[2\zeta,f_i] = (-2\deg p_i + 2) f_i$, and hence $E_{T'}(\gen{f_i}) \otimes K_C(D) \simeq (K_C(D))^{\deg p_i - 1} \otimes K_C(D) \simeq (K_C(D))^{\deg p_i}$. This means that we have a well-defined constant section $f_i \in H^0(C,E_{T'}(\gen{f_i}) \otimes K_C(D) \otimes (K_C(D))^{-\deg p_i})$ and hence $a_if_i \in H^0(C, E^0(\lieg) \otimes K_C) \subseteq H^0(E^0(\lieg) \otimes K_C(D))$. This results in a well defined Higgs field
		$$\varphi^a := \varphi^0 + \sum_{i=1}^ra_if_i$$
		\noindent such that $\varphi^a(c_i) = \varphi^0 \in \lieg_1$. In particular, the parabolic structure $Q^0$ is still compatible with all the Higgs fields $\varphi^a$.
		
		The $f_i$ can be selected so that $h_G(E^0,\varphi^a,Q^0) = a$. We also have that $(E^0,\varphi^a,Q^0)$ are smooth points \cite[Section 5]{hitchin_lie-groups_1992} in $\mathcal M_{sp}(G,\alpha)$. By the $\C^\times$-equivariance of $h_G$, we get the subvariety
		$$\mathcal A(G) \simeq \{(E^0,\varphi^a,Q^0) : a \in \mathcal A(G)\} \subseteq W^+_{(E^0,\varphi^0,Q^0)}.$$
		
		However, both sides are vector spaces of the same dimension, hence equality holds. It then follows from this explicit description of the upward flow that $(E^0,\varphi^0,Q^0)$ is very stable.
		
		Notice from Remark \ref{recoverfromtwistedmult} that when $G$ is not of adjoint type there may not exist an everywhere regular fixed point of Borel type since the association of the line bundle $(K_C(D))^{-1}$ to every simple root need not extend to the cocharacter lattice of $T$ (essentially because $\deg K_C(D)$ may be odd).
	\end{example}
	
	\section{Hecke transformations and affine flag varieties}\label{sechecke}
	
	\subsection{Loop groups and affine flag varieties}
	In order to understand the space of Hecke transformations we first recall the relevant notions about loop groups and affine flag varieties. Our main references are \cite{pappasTwistedLoopGroups2008, richarzSchubertVarietiesTwisted2013,zhu_introduction_2016}.
	
	Let $\mathbb D := \Spec \C[[z]]$ be the formal disk and $\mathbb D^* := \Spec \C((z))$ be the punctured disk. We denote by $0 \in \mathbb D$ the closed point. For a $\C$-algebra $R$ we also define $\mathbb D_R := \Spec(R[[z]])$ and $\mathbb D_R^* := \Spec(R((z)))$. We also have a closed subscheme $0_R := \Spec R \hookrightarrow \mathbb D_R$. The notion of Hecke transformation captures local modifications of principal and parabolic bundles. Hence, the following are the natural definitions in this context. We still assume that $\alpha \in \liet_\R$ is in the interior of the standard Weyl alcove.
	
	\begin{definition}
		The \textbf{affine Grassmannian} $\Gr_G$ is defined by its functor of points: to a $\C$-algebra $R$ it assigns the space of isomorphism classes of pairs $(E,\Psi)$ where $E$ is a principal $G$-bundle over $\mathbb D_R$ and $\Psi : E|_{\mathbb D_R^*} \to \mathbb D_R^* \times G$ is a trivialisation over $\mathbb D_R^*$.
		
		Similarly, the \textbf{affine flag variety} $\Fl_G$ maps $R$ to the space of isomorphism classes of triples $(E,Q,\Psi)$ where $(E,Q)$ is a parabolic $G$-bundle with weight $\alpha$ over $(\mathbb D_R,0_R)$ and $\Psi : E|_{\mathbb D_R^*} \to \mathbb D_R^* \times G$ is a trivialisation over $\mathbb D_R^*$. 
	\end{definition}
	
	In this setting, the parabolic structure $Q$ corresponds to a reduction of structure group of $E$ to $B$ over $0_R$. Both functors have the structure of an ind-projective ind-scheme, and there is a forgetful map $(E,Q,\Psi) \mapsto (E,\Psi)$ defining a projection $\Fl_G \to \Gr_G$.
	
	For our purposes it suffices to understand the $\C$-points, which can be done in terms of loop groups. Consider the \textbf{loop group} $G((z)) := G(\mathbb D^*)$ and the \textbf{positive loop group} $G[[z]] := G(\mathbb D)$. The latter is a subgroup of the former. Similarly, we can define the \textbf{standard Iwahori subgroup} $\mathcal I \subseteq G[[z]]$ by considering the evaluation map
	$$G[[z]] \to G$$
	\noindent and taking the preimage of $B \subseteq G$.
	
	Then, the $\C$-points of $\Gr_G$ can be identified (see \cite[Proposition 1.3.6]{zhu_introduction_2016}) with the quotient $G((z))/G[[z]]$. For a pair $(E,\Psi)$ we can fix a trivialisation $\varepsilon : \mathbb D \times G \xrightarrow{\sim} E$, which always exists since the base is $\mathbb D$, and then we get an element $\Psi \circ \varepsilon \in \Aut(\mathbb D^* \times G) \simeq G((z))$. A change in trivialisation $\varepsilon'$ results in a change by $(\varepsilon')^{-1} \circ \varepsilon \in \Aut(\mathbb D \times G) \simeq G[[z]]$.
	
	Similarly, the $\C$-points of $\Fl_G$ can be identified with $G((z))/\mathcal I$. Indeed, now the trivialisation $\varepsilon : \mathbb D \times G \xrightarrow{\sim} E$ is required to send $0 \times B$ to $Q \subseteq E|_0/B$. Thus, a change in trivialisation results in a change by an element of $\Aut(\mathbb D \times G, 0 \times B) \simeq \mathcal I$.
	
	\begin{remark}
		Since we only need the $\C$-points, we will use the previous identifications liberally, in particular denoting by $\Fl_G$ and $\Gr_G$ the spaces of $\C$-points. Nevertheless, the previous identifications can be defined schematically as well, by considering the loop groups and their subgroups as schemes (for example, $LG(R) := G(\mathbb D_R)$).
	\end{remark}
	
	\begin{remark}\label{twistedaffine}
		For more general parabolic weights $\alpha \in \liet_\R$, everything works identically by replacing parabolic bundles with parahoric torsors \cite{balajiModuliParahorictorsors2015} as indicated in Remark \ref{parahoric}.
		The subgroups of $G((z))$ that appear as automorphisms of the trivial parahoric torsor over $\mathbb D$ are the \textbf{parahoric subgroups} $G_\alpha \subseteq G((z))$, defined by restricting their valuation along each direction of the root space decomposition: they have valuation at least $0$ along $T$, and at least $\lceil -\beta(\alpha) \rceil$ along the root subgroup for $\beta \in \Delta$.
		In particular, for $\alpha = 0$ we have $G_\alpha = G[[z]]$ and for $\alpha$ in the interior of the standard Weyl alcove we have $G_\alpha = \mathcal I$, recovering our two main cases of interest.
	\end{remark}
	
	Now we recall the Bruhat decomposition for both $\Fl_G$ and $\Gr_G$. This can be done simultaneously by setting $\Fl_\alpha := G((z))/G_\alpha$ (see Remark \ref{twistedaffine}) for $\alpha \in \liet_\R$; the case $\alpha = 0$ recovers $\Gr_G$ and any $\alpha$ in the interior of the standard Weyl alcove recovers $\Fl_G$. The following decomposition results will be stated for $\C$-points, but they hold schematically.
	
	Given $\mu \in X_*(T)$ a cocharacter, we may regard it as an element $z^\mu \in G((z))$ by taking the map on $\mathbb C((z))$-points, $\mu: \mathbb C((z))^\times \to T((z)) \subseteq G((z))$, and then taking $z^\mu$ to be the image of $z$. Moreover, since $T \subseteq G_\alpha$, we may map the Weyl group $W = N(T)/T$ into $\Fl_\alpha$. This gives a way to map
	$$\widetilde{W}_G = X_*(T)\rtimes W \rightarrow \Fl_\alpha.$$ 
	
	\begin{remark}
		At this point we can clarify the sign convention in Remark \ref{signconvetion}. We want our action of $X_*(T)$ on $\liet$ to satisfy $G_{\mu(\alpha)} = z^{\mu}G_\alpha z^{-\mu}$. Since $G_{\alpha=0}= G[[z]]$, $z^\mu$ sends $0$ to the point stabilised by
		$$z^\mu G[[z]]z^{-\mu}.$$
		Along the direction of the root $\beta \in \Delta$, this group has elements with valuation at least $\beta(\mu)$. Now, the stabiliser of a point $\alpha \in \liet$ is by definition the parahoric subgroup $G_\alpha$, whose elements along that direction have valuation at least $-\beta(\alpha)$. We therefore conclude that $z^\mu$ maps $0$ to $-\mu$.
		Another common way around this is to change the sign in the definition of $z^\mu$ instead of in the action of $\widetilde{W}_G$. 
	\end{remark}
	
	\begin{definition}
		Let $w \in \widetilde{W}_G$. We define the \textbf{Bruhat cell} of $w$ in $\Fl_\alpha$ as 
		$$\Fl\nolimits_{\alpha,w} := G_\alpha \cdot w \subseteq \Fl\nolimits_\alpha.$$
		
		We define the \textbf{Schubert variety} of $w$ in $\Fl\nolimits_\alpha$ as
		$$\Fl\nolimits_{\alpha, \le w} := \overline{\Fl\nolimits_{\alpha,w}} \subseteq \Fl\nolimits_\alpha.$$
	\end{definition}
		
	For what follows, recall the Bruhat order on $\widetilde{W}_G$ from Definition \ref{bruhatorder} and the parabolic quotients in Definition \ref{parabolicweyl} and Remark \ref{Jmu}.
	
	\begin{proposition}[{\cite{iwahori_bruhat_1965}, \cite[Proposition 0.1]{richarzSchubertVarietiesTwisted2013}}]\label{bruhatdecomposition}
		Let $\alpha,\alpha' \in \liet_\R$ be parabolic weights. We have the Cartan--Iwahori--Matsumoto decomposition in Bruhat cells:
		$$\Fl\nolimits_\alpha := \bigsqcup_{w \in {}_{J_{\alpha'}}\widetilde{W}_G^{J_\alpha}} G_{\alpha'} \cdot w.$$
		
		Furthermore, when $\alpha=\alpha'$, we have the following description of the Schubert varieties:
		$$\Fl\nolimits_{\alpha, \le w} = \overline{\Fl\nolimits_{\alpha,w}} = \bigsqcup_{w' \le w}\Fl\nolimits_{\alpha,w'}$$ 
		\noindent and their dimension equals the length $l(w)$.
	\end{proposition}
	
	As an example let us examine the Bruhat order in the main two cases of interest.
	
	\begin{example}
		First, consider $\alpha = 0$ corresponding to the affine Grassmannian. In this case, $J_0 = S_W$ is the set of simple root reflections, so $(\widetilde{W}_G)_{S_W} = W$ is the Weyl group (such values of $\alpha$ are called \textbf{special} in the literature). The left quotient is $(\widetilde{W}_G)^{S_W} = (X_*(T) \rtimes W)/ W = X_*(T)$. The double quotient ${}_{S_W}(\widetilde{W}_G)^{S_W} = X_*^+(T)$ consists of the dominant cocharacters: abstractly, it is clear that double orbits in $ W \backslash\widetilde{W}_G/W$ contain a dominant cocharacter representative, and it turns out \cite[Corollary 1.8]{richarzSchubertVarietiesTwisted2013} that it is precisely the one in ${}_{S_W}(\widetilde{W}_G)^{S_W}$. Its Bruhat order agrees with the usual dominance order: $\lambda \ge \mu$ if and only if $\lambda - \mu$ is a sum of positive coroots. Now consider $\alpha$ in the interior of the standard alcove. Then, $J_\alpha = \emptyset$, so $(\widetilde{W}_G)_{J_\alpha} = 1$ and the relevant (left, right or double) quotient is ${}_{J_\alpha}(\widetilde{W}_G)^{J_\alpha} = \widetilde{W}_G$ with its Bruhat order.
	\end{example}
	
	These affine flag varieties are related to local transformations of parabolic $G$-bundles, but in the presence of a Higgs field we need to look at certain closed subspaces. We consider the infinitesimal versions of the loop groups above, namely the loop Lie algebra $\lieg((z)) := \lieg \otimes_{\C} \C((z))$ and its parahoric subalgebras for $\alpha \in \liet_\R$ given by
	$$\lieg_\alpha = \liet[[z]] \oplus \bigoplus_{\beta \in \Delta}z^{\lceil -\beta(\alpha)\rceil}\lieg_\beta[[z]] \subseteq \lieg((z)).$$
	For $\alpha=0$ we recover the positive loop Lie algebra
	$$\lieg[[z]] := \lieg \otimes_{\C} \C[[z]]$$
	\noindent and for $\alpha$ in the interior of the standard Weyl alcove we recover the standard Iwahori Lie algebra $\lie i \subseteq \lieg((z))$. Note that, when considering Higgs fields over $\mathbb D$, they are simply sections in $H^0(\mathbb D, E(\lieg))$ without further twisting.
	
	\begin{definition}
		Let $X \in \mathfrak \lieg[[z]]$. The moduli space of isomorphism  classes of pairs $((E,\varphi),\Psi)$ where $(E,\varphi)$ is a $G$-Higgs bundle over $\mathbb D$ and $\Psi : (E,\varphi)|_{\mathbb D^*} \xrightarrow{\sim} (\mathbb D^* \times G, X)$ is a trivialisation over $\mathbb D^*$ is called the \textbf{affine Springer fibre} of $X$ in $\Gr_G$ and denoted $\Gr_G^X$.
				
		Let $X \in \mathfrak i$. The moduli space of isomorphism  classes of pairs $((E,\varphi,Q),\Psi)$ where $(E,\varphi,Q)$ is a parabolic $G$-Higgs bundle over $\mathbb D$ of weight $\alpha$ in the interior of the standard Weyl alcove and $\Psi : (E,\varphi,Q)|_{\mathbb D^*} \xrightarrow{\sim} (\mathbb D^* \times G, X, 0 \times B)$ is a trivialisation over $\mathbb D^*$ is called the \textbf{affine Springer fibre} of $X$ in $\Fl_G$ and denoted $\Fl_G^X$.
	\end{definition}
	
	The reason for the name and the notation is that there is an inclusion inside of $\Gr_G$ and $\Fl_G$, respectively, given by the forgetful maps that omit $\varphi$. Naturally, using the descriptions via quotients of $G((z))$ explained above, we have the following concrete identifications:
	$$\Gr\nolimits_{G}^{X} := \left\{\sigma \in \Gr\nolimits_{G} : \Ad_{\sigma^{-1}}(X) \in \lie \lieg[[z]]\right\} \subseteq \Gr\nolimits_{G},$$
	$$\Fl\nolimits_{G}^{X} := \left\{\sigma \in \Fl\nolimits_{G} : \Ad_{\sigma^{-1}}(X) \in \lie i\right\} \subseteq \Fl\nolimits_{G}.$$
	
	These descriptions follow by taking a given $(E,\Psi)$ or $((E,Q),\Psi)$ and imposing that $\Psi^*(X)$ extends to $0 \in \mathbb D$ and is compatible with the parabolic structure. It is clear from these descriptions how they can be generalised to other $\alpha \in \liet_\R$, and the modular interpretation corresponds to that of \textit{logahoric Higgs torsors} \cite{kydonakisLogahoricHiggsTorsors2024} over $\mathbb D$ with trivialisation over $\mathbb D^*$.
	
	For our purposes with fixed points of Borel type, it is interesting as a first step to understand which elements in $\widetilde{W}_G$ lie in the affine Springer fibre in $\Fl_G$ for the regular nilpotent $e = \sum_{j=1}^rX_{\beta_j}$. It suffices to understand the \textit{larger} case of $\widetilde{W}_{G_{ad}} = \widetilde{W}$.
	
	\begin{lemma}\label{affineflagspringercondition}
		An element $w=\mu \cdot u \in \widetilde{W} = P_{\Delta^\vee} \rtimes W$ belongs to $\Fl^e_{G_{\ad}}$ if and only if $\mu$ is antidominant and $u \in {}^{J_\mu}W$. Equivalently, if and only if $w \in {}^{S_W}{\widetilde{W}}$.
	\end{lemma} 
	\begin{proof}
		Write $\mu = \sum_{i=1}^rn_i\omega_i^\vee$ in terms of the fundamental coweight basis. Then
		$$\Ad_{u^{-1}z^{-\mu}}(e) = \Ad_{u^{-1}}(\sum_{i=1}^rz^{-n_i}X_{\beta_i}) = \sum_{i=1}^rz^{-n_i}X_{u^{-1}(\beta_i)}.$$
		For this to be in the standard Iwahori subalgebra in particular it has to be in $\lieg[[z]]$, so $n_i \le 0$ for all $i$ and $\mu$ is antidominant. Moreover, we need that $u^{-1}(\beta_i)$ is a positive root whenever $n_i = 0$. This can be restated as $l(u^{-1}s_{\beta_i}) > l(u^{-1})$ whenever $n_i = 0$. Inverting, this is the same as $l(s_{\beta_i}u) > l(u)$ for all $\beta_i$ with $\beta_i(\mu) = 0$, as desired.
	\end{proof}
	
	\begin{remark}
		The general philosophy, at least for the adjoint group $G_{ad}$, is that every fixed point of Borel type can be obtained by Hecke transformations of the canonical uniformising Higgs bundle (i.e. everywhere regular nilpotent Higgs field) in Example \ref{paraboliccanonical}. The local picture around a parabolic puncture, as we will detail in the next subsection, is therefore given by $\Fl_{G_{ad}}^e$ from previous Lemma, which is therefore interesting to compare with Remark \ref{sphericalquotient}. 
	\end{remark}
	
	\subsection{Hecke transformations for parabolic \texorpdfstring{$G$}{G}-Higgs bundles}
	In this section we describe Hecke transformations for parabolic $G$-Higgs bundles, both at parabolic and non-parabolic points. Due to the local nature of the operation, at non-parabolic points the analysis is the same as for usual $G$-Higgs bundles \cite[Section 4.2]{gonzalez_very_2025}. 
	
	First, we understand Hecke transformations of principal $G$-bundles \cite{wong_hecke_2013} and then we incorporate the Higgs field and the parabolic structure. Fix a point $c \in C$ and let $C_0 := C \setminus \{c\}$ be the complement of $c$ and $C_1 := \Spec \C[[z]]$ a formal disk. We have a natural open embedding $C_0 \hookrightarrow C$, as well as a map $C_1 \to C$ defined from the algebraic map
	$$\mathcal O_C \to \mathcal O_{C,c} \to \widehat{\mathcal O_{C,c}} \simeq \C[[z]]$$
	\noindent given by localisation at $c$ and completion, where $z$ is a uniformising parameter at $c \in C$. Let $C_{01} = C_0 \cap C_1 \simeq \Spec \C((z))$ be the intersection.

	Now, given a principal $G$-bundle $E$, it defines restrictions $E_0 := E|_{C_0}$, $E_1 := E|_{C_1}$ as well as a gluing map
	$$\Phi : E_1|_{C_{01}} \xrightarrow{\sim} E_0|_{C_{01}}.$$
	By the Beauville--Laszlo theorem, the data $(E_0,E_1,\Phi)$ is equivalent to $E$. More precisely, the category of principal $G$-bundles over $C$ is equivalent to the category of triples $(E_0,E_1,\Phi)$ as above, where a morphism between $(E_0,E_1,\Phi)$ and $(E'_0,E_1',\Phi')$ consists of a pair of $G$-bundle morphisms $(f_0 : E_0 \to E_0',f_1: E_1 \to E_1')$ making the following diagram commute:
	
	\[\begin{tikzcd}[ampersand replacement=\&]
		{E_0|_{C_{01}}} \&\& {E_0'|_{C_{01}}} \\
		\\
		{E_1|_{C_{01}}} \&\& {E_1'|_{C_{01}}}
		\arrow["{f_0|_{C_{01}}}", from=1-1, to=1-3]
		\arrow["\Phi"', from=3-1, to=1-1]
		\arrow["{\Phi'}", from=3-3, to=1-3]
		\arrow["{f_1|_{C_{01}}}", from=3-1, to=3-3]
	\end{tikzcd}\]
	
	\begin{definition}
		Let $E$ be a principal $G$-bundle over $C$. A \textbf{Hecke transformation} of $E$ at $c \in C$ is another principal $G$-bundle $E'$ over $C$ together with an isomorphism
		$$\Psi : E'|_{C_0} \to E|_{C_0}.$$
		Two Hecke transformations $(E',\Psi)$ and $(E'',\Psi')$ are \textbf{equivalent} if and only if there is an isomorphism
		$E' \xrightarrow{f} E''$
		\noindent such that $\Psi' \circ f|_{C_0} = \Psi$.
	\end{definition}
	
	Now we describe the space of Hecke transformations for a fixed principal $G$-bundle $E$. For this, fix an isomorphism
	$$t_1 : E|_{C_1} \to C_1 \times G$$
	\noindent with the trivial $G$-bundle over $C_1$. This can be done because $C_1$ is a formal disk. Note here that it actually suffices that $G|_{C_1}$ is a smooth affine group scheme, allowing all of this framework to work for parahoric torsors.
	
	Then, $E$ defines local data $(E_0,E_1,\Phi')$ which is isomorphic via $(\Id,t_1)$ to the local data $(E_0,C_1 \times G, \Phi)$ with which we work (here $\Phi = \Phi' \circ t_1|_{{C_{01}}}^{-1}$). Given $\sigma \in \Aut(C_{01} \times G) \simeq G((z))$, we can transform
	$$(E_0,C_1 \times G, \Phi) \mapsto (E_0,C_1 \times G, \Phi \circ \sigma).$$
	
	The result is clearly a Hecke transformation, where the isomorphism $\Psi$ is the identity in terms of the local data.	
	
	Suppose now that $\sigma, \sigma' \in G((z))$ produce equivalent Hecke transformations, that is, $(E_0,C_1 \times G, \Phi \circ \sigma) \simeq (E_0,C_1 \times G, \Phi \circ \sigma')$. Let $(f_0,f_1)$ denote the corresponding isomorphisms. Since $\Psi$ is the identity automorphism of $E_0$, we must have $f_0 = \Id$. On the other hand, $f_1 \in \Aut(C_1 \times G) \simeq G[[z]]$. The compatibility condition is
	$$\Phi \circ \sigma = \Phi \circ \sigma' \circ f_1$$
	\noindent so that $\sigma = \sigma'f_1$. In other words, $\sigma$ and $\sigma'$ differ by a positive loop $f_1 \in G[[z]]$.
	
	We have seen that given an element of $\Gr_G = G((z))/G[[z]]$ we obtain a Hecke transformation of $E$, and different elements give inequivalent transformations. We conclude by showing that any transformation is equivalent to one that can be obtained in this way. Suppose that $(E',\Psi)$ is a Hecke transformation of $E$. Consider the local data $(E'_0,E'_1,\Phi')$. Note that we have a distinguished trivialisation induced from $t_1$:
	$$t_1' := (\Phi'): E'_1 \to C_1 \times G$$
	and apply the isomorphism given by $(\Psi,t_1')$, resulting in isomorphic local data $(E_0, C_1 \times G, \Phi'')$ where $\Phi'' = \Psi|_{C_{01}} \circ \Phi' \circ (t_1'|_{C_{01}})^{-1}$. Now, both $\Phi''$ and $\Phi$ are isomorphisms from $C_{01} \times G$ to $E_0|_{C_{01}}$, therefore they must differ by an element of $\Aut(C_{01} \times G) \simeq G((z))$, as desired. Therefore, we get the following.
	
	\begin{proposition}
		Fixing a trivialisation
		$$t_1 : E|_{C_1} \to C_1 \times G$$
		\noindent identifies the space of Hecke transformations of $E$ at $c \in C$ up to equivalence with $\Gr_G$.
	\end{proposition}
	
	Before continuing, let us remark that, had we chosen a different trivialisation $t_1' = \eta \circ t_1$ with $\eta \in \Aut(C_1 \times G) \simeq G[[z]]$, then the local data via $t_1$, which we denoted by $(E_0, C_1 \times G, \Phi)$, is isomorphic to the local data via $t_1'$, which is $(E_0,C_1 \times G, \Phi \circ \eta^{-1})$. Therefore the spaces of Hecke transformations are related by a left multiplication by $\eta \in G[[z]]$. This means that if we do not fix any trivialisation, we still have a well defined \textit{type} of the Hecke transformation, as an element in
	$$G[[z]] \backslash G((z)) / G[[z]] \simeq X_*^{-}(T).$$
	\noindent the space of antidominant cocharacters.
	
	Now we consider the presence of a Higgs field $\varphi \in H^0(E(\lieg) \otimes L)$ where $L$ is any line bundle over $C$, still without parabolic structures. First we understand the local data. The Higgs field gives
	$$\varphi_0 := \varphi|_{C_0} : E_0 \to \lieg \otimes L$$
	$$\varphi_1 := \varphi|_{C_1} : E_1 \to \lieg \otimes L$$
	\noindent where both maps are $G$-equivariant. The data must satisfy the compatibility condition
	
	\[\begin{tikzcd}[ampersand replacement=\&]
		{E_0|_{C_{01}}} \&\& {E_1|_{C_{01}}} \\
		\\
		\& {\lieg \otimes L}
		\arrow["{\varphi_0|_{C_{01}}}"', from=1-1, to=3-2]
		\arrow["\Phi"', from=1-3, to=1-1]
		\arrow["{\varphi_1|_{C_{01}}}", from=1-3, to=3-2]
	\end{tikzcd}.\]
	
	A morphism between local data $(E_0,E_1,\Phi,\varphi_0,\varphi_1)$ and $(E_0',E_1',\Phi',\varphi_0',\varphi_1')$ is again a pair $(f_0,f_1)$ which, on top of being a morphism between $(E_0,E_1,\Phi)$ and $(E_0',E_1',\Phi')$, has to satisfy the compatibility conditions
	$$\varphi_i' \circ f_i = \varphi_i$$
	\noindent for $i \in \{0,1\}$.
	
	\begin{definition}
		Let $(E,\varphi)$ be a $G$-Higgs bundle over $C$. A \textbf{Hecke transformation} of $(E,\varphi)$ at $c \in C$ is another principal $G$-bundle $E'$ over $C$ together with an isomorphism
		$$\Psi : E'|_{C_0} \to E|_{C_0}$$
		\noindent such that $\varphi_0 \circ \Psi = \varphi_0'$.
		Two Hecke transformations $(E',\varphi',\Psi)$ and $(E'',\varphi'',\Psi')$ are \textbf{equivalent} if and only if there is an isomorphism
		$E' \xrightarrow{f} E''$
		\noindent such that $\Psi' \circ f|_{C_0} = \Psi$ and $\varphi'' \circ f = \varphi'$.
	\end{definition}
	
	Note that Hecke transformations for $(E,\varphi)$ are in particular Hecke transformations for $E$. Therefore, we may describe them (as before) by fixing a trivialisation
	$$t_1 : E_1 \to C_1 \times G.$$
	Using $t_1$ we may assume that our local data is $(E_0,C_1 \times G, \Phi, \varphi_0, \varphi_1)$.  Since $\varphi_1 : C_1 \times G \to \lieg \otimes L$ is $G$-equivariant, fixing once and for all a trivialisation of $L$ we may also regard it as an element $\varphi_1 : C_1 \to \lieg$ of $\mathrm{Map}(C_1,\lieg) \simeq \lieg[[z]]$ by evaluating at the identity element of $G$. This is an identification that we will apply liberally.
	
	As before, let $(E',\varphi',\Psi)$ be a Hecke transformation with local data $(E'_0,E'_1, \Phi', \varphi'_0, \varphi'_1)$. Taking any trivialisation $t_1' : E_1 \simeq C_1 \times G$ and applying the isomorphism $(\Psi,t_1')$ allows us to assume that the local data of the transformation is of the form
	$$(E_0, C_1 \times G, \Phi'', \varphi_0, \varphi_1'').$$
	(We know that the local expression of the Higgs field over $C_0$ has to be $\varphi_0$ because the Hecke transformation isomorphism over $C_0$ with respect to this local data is the identity).
	
	Therefore, $\Phi'' = \Phi \circ \sigma$ for some $\sigma \in \Aut(C_{01} \times G) \simeq G((z))$. Then, over $C_{01}$ we are forced to have
	$$\varphi_1'' = \varphi_0 \circ \Phi \circ \sigma = \varphi_1 \circ \sigma = \Ad_{\sigma^{-1}}\varphi_1$$
	\noindent where the last step is due to the $G$-equivariance of $\varphi_1$. Therefore, we must require the condition
	$$\Ad_{\sigma^{-1}}\varphi_1 \text{ extends to } C_1.$$
	Therefore, we get a Hecke modification for every $\sigma \in G((z))$ such that
	$$\Ad_{\sigma^{-1}}(\varphi_1) \in \lieg[[z]].$$
	As before, two such modifications are equivalent if and only if they differ by a positive loop on the right.
	
	\begin{proposition}\label{heckespace}
		Fixing a trivialisation
		$$t_1 : E|_{C_1} \to C_1 \times G$$
		\noindent as well as a trivialisation of $L$ identifies the space of Hecke transformations of $(E,\varphi)$ at $c \in C$ up to equivalence with the affine Springer fibre $\Gr_G^{\varphi_1}$
	\end{proposition}
	
	Similarly to the previous analysis, changing the trivialisation to $t_1' = \eta \circ t_1$ results in a change of Springer fibre, since $\Ad_{\eta^{-1}}(\varphi_1') = \varphi_1.$ The previous relation between affine Grassmannians for each trivialisation (multiplying by $\eta$) naturally maps one space to the other.
	$$\Gr_G^{\Ad_{\eta}(\varphi_1)} = \eta \Gr_G^{\varphi_1}.$$
	
	Finally, we study Hecke transformations for parabolic $G$-Higgs bundles. Once again we start by understanding the local data. Fortunately, parabolic structures are local in nature, so if $c_i \neq c$ we can identify $E|_{c_i}/B \simeq E_0|_{c_i}/B$ and if $c_i = c$ we can identify $E|_{c_i}/B \simeq E_1|_{c_i}/B$. The local data is then
	$$(E_0,E_1,\Phi,\varphi_0,\varphi_1,Q).$$
	
	The only main difference when considering Hecke transformations for parabolic $G$-bundles and $G$-Higgs bundles is that bundle isomorphisms must be compatible with the parabolic structures.
	
	\begin{definition}
		Let $(E,\varphi,Q)$ be a parabolic $G$-Higgs over $C$. A \textbf{Hecke transformation} of $(E,\varphi,Q)$ at $c \in C$ is another parabolic $G$-Higgs bundle $(E',\varphi',Q')$ over $C$ together with a parabolic isomorphism
		$$\Psi : E'|_{C_0} \to E|_{C_0}$$
		\noindent such that $\varphi_0 \circ \Psi = \varphi_0'$. The isomorphism being parabolic means that $\Psi|_{c_i}(Q'_i) = Q_i$ for every $c_i \in D \setminus \{c\}$.
		Two Hecke transformations $(E',\varphi',\Psi)$ and $(E'',\varphi'',\Psi')$ are \textbf{equivalent} if and only if there is a parabolic isomorphism
		$E' \xrightarrow{f} E''$
		\noindent such that $\Psi' \circ f|_{C_0} = \Psi$ and $\varphi'' \circ f = \varphi'$.
	\end{definition}
	
	The study of Hecke transformations mimics the case without parabolic structures, so we only mention the differences. If $c \notin D$, everything is identical. If $c \in D$ (say, without losing generality, that $c=c_1$) we fix a trivialisation
	$$t_1 : E_1 \to C_1 \times G$$
	\noindent that sends the parabolic structure to the trivial $B$-orbit, i.e. such that $t_1|_{c_1}(Q_1) = B \in (C_1 \times G)|_{c_1}/B = G/B$. Then, since $B$ is self-normalising, we have that $\Aut(C_1 \times G,B) \simeq \mathcal I$ and the compatibility of the Higgs field with the parabolic structure at $c_1$ translates to $\varphi_1 \in \lie i$. Therefore, we have:
	
	\begin{proposition}\label{heckespaceparab}
		Fixing a trivialisation
		$$t_1 : E|_{C_1} \to C_1 \times G,$$
		\noindent with $t|_{c_1}(Q_1) = B \in (C_1 \times G)|_{c_1}/B = G/B$, as well as a trivialisation of $L$, identifies the space of Hecke transformations of $(E,\varphi,Q)$ at $c_1 \in D$ up to equivalence with the affine Springer fibre
		$$\Fl\nolimits_G^{\varphi_1} := \{\sigma \in \Fl\nolimits_G : \Ad_{\sigma^{-1}}(\varphi_1) \in \lie i\} \subseteq \Fl\nolimits_G$$
		\noindent inside of the affine flag variety $\Fl\nolimits_G= G((z))/\mathcal I$.
	\end{proposition} 
	
	Choosing a different trivialisation $t_1' = \eta \circ t_1$, where now $\eta \in \mathcal I$ (since $t_1'$ must also correspond $Q_1$ with the trivial parabolic structure $B$), results in a different space of Hecke transformations, since now $\varphi_1' = \Ad_{\eta}(\varphi_1)$. As before, they are related by left multiplication by $\eta$.
	$$\Fl\nolimits_G^{\Ad_{\eta}(\varphi_1)} = \eta\Fl\nolimits_G^{\varphi_1}.$$
	
	Thus, without fixing any trivialisation we still have a well defined \textit{type} of the Hecke transformation, which is an element in
	$$\mathcal I \backslash G((z)) / \mathcal I = \widetilde{W}_G.$$  
	
	\begin{remark}\label{heckespacesrelation}
		Suppose that the trivialisation $t_1$ that we fix does not send the parabolic structure to the trivial one, but rather $t|_{c_1}(Q_1) = gB \in (C_1 \times G)|_{c_1}/B = G/B$. Then, everything proceeds similarly, with the only difference that $\Aut(C_1 \times G, gB) = g\mathcal Ig^{-1}$. Therefore the space of Hecke transformations is
		$$\Fl\nolimits_{G,g}^{\varphi_1} = \{\sigma \in \Fl\nolimits_{G,g} : \Ad_{\sigma^{-1}}(\varphi_1) \in \Ad_{g}(\lie i)\} \subseteq \Fl\nolimits_{G,g} := G((z))/g\mathcal I g^{-1}.$$
		
		If we change the previous trivialisation by $t_1' = \eta \circ t_1$, where $\eta \in G[[z]]$ (i.e. without forcing that the parabolic structure is mapped to the the same element), the relation between the spaces of Hecke transformations is as follows:
		$$\Fl\nolimits_{G,\eta(c_1)g}^{\Ad_{\eta}(\varphi_1)} = \eta\Fl\nolimits_{G,g}^{\varphi_1}\eta(c_1)^{-1}.$$
	\end{remark}
	
	\begin{example}\label{exampleheckemod}
		Suppose $D = c$ for $c \in C$, structure group $G=\PGL_2(\C)$, maximal torus $T$ of diagonal matrices and Borel $B$ of lower triangular matrices. Let $\beta = \beta_1$ denote the positive root. Take the following parabolic $G$-Higgs bundle, which is a fixed point of Borel type:
		$$E = \mathcal O_C \oplus (K_C(D))^{-1}(c)$$
		$$\varphi = \begin{pmatrix}
			0 & 0\\
			s_c & 0
		\end{pmatrix}$$
		$$Q = 0 \subsetneq (K_C(D))^{-1}|_c \subsetneq E|_c.$$
		
		Here $s_c \in H^0(C,\mathcal O(c))$ denotes the defining section. Note that any choice of flag $Q$ is compatible with $\varphi$ since the residue vanishes at $c$. We pick as our trivialisation $t_1$ one that splits as a direct sum of a trivialisation of each line bundle summand. Then, with a Hecke transformation we may modify just the flag by using the (standard) Weyl group component of $\widetilde{W}$. For example, take
		$$\sigma = \begin{pmatrix}
			0 & -1\\
			1 & 0
		\end{pmatrix}.$$
		
		This element is in the space of Hecke modifications, i.e. in the affine Springer fibre at $\varphi_1 = zX_\beta$. Indeed, $\Ad_{\sigma^{-1}}(zX_\beta) = zX_{-\beta} \in \lie i$. It acts by permuting the transition functions of both line bundle summands, therefore the resulting bundle is
		
		$$E' = (K_C(D))^{-1}(c) \oplus \mathcal O_C$$
		$$\varphi' = \begin{pmatrix}
			0 & s_c \\
			0 & 0
		\end{pmatrix}$$
		$$Q' = 0 \subsetneq (\mathcal O_C)|_c \subsetneq E'|_c.$$
		
		Notice that only the parabolic structure $Q'$ is changed, whereas $E'$ and $\varphi'$ are the same. 
		
		By using the cocharacter component of $\widetilde{W}$ we may change the vanishing multiplicities of our Higgs field. Indeed, take now the fundamental cocharacter 
		$$\sigma = z^{\omega^\vee} = \begin{pmatrix}
			1 & 0 \\
			0 & z
		\end{pmatrix}.$$ 
		It is in the space of Hecke modifications, because $\Ad_{\sigma^{-1}}(zX_\beta) = z^{1+\beta(-\omega^\vee)}X_\beta = X_\beta \in \lie i$. The resulting modification is
		$$E' = \mathcal O_C \oplus (K_C(D))^{-1}$$
		$$\varphi' = \begin{pmatrix}
			0 & 0\\
			1 & 0
		\end{pmatrix}$$
		$$Q' = Q.$$
		
		Of course, we may perform modifications combining both operations if they are compatible. For example, if we take 
		$$\sigma = \begin{pmatrix}
			1 & 0\\
			0 & z
		\end{pmatrix}\begin{pmatrix}
			0 & -1\\
			1 & 0
		\end{pmatrix}$$
		
		\noindent the product of the previous two transformations, it is not compatible since $\Ad_{\sigma^{-1}}(zX_\beta) = X_{-\beta} \notin \lie i$. Indeed, this is reflected by the fact that $\mathcal O_C \oplus (K_C(D))^{-1}$ with Higgs field $\begin{pmatrix}
			0 & 0 \\
			1 & 0
		\end{pmatrix}$ only admits one parabolic structure at $c$ and it cannot be modified without breaking the compatibility with $\varphi$. Of course, we can also perform modifications by more general elements in $\Fl_G$ (not necessarily in $\widetilde{W}$) but these are not as easily expressible as the previous ones (they involve nontrivial extensions of the line bundles).  
	\end{example}
	
	\subsection{Explicit curves in Hecke transformation spaces}	
	Now we construct some explicit curves in $\Fl_G$ that showcase the closure relations in Proposition \ref{bruhatdecomposition}. The main interest for this is that these particular curves are contained in certain affine Springer fibres, providing paths in the spaces of Hecke transformations explained above. Via the projection $\Fl_G \to \Gr_G$ we obtain similar curves that work outside of parabolic punctures.
	
	We consider two main kinds of curves. We start by the simplest ones, which project to a point in $X_*(T) \subseteq \Gr_G$. These are of course well-known from the study of the classical Bruhat decomposition of $G/B$. Conceptually, they correspond to curves in the space of parabolic structures of a fixed meromorphic Higgs bundle.
	
	\begin{lemma}\label{bruhatcurve}
		Let $w \in \widetilde{W}_G$ and $\beta \in \Delta^+$. Let $u_{-\beta} : \mathbb G_a \to G$ be a root subgroup for $-\beta$. Then, the expression $w \cdot u_{-\beta}(\lambda) \subseteq \Fl_G$, for $\lambda \in \C^\times$, defines a curve whose closure is isomorphic to $\mathbb P^1_{\C}$ in $\Fl_G$, with limits $w$ as $\lambda \to 0$ and $ws_\beta$ when $\lambda \to \infty$.
	\end{lemma}
	\begin{proof}
		It suffices to work in the subgroup of $G$ corresponding to the $\liesl_2$-subalgebra for $\beta$. There,
		$$u_{-\beta}(\lambda) = \begin{pmatrix}
			1 & 0\\
			\lambda & 1
		\end{pmatrix}.$$
		It is clear that when $\lambda \to 0$ this limits to the identity. The other limit is computed by using
		$$u_{-\beta}(\lambda) = \begin{pmatrix}
			1 & \lambda^{-1}\\
			0 & 1
		\end{pmatrix}\begin{pmatrix}
		0 & -1\\
		1 & 0
		\end{pmatrix}\begin{pmatrix}
		\lambda & 1 \\
		0 & \lambda^{-1}
		\end{pmatrix}.$$
		The rightmost term lies in $B$ so it vanishes in the quotient $\Fl_G$, and the limit when $\lambda \to \infty$ is $s_\beta = \begin{pmatrix}
			0 & -1\\
			1 & 0
		\end{pmatrix}$.
	\end{proof}
	
	Now we show that, whenever the endpoints lie in the affine Springer fibre $\Fl_{G}^e$ for the regular nilpotent element $e=\sum_{j=1}^rX_{\beta_j}$, the curve also does.
	
	\begin{lemma}\label{bruhatcurvespringer}
	Let $e=\sum_{j=1}^rX_{\beta_j}$. Let $w \in \widetilde{W}_G$ and $\beta \in \Delta^+$ be such that $l(ws_\beta) = l(w)+1$ and $w \le ws_\beta$ in the Bruhat order. Assume that $w,ws_\beta \in \Fl_{G}^e$. Then, the curve in Lemma \ref{bruhatcurve} for $w$ and $\beta$ is also contained in $\Fl_{G}^e$.
	\end{lemma}
	\begin{proof}		
		
		Write $w=z^\mu \cdot u$ with $u \in W$. We know that $\mu$ is antidominant (recall Lemma \ref{affineflagspringercondition}) and then $w,ws_\beta \in \Fl_{G}^e$ is equivalent to $u^{-1}(\beta_p), s_\beta u^{-1}(\beta_p) \in \Delta^+$ for every $s_{\beta_p} \in J_\mu$. The statement of the Lemma follows by checking that
		$$\Ad_{\exp(X_{-\beta})u^{-1}}(X_{\beta_p}) \in \lie b$$
		\noindent holds for every $s_{\beta_p} \in J_\mu$.
		
		In order to show this, it suffices to prove that
		$$u^{-1}(\beta_p) - k\beta$$
		\noindent can never be a negative root when $k \ge 0$. Suppose otherwise, i.e. that $u^{-1}(\beta_p)-k\beta \in \Delta^-$ for some $k \ge 0$. Then, we must have $u^{-1}(\beta_p)+\beta \in \Delta^+$: if not, the root reflection $s_\beta$ (which acts on the root string through $u^{-1}(\beta_p)$ with shift $\beta$ by \textit{flipping} it) would map $u^{-1}(\beta_p)$ to a negative root. Therefore, the root string through $u^{-1}(\beta_p)$ with shift $\beta$ has length at least three. We start with the case where the length is exactly three. In other words, the root string consists of 
		$$-\gamma := u^{-1}(\beta_p)-\beta,$$ 
		$$u^{-1}(\beta_p),$$ \noindent and 
		$$-s_\beta(\gamma) = u^{-1}(\beta_p)+\beta,$$
		\noindent with $\gamma \in \Delta^+$.
		
		Applying $u$, we have $-us_\beta(\gamma) = \beta_p + u(\beta)$. Note that $l(u)<l(us_\beta)$ implies $u(\beta) \in \Delta^+$, so we deduce $us_\beta(\gamma) \in \Delta^-$. This implies (e.g. \cite[Section 5.6]{humphreys1990reflection} applied to a reduced expression of $us_\beta$) that we can write
		$$\gamma = s_{l_k}\dots s_{l_{j+1}}(\beta_{l_j})$$
		\noindent for some $1 \le j \le k$, where
		$$u = s_{l_1}\dots s_{l_{i-1}}s_{l_{i+1}}\dots s_{l_k}$$
		\noindent is a reduced expression and 
		$$s_\beta = s_{l_k}\dots s_{l_{i+1}}s_{l_i}s_{l_{i+1}}\dots s_{l_k},$$
		\noindent for simple root reflections $s_{l_i} = s_{\beta_{l_i}} \in S_W$. Note that here we use that the lengths of $u$ and $us_\beta$ differ by one.
		
		On the other hand, applying $u$ to the other endpoint of the root string gives $-u(\gamma) = \beta_p - u(\beta)$. Since $\beta_p$ is simple (and therefore cannot be the sum of two positive roots) we deduce that $u(\gamma) \in \Delta^+$. From the reduced expression of $u$ this implies that $j \le i$. We claim that $i=j$, for if $j < i$ we can compute
		$$s_\beta(\gamma) = (s_{l_k}\dots s_{l_{i+1}}s_{l_i}s_{l_{i+1}}\dots s_{l_k})s_{l_k}\dots s_{l_{j+1}}(\beta_{l_j}) = s_{l_k}\dots \widehat{s_{l_i}}\dots s_{l_{j+1}}(\beta_{l_j})$$
		\noindent and then, using the reduced expression of $u$ and \cite[Section 5.6]{humphreys1990reflection} again, we deduce $s_\beta(\gamma) \in \Delta^+$, contradicting the fact that $s_\beta(\gamma) = -(u^{-1}(\beta_p)+\beta)$. However, this means that $\beta = \gamma$, so that $u^{-1}(\beta_p) = \beta - \gamma = 0$, a contradiction.
		
		Having ruled out the case where the root string has length three, it suffices to address the case of length four. Root strings of this length only occur when one of the simple factors of $\lieg$ is of type $\lieg_2$, and there a manual check suffices: if $\beta_1$ denotes the short simple root and $\beta_2$ the long simple root, the Weyl group elements in the Springer fibre over $X_{\beta_1}$ are $$\{1, s_2, s_2s_1, s_2s_1s_2, s_2s_1s_2s_1, s_2s_1s_2s_1s_2\}$$\noindent and for $X_{\beta_2}$ they are $$\{1, s_1, s_1s_2, s_1s_2s_1, s_1s_2s_1s_2, s_1s_2s_1s_2s_1\}.$$ In both cases, we see that the condition on the length (differing by one) forces $\beta$ to be a simple root, but then $u^{-1}(\beta_p)-k\beta \in \Delta^-$ for some $k \ge 0$ forces the root string to be of the form $\{-\beta,0,\beta\}$, which is of length three.
	\end{proof}
	
	\begin{remark}\label{chainproperty}
		Lemma \ref{bruhatcurvespringer} implies that if $w = z^\mu \cdot u \in \Fl_{G}^e$ is such that $u$ is not the longest element under the Bruhat order of ${}^{J_\mu}W$, then $\beta \in \Delta^+$ can be found such that the curve $wu_{-\beta}(\lambda)$ belongs to $\Fl_{G}^e$ and so does the limit $ws_\beta \ge w$. This is due to the \textbf{chain property} of ${}^{J_\mu}W$ \cite[Theorem 2.5.5]{bjorner2005combinatorics}, one of which consequences is that any element different than the maximum has another element covering it in the Bruhat order, with length differing by one.
	\end{remark}

	The second kind of curve projects to a curve in $\Gr_G$ and involves shifting the cocharacter part by a coroot.
	
		\begin{proposition}\label{affineflagcurve}
		Let $z^\mu \cdot u \in \widetilde{W}_G$ and $\beta \in \Delta$. Let $u_\beta : \mathbb G_a \to G$ be a root subgroup for $\beta$. Let $m:=\beta(\mu)+1$. Then, the expression
		$$u_\beta(z^m\lambda^{-1})z^{\mu+\beta^\vee} s_\beta u \in G((z)),$$
		\noindent for $\lambda \in \C^\times$, defines a curve in $\Fl_G$ with limit $z^{\mu}u$ as $\lambda \to 0$ and $z^{\mu+\beta^\vee}s_\beta u$ as $\lambda \to \infty$.
	\end{proposition}
	\begin{proof}
	Since $u_\beta(0) = \Id$ it is clear that the limit when $\lambda \to \infty$ is the claimed one. Now we compute the other limit. Start by assuming $u=1$. Note that the curve may be rewritten as 
	$$u_\beta(z^m\lambda^{-1})z^{\mu+\beta^\vee} s_\beta = z^{\mu}u_\beta(z^{m-\beta(\mu)}\lambda^{-1})z^{\beta^\vee}s_\beta = z^{\mu}u_\beta(z\lambda^{-1})z^{\beta^\vee}s_\beta.$$
	This expression has the advantage that the factor multiplying $z^\mu$ on the right is fully contained in the subgroup of $G((z))$ corresponding to the $\liesl_2$-subalgebra of $\lieg$ associated to $\beta$. Thus, we may work in the copy of $\SL_2((z))$ mapping to said subgroup, and rewrite that term as
	$$\begin{pmatrix}
		1 & \lambda^{-1}z\\
		0 & 1
	\end{pmatrix} \begin{pmatrix}
		z & 0\\
		0 & z^{-1}
	\end{pmatrix}\begin{pmatrix}
		0 & -1\\
		1 & 0
	\end{pmatrix} = \begin{pmatrix}
		\lambda^{-1} & -z\\
		z^{-1} & 0
	\end{pmatrix}.$$
	Now we can rewrite as follows:
	$$\begin{pmatrix}
		\lambda^{-1} & -z\\
		z^{-1} & 0
	\end{pmatrix} = \begin{pmatrix}
		1 & 0\\
		\lambda z^{-1} & 1
	\end{pmatrix}\begin{pmatrix}
		\lambda^{-1} & -z\\
		0 & \lambda
	\end{pmatrix}.$$
	
	The rightmost term can be disregarded in $\Fl_G$ since it is mapped inside of $\mathcal I$, therefore we see that the limit as $\lambda \to 0$ of the above expression is the identity, so the limit of the whole curve is $z^{\mu}$.
	
	If $u$ is not trivial, we may observe the following:
	
	$$u_\beta(z^m\lambda^{-1})z^{\mu+\beta^\vee} s_\beta u = uu_{u^{-1}(\beta)}(\pm z^m\lambda^{-1})z^{u^{-1}(\mu+\beta^\vee)} s_{u^{-1}(\beta)}$$
	\noindent where the sign change in the root subgroup is not relevant for our purposes. Since now $u$ is at the very left, it does not interfere with the right quotient by $\mathcal I$ and we may repeat the computations above (notice that $u^{-1}(\beta)(u^{-1}(\mu)) = \beta(\mu)$ so the multiplicity $m$ does not change) to find that the limit at $\lambda \to 0$ is $uz^{u^{-1}(\mu)} = z^\mu u$.
	\end{proof}

	The remainder of this section is destined, as before, to investigating when these curves are contained in $\Fl_{G}^e$.
	
	\begin{lemma}\label{affineflagcurvespringer}
		Let $z^\mu \cdot u \in \widetilde{W}_G$ and $\beta \in \Delta$. Let $e = \sum_{i=1}^rX_{\beta_i} \in \lieg$. Then, the curve in Proposition \ref{affineflagcurve} is contained in $\Fl^e_{G}$ if and only if, for every simple root $\beta_i \in \Pi$, either
		$$-\mu(\beta_i) - \beta^\vee(\beta_i) - q_i > 0,$$
		\noindent or
		$$-\mu(\beta_i) - \beta^\vee(\beta_i) - q_i = 0 \,\text{ and }\,u^{-1}(s_\beta(\beta_i+q_i\beta)) \in \Delta^+,$$
		\noindent where $q_i \ge 0$ is the greatest integer such that $\beta_i + q_i\beta \in \Delta$.
	\end{lemma}
	\begin{proof}
		By definition of the affine Springer fibre, we just have to check when does $$\Ad_{u^{-1}s_\beta z^{-\mu-\beta^\vee}u_\beta(-z^m\lambda^{-1})}(e)$$ \noindent belong to the standard Iwahori subalgebra $\mathfrak i = \mathrm{Lie}(\mathcal I)$. For this, it suffices to check that every simple root vector $X_{\beta_i}$ is mapped to $\lie{i}$. Applying the first factor, we get
		$$\Ad_{u_\beta(-z^m\lambda^{-1})}(X_{\beta_i}) = \sum_{j_i=0}^{q_i}c_{ij}z^{j_im}X_{\beta_i+j_i\beta}$$
		\noindent where the $c_{ij}$ are nonzero complex numbers. Now we apply the second factor to each of those terms:
		$$\Ad_{z^{-\mu-\beta^\vee}}(z^{j_im}X_{\beta_i + j_i\beta}) = z^{j_im-\mu(\beta_i+j_i\beta)-\beta^\vee(\beta_i+j_i\beta)}X_{\beta_i + j_i\beta}.$$
		The valuation of the term above simplifies (using $\beta^\vee(\beta) = 2$ and $m = \beta(\mu)+1$) to
		$$j_im-\mu(\beta_i)-j_i(m-1)-\beta^\vee(\beta_i)-2j_i = -\mu(\beta_i)-\beta^\vee(\beta_i)-j_i.$$
		Therefore, the minimum valuation that appears in the terms coming from $X_{\beta_i}$ is $-\mu(\beta_i)-\beta^\vee(\beta_i)-q_i$. The two remaining operations (i.e. applying $s_\beta$ and then $w^{-1}$) only permute the root vectors and do not change the valuations. The result follows by asking these to be positive or zero, noting that the latter requires the corresponding root vector to be in a positive root space.
	\end{proof}
	
	The conditions in the previous Lemma apply whenever $\mu$ is not minuscule and we have already arranged the part in ${}^{J_\mu}W$ to be \textit{maximal} by following the curves of the first kind.
	
	\begin{lemma}\label{affineflagcurvespringer2}
		Let $z^\mu \cdot u \in \widetilde{W}_G \cap \Fl^e_{G}$ and suppose that $w$ is the maximal element under the Bruhat order of ${}^{J_\mu}W$ and $\mu$ is not minuscule. Then, there exists $\beta \in \Delta^+$ such that $\mu+\beta^\vee$ is antidominant and the curve in Proposition \ref{affineflagcurve} is contained in $\Fl^e_G$.
	\end{lemma}
	\begin{proof}
		We construct $\beta$ verifying the properties of Lemma \ref{affineflagcurvespringer}. By \cite[Corollary 5.4]{gonzalez_very_2025} there exists $\beta^\vee$ such that for every $\beta_i \in \Delta^+$ we have
		$$-\mu(\beta_i) - \beta^\vee(\beta_i) - q_i \ge 0.$$
		Now we show that, for the chosen $u$, equality of the above expression automatically implies that $u^{-1}(s_\beta(\beta_i+q_i\beta)) \in \Delta^+$. Since $u \in {}^{J_\mu}W$ is the longest, we have:
		$$u^{-1}(\beta_i) \in \Delta^+\, \text{ for simple roots }\, \beta_i \in \Pi \text{ with } \mu(\beta_i) = 0$$
		$$u^{-1}(\gamma) \in \Delta^-\, \text{ for positive roots }\, \gamma \in \Delta^+ \text{ with } \mu(\gamma) \neq 0.$$
		Now observe that, when we apply $s_\beta$, we get
		$$s_\beta(\beta_i+q_i\beta) = \beta_i - (q_i+\beta^\vee(\beta_i))\beta.$$
		If equality holds in the first expression of the proof, that is, if $q_i+\beta^\vee(\beta) = -\mu(\beta_i)$, we have
		$$s_\beta(\beta_i+q_i\beta) = \beta_i + \mu(\beta_i)\beta.$$
		Finally, in order to determine whether $u^{-1}(\beta_i + \mu(\beta_i)\beta)$ is a positive root, we simply compute
		$$\mu(\beta_i + \mu(\beta_i)\beta) = (1+\mu(\beta))\mu(\beta_i).$$
		Since $\mu + \beta^\vee$ is antidominant, we have $\mu(\beta) \le -\beta^\vee(\beta) = -2$. As a consequence, the factor $1+\mu(\beta) \le -1$ never vanishes. Therefore, it suffices to analyse these two cases:
		\begin{itemize}
			\item If $\mu(\beta_i) = 0$, then $u^{-1}(\beta_i + \mu(\beta_i)\beta) =u^{-1}(\beta_i) \in \Delta^+$.
			\item If $\mu(\beta_i) \neq 0$ then $\mu(\beta_i + \mu(\beta_i)\beta) \neq 0$ and $\beta_i + \mu(\beta_i)\beta$ is a negative root ($\beta_i$ is simple, $\beta$ positive and $\mu(\beta_i) \le 0$) so $u^{-1}(\beta_i + \mu(\beta_i)\beta) \in \Delta^+$.
		\end{itemize}
	\end{proof}
	
	For later use we need the following shifted versions of the previous curves, which have the advantage of allowing for values of $\mu$ in the coweight lattice (even if they $\mu$ does not lift to $T$).
	
	\begin{lemma}\label{bruhatcurveshifted}
		Let $w=\mu \cdot u \in \widetilde{W} = P_{\Delta^\vee} \rtimes W$ and $\beta \in \Delta$. Let $u_\beta : \mathbb G_a \to G$ be a root subgroup for $\beta$. Then, the expression $u_{-\beta}(\lambda)$ for $\lambda \in \C^\times$ defines a curve in $\Fl_G$ with limit $1$ as $\lambda \to 0$ and $s_\beta$ as $\lambda \to \infty$. Additionally, if $u$ is not the maximal element under the Bruhat order of ${}^{J_\mu}W$, then there exists $\beta \in \Delta^+$ such that the previous curve is contained in $\Fl^{\Ad_{w^{-1}}(e)}_G$.
	\end{lemma}
	
	\begin{proof}
		This follows directly by taking the corresponding curve in Lemma \ref{bruhatcurvespringer}, which can only be guaranteed to be defined when projected to the adjoint group $G_{\ad} = G/Z(G)$ (since $\widetilde{W} = \widetilde{W}_{G_{ad}}$), and then multiplying on the left by $w^{-1}$. The result belongs to $w^{-1}\Fl^e_{G_{\ad}} = \Fl^{\Ad_{w^{-1}}(e)}_{G_{\ad}}$ and, as can be seen from its expression, lifts to $\Fl^{\Ad_{w^{-1}}(e)}_{G} \subseteq \Fl_G$.
	\end{proof}
	
	\begin{proposition}\label{affineflagcurvespringershifted}
		Let $w=\mu \cdot u \in \widetilde{W} = P_{\Delta^\vee} \rtimes W$ and $\beta \in \Delta$. Let $u_\beta : \mathbb G_a \to G$ be a root subgroup for $\beta$. Let $m:=\beta(\mu) + 1$. Then, the expression
		$$u_{u^{-1}(\beta)}(z^{m-\beta(\mu)}\lambda^{-1})z^{u^{-1}(\beta^\vee)}u^{-1}s_\beta u \in G((z)),$$
		\noindent for $\lambda \in \C^\times$, defines a curve in $\Fl_G$ with limit $1$ as $\lambda \to 0$ and $z^{u^{-1}(\beta^\vee)}u^{-1}s_\beta u$ as $\lambda \to \infty$. Additionally, if $\mu$ is antidominant, not minuscule and $u$ is the maximal element under the Bruhat order of ${}^{J_\mu}W$, then there exists $\beta \in \Delta^+$ such that the previous curve is contained in $\Fl^{\Ad_{w^{-1}}(e)}_G$.
	\end{proposition}
	
	\begin{proof}
		This follows directly by taking the corresponding curve in Lemmas \ref{affineflagcurvespringer} and \ref{affineflagcurvespringer2} for the adjoint group $G_{\ad} = G/Z(G)$ and multiplying on the left by $w^{-1}$. The result belongs to $w^{-1}\Fl^e_{G_{\ad}} = \Fl^{\Ad_{w^{-1}}(e)}_{G_{\ad}}$ and, as can be seen from its expression, lifts to $\Fl^{\Ad_{w^{-1}}(e)}_{G} \subseteq \Fl_G$.
	\end{proof}

	\section{Classification of very stable points of Borel type}\label{sectheorem}
		
	\subsection{The \texorpdfstring{$\C^\times$}{C*}-action on the space of Hecke transformations}
	The main strategy for the classification of very stable points of Borel type is to introduce a $\C^\times$-action on the space of Hecke transformations of $(E,\varphi,Q) \in \mathcal M_{sp}(G,\alpha)^{\C^\times}$ at a point $c \in C$, in such a way that it intertwines with the $\C^\times$-action on $\mathcal M_{sp}(G,\alpha)$. Now we describe this action for \textbf{any} fixed point $(E,\varphi,Q)$ (for any Vinberg $\C^\times$-pair $(G_0,\lieg_1)$ with grading element $\zeta \in \lieg_0$ and corresponding cocharacter $\xi_\lambda \in G_{ad}$) and in the remaining subsections we specialise to the Borel type. We keep the same notation of the previous sections, in particular we recall that Hecke transformations are performed at $c \in C$, $C_1$ is a formal disk around $c$ and $C_0 = C \setminus \{c\}$.
	
	There are two kinds of important local trivialisations at play: those aligned with the $(G_0,\lieg_1)$-reduction and those aligned with the parabolic structure.
	
	\begin{definition}
		Let $(E,\varphi,Q) \in \mathcal M_{sp}(G,\alpha)^{\C^\times}$ be a parabolic $G$-Higgs bundle fixed by the $\C^\times$-action reducing to a Vinberg pair $(G_0,\lieg_1)$ and $c \in C$. A local trivialisation $t_1 : E_1 \to C_1 \times G$ is called a \textbf{$G_0$-trivialisation} if $t_1(E_{G_0}) = C_1 \times G_0$. If, furthermore, we have $\varphi_1 \in \lieg_1[[z]]$, then it is called a \textbf{$(G_0,\lieg_1)$-trivialisation}.
		
		If $c \in D$, a local trivialisation $t_1 : E_1 \to C_1 \times G$ is called \textbf{parabolic} if $t_1(Q_c) = B \in G/B$.
	\end{definition}
	
	\begin{remark}\label{trivialisationsrelation}
		It is clear (since every principal bundle trivialises over a formal disk $C_1$) that the previous trivialisations always exist. However, in general it is not possible to have simultaneously a $(G_0,\lieg_1)$-trivialisation which is parabolic at $c \in D$. Given a $(G_0,\lieg_1)$-trivialisation, Theorem \ref{fixedpointsparab} shows that $t_1(Q_c) = g_0 wB$ for some $g_0 \in G_0$, $w \in W_{G_0} \backslash W$. In that case,
		$$t_1^p := w^{-1}g_0^{-1} \cdot t_1 : E_1 \to C_1 \times G$$
		\noindent defines a parabolic $G_0$-trivialisation. We say that it is \textbf{associated} to the $(G_0,\lieg_1)$-trivialisation $t_1$. 
	\end{remark}
	
	Recall that Proposition \ref{heckespaceparab} identifies the space of Hecke transformations at $c \in D$ with $\Fl_G^{\varphi_1}$ using a parabolic trivialisation. Therefore, if we instead work with a $(G_0,\lieg_1)$-trivialisation with $t_1(Q_c) = g_0wB$, the space becomes $$\Fl^{\varphi_1}_{G,g_0w} = g_0w\left(\Fl_G^{\Ad_{w^{-1}g_0^{-1}}(\varphi_1)}\right)w^{-1}g_0^{-1}$$
	\noindent according to Remarks \ref{heckespacesrelation} and \ref{trivialisationsrelation}.
	
	\begin{proposition}\label{heckeequivariance}
		Let $(E,\varphi,Q) \in \mathcal M_{sp}^{\C^\times}(G,\alpha)$ be a parabolic $G$-Higgs bundle fixed by the $\C^\times$-action and fix a $(G_0,\lieg_1)$-trivialisation $t_1 : E_1 \to C_1 \times G$ identifying the space of Hecke transformations with $\Gr_G^{\varphi_1}$ (if $c \notin D$, see Proposition \ref{heckespace}) or with $\Fl_{G,g_0w}^{\varphi_1}$ (if $c \in D$, where $t_1(Q_c) = g_0wB \in G/B$, see Proposition \ref{heckespaceparab} and the discussion above).
		
		There is a natural $\C^\times$-action in this space of Hecke transformations:
		$$\sigma \mapsto \xi_\lambda^{-1}\sigma.$$ Denoting by $(E',\varphi',Q')$ the Hecke transformation by $\sigma$ and by $(E'_\lambda, \varphi'_\lambda,Q'_\lambda)$ the Hecke transformation by $\xi_\lambda^{-1}\sigma$, we have
		$$(E',\lambda \varphi',Q') \simeq (E'_\lambda,\varphi'_\lambda,Q'_\lambda).$$
	\end{proposition}
	
	\begin{proof}
		We assume $c \in D$, otherwise the proof is identical by ignoring the parabolic structure at $c$ (see also \cite[Proposition 4.19]{gonzalez_very_2025}). It is clear that the action is well defined on the space of Hecke transformations, since $\Ad{\sigma^{-1}\xi_\lambda}(\varphi_1) = \lambda \cdot \Ad_{\sigma^{-1}}(\varphi_1)$.
		
		The isomorphism
		$$(E,\varphi,Q) \simeq (E, \lambda \varphi,Q)$$
		\noindent can be expressed using the local data over $C_0$ and $C_1$, $(E_0,C_1 \times G,\Phi,\varphi_0,\varphi_1, Q_0, g_0wB)$ and $(E_0,C_1 \times G,\Phi,\lambda \varphi_0,\lambda \varphi_1, Q_0, g_0wB)$ by $(f_\lambda, \xi_\lambda)$ where $f_\lambda : E_0 \xrightarrow{\sim} E_0$ is the restriction of the isomorphism to $C_0$. The compatibility conditions with the Higgs field are
		$$(\lambda \varphi_0) \circ f_\lambda = \varphi_0$$
		$$(\lambda \varphi_1) \circ \xi_\lambda = \varphi_1$$
		\noindent (the last one holds automatically due to $G$-equivariance and the fact that $\Ad_{\xi_\lambda^{-1}}(X) = \lambda^{-1}X$ for any $X \in \lieg_1$). The glueing condition is
		$$f_\lambda \circ \Phi = \Phi \circ \xi_\lambda$$
		\noindent over $C_{01}$.
		
		Now we perform two different operations. On one hand we apply a Hecke transformation by $\sigma \in G((z))$ and then we apply the $\C^\times$-action, yielding
		$$(E_0,C_1 \times G, \Phi\circ \sigma, \lambda \varphi_0, \lambda \Ad_{\sigma^{-1}}(\varphi_1),Q_0,g_0wB).$$
		
		On the other hand we apply a Hecke transformation by $\xi_\lambda^{-1}\sigma$, yielding
		$$(E_0,C_1 \times G, \Phi\circ \xi_\lambda^{-1}\sigma, \varphi_0, \lambda \Ad_{\sigma^{-1}}(\varphi_1),Q_0,g_0wB).$$
		
		We claim that both are isomorphic, using the pair $(f_\lambda^{-1},\Id)$. First we check the glueing condition, which is
		$$f_\lambda^{-1} \Phi \sigma = \Phi \circ \xi_\lambda^{-1}\sigma.$$
		This is equivalent to $f_\lambda \circ \Phi = \Phi \circ \xi_\lambda$ which we know from above is true. Then we check the compatibility with the Higgs fields:
		$$\varphi_0 \circ f_\lambda^{-1} = \lambda \varphi_0$$
		$$\lambda \Ad_{\sigma^{-1}}(\varphi_1) \circ \Id = \lambda \Ad_{\sigma^{-1}}(\varphi_1).$$
		The second is tautological and the first holds as seen above.
	\end{proof}
	
	\begin{remark}\label{conjugatedheckeequivariance}
		Note that if $c\in D$ and we take instead the associated parabolic $G_0$-trivialisation, the action on $\Fl_G^{\varphi_1}$ becomes $\sigma \mapsto w^{-1}g_0^{-1}\xi_\lambda^{-1}g_0w\sigma = w^{-1}\xi_\lambda^{-1}w\sigma$.
	\end{remark}
	
	\subsection{Hecke transformations of fixed points of Borel type and very stable upward flows}
	
	For the remainder of the section we assume that fixed points are of Borel type. In this case, fixing a parabolic $T$-trivialisation and performing a Hecke transformation by an element in $\widetilde{W}_G \cap \Fl_G^{\varphi_1}$ results in a new fixed point (since $\widetilde{W}_G$ is fixed by the $\C^\times$-action in Proposition \ref{heckeequivariance}) of Borel type, and we can relate their multiplicity divisors.
	
	\begin{proposition}\label{twistedmulthecke}
		Let $(E,\varphi,Q) \in \mathcal M_{sp}(G,\alpha)^{\C^\times}$ be a fixed point of Borel type and fix a parabolic $T$-trivialisation at $c \in D$ associated to some $(T,\lieg_1)$-trivialisation. Let $w = z^\mu \cdot u \in \Fl_G^{\varphi_1} \cap \widetilde{W}_G$ and denote by $(E',\varphi',Q')$ the corresponding Hecke transformation. It is again a fixed point of Borel type, with
		$$w_c(E',\varphi',Q') = w_c(E,\varphi,Q) \cdot w$$
		\noindent and $w_p(E',\varphi',Q') = w_p(E,\varphi,Q)$ for $p \neq c$.
	\end{proposition}
	
	\begin{proof}
		Denote $X:=\varphi_1 \in \mathfrak i$ the local expression of the Higgs field in the fixed trivialisation. We abbreviate $\mu_c := \mu_c(E,\varphi)$ and $u_c := u_c(E,\varphi,Q)$. In the induced trivialisation for $(E',\varphi',Q')$, which is still parabolic, we have that the Higgs field is $\Ad_{u^{-1}z^{-\mu}}(X) \in \mathfrak i$. We also know that $\Ad_{u_c}(X) \in \lieg_1[[z]]$ and $\Ad_{u_c'u^{-1}z^{-\mu}}(X) \in \lieg_1[[z]]$. This means that the $u_c'u^{-1}z^{-\mu}u_c^{-1}$ must be a cocharacter (otherwise it would send $\Ad_{u_c}(X)$, which is of the form $\sum_{i=1}^ru_i(z)z^{k_i}X_{\beta_i}$ where $u_i(0) \neq 0$, outside of $\lieg_1[[z]]$). We rewrite
		$$u_c'u^{-1}z^{-\mu}u_c^{-1} = (z^{-u_c'u^{-1}(\mu)})u_c'u^{-1}u_c^{-1}.$$
		Therefore, $u_c' = u_cu$ and the resulting cocharacter is $z^{-u_c(\mu)}$, so $\mu_c' = \mu_c - u_c(\mu)$. We can check in $\widetilde{W}$ that
		$$w_c \cdot w = (\mu_c^{-1}u_c)(\mu \cdot u) = \mu_c^{-1}u_c(\mu)u_cu = (u_c(\mu)^{-1}\mu_c)^{-1}u_cu,$$
		\noindent as desired.
	\end{proof} 
	
	For completeness (although we will not need it) we note that the same result clearly holds for $c \notin D$ and $z^\mu \in \Gr_G^{\varphi_1} \cap X_*^-(T)$. 
	
	We introduce the following terminology.
	
	\begin{definition}\label{reduced}
		We say that the twisted multiplicity divisor $w(E,\varphi,Q)$ of a fixed point of Borel type is \textbf{reduced} if its coefficient at every $c \in D$ is minimal under the Bruhat order of $\widetilde{W}$, and its coefficient at every $c \notin D$ is minimal under the Bruhat order of $W \backslash \widetilde{W}/W = P_{\Delta^\vee}^-$.
		
		Equivalently, it is reduced if for every $c \in D$ we have $l(w_c(E,\varphi,Q)) = 0$ (see Proposition \ref{lengthzero}) and for every $c \notin D$ we have that $w_c(E,\varphi,Q)$ is minuscule.
	\end{definition}
	
	Reduced twisted multiplicity divisors will correspond to very stable upward flows. We already know that, in the simplest case where $w(E,\varphi,Q) = 1$ from Example \ref{paraboliccanonical}, the upward flow equals the Hitchin section and it is indeed very stable. Now we illustrate that if $w(E,\varphi,Q)$ is supported at $D$ and reduced, we still get very stable sections of the Hitchin map as upward flows.
	
		\begin{example}\label{lengthzeroverystable}
		Assume that $G=G_{ad}$ is of adjoint type and suppose that $(E,\varphi,Q)$ is a smooth fixed point of Borel type with twisted multiplicity divisor $w(E,\varphi,Q)$ supported at $D$ (i.e. such that $w_c(E,\varphi,Q) = 1$ whenever $c \notin D$). Assume that $l(w_c(E,\varphi,Q)) = 0$ for $c \in D$, see Proposition \ref{lengthzero}. We are going to prove that $(E,\varphi,Q)$ is very stable by constructing its upward flow. We do so by performing a sequence of Hecke transformations at the parabolic punctures.
		
		We start at $c_1 \in D$ with a local $(T,\lieg_1)$-trivialisation $(C_{c_1} \simeq \Spec C[[z]], t_{c_1})$ around $c_1$ and perform the Hecke transformation by $\mu_{c_1}(E,\varphi)u_{c_1}(E,\varphi,Q)^{-1} \in \widetilde{W} = \widetilde{W}_G$. This is equivalent to performing the Hecke transformation by
		$$u_{c_1}^{-1}(E,\varphi,Q)\left(\mu_{c_1}(E,\varphi)u_{c_1}(E,\varphi,Q)^{-1}\right)u_{c_1}(E,\varphi,Q) = w_{c_1}(E,\varphi,Q)^{-1}$$
		\noindent in the associated parabolic $T$-trivialisation, so by Proposition \ref{twistedmulthecke} the resulting transformation $(E^1,\varphi^1,Q^1)$ has
		$$w_{c_1}(E^1,\varphi^1,Q^1) = w_{c_1}(E,\varphi,Q)w_{c_1}(E,\varphi,Q)^{-1} = 1$$
		\noindent and, since $1 \in \Fl^e$, it is a well-defined parabolic $G$-Higgs bundle.
		
		The transformation $(E^1,\varphi^1, Q^1)$ comes with an induced local trivialisation around $c_1$ which may no longer be $(T,\lieg_1)$-trivialisation. However, we can change it by a left multiplication by $u_{c_1}(E,\varphi,Q)^{-1}$ in order to make it again a $(T,\lieg_1)$-trivialisation at $c_1$, which is moreover parabolic at $c_1$. Of course, we can recover $(E,\varphi,Q)$ by performing a Hecke transformation by $w_{c_1}(E,\varphi,Q)$ with this trivialisation, which reverses the whole process.
		
		Now we iterate this at every puncture $c_i$, reaching an everywhere regular fixed point $(E^s,\varphi^s,Q^s)$ with $w(E^s,\varphi^s,Q^s) = 1$. By Example \ref{paraboliccanonical}, this fixed point is isomorphic to the canonical uniformising $(E^0,\varphi^0,Q^0)$. Therefore, applying this isomorphism, we may assume that the local expression of the Higgs field in our local trivialisations around each $c_i$ is constant, equal to $e = \sum_{j=1}^rX_{\beta_j}$. 
		
		An important remark is that, even after having applied this isomorphism, the Hecke transformations that revert back to $(E,\varphi,Q)$ are still given by the elements $w_{c_i}(E,\varphi,Q)$ in the new local trivialisations. Indeed, \textit{a priori} we can only guarantee that they have to be given by certain elements in the Schubert cell
		$$\mathcal I \cdot w_{c_i}(E,\varphi,Q) \in \Fl_G.$$
		However, by Proposition \ref{bruhatdecomposition}, since $l(w_{c_i}(E,\varphi,Q))=0$, this Schubert cell consists of a single point equal to $w_{c_i}(E,\varphi,Q)$.
		
		To conclude, recall from the conclusion in Example \ref{paraboliccanonical} that the upward flow of $(E^0,\varphi^0,Q^0)$ is given by the Hitchin section $(E^0,\varphi^a,Q^0)$ for $a \in \mathcal A_G$. Since $\varphi^a(c_i) = \varphi^0(c_i)$ at every $c_i \in D$, locally we have
		$$\varphi^a|_{C_{c_i}} = e + \sum_{j=1}^ra_j^i(z) \in \mathfrak i,$$ 
		\noindent where $a_j^i(z) \in z\left< f_j \right >[[z]]$ is the local expression of the section $a_jf_j \in H^0(E^0(\lieg) \otimes K_C(D))$ used to define $\varphi^a$. Thus, again by Proposition \ref{bruhatdecomposition} and $l(w_{c_i}(E,\varphi,Q)) = 0$, which implies that $w_{c_i}(E,\varphi,Q)$ normalises $\mathcal I$, we may perform the Hecke modification by $w_{c_i}(E,\varphi,Q)$ at every $c_i$ to $(E^0,\varphi^a,Q^0)$ as well. Doing this in order (from $c_s$ back to $c_1$) results in $(E,(\varphi^a)',Q)$, which define a section of the Hitchin map: locally around $c_i$ (and after conjugating by $u_{c_i}(E,\varphi,Q)$ to make the resulting trivialisation a $(T,\lieg_1)$-trivialisation) we have
		$$(\varphi^a)'|_{C_{c_i}} = zX_{\beta_{j_0}} + \sum_{j \neq j_0}X_{\beta_j} + \sum_{j=1}^r(a')_j^i(z).$$
		
		\noindent where $j_0$ is such that $\mu_{c_i}(E,\varphi) = \omega_{j_0}^\vee$, and $(a')_j^i(z) \in \lieg[[z]]$ is given by $(a')_j^i(z) = \Ad_{z^{-\omega_{j_0}^\vee}}a_j^i(z)$. Note that $(a')_j^i(z)$ does not have poles as $\omega_{j_0}^\vee$ is minuscule, so its pairing with every root is at most one, while the $a_j^i(z)$ have at least valuation one. From the local expressions above it is clear that this section is in the upward flow of $(E,\varphi,Q)$: applying the automorphism given by $\xi_\lambda^{-1}$ we have
		$$\Ad_{\xi_\lambda^{-1}}(\lambda \cdot (\varphi^a)'|_{C_{c_i}}) = zX_{\beta_{j_0}} + \sum_{j \neq j_0}X_{\beta_j} + \sum_{j=1}^r\lambda^{\deg(p_j)-1}(a')_j^i(z)$$
		\noindent since every $f_j$ has degree $1-\deg(p_j)$ in the Borel grading, as explained in Example \ref{paraboliccanonical}. Thus, the rightmost sum vanishes when $\lambda \to 0$.
		
		Therefore, we get the subvariety
		$$\mathcal A(G) \simeq \left\{(E,(\varphi^a)',Q) : a \in \mathcal A(G)\right\} \subseteq W^+_{(E,\varphi,Q)}.$$
		
		Moreover, both sides are vector spaces of the same dimension (since $(E,\varphi,Q)$ is smooth), hence equality holds. It then follows from this explicit description of the upward flow that $(E,\varphi,Q)$ is very stable.
	\end{example}
	
	\begin{example}
		We now clarify the construction in the previous example in the simplest possible case. Let $G=\PGL_2(\C)$, $D=c$ and $w_c(E,\varphi,Q) = (-\omega_1^\vee) \cdot s_{\beta_1}$. The corresponding fixed point, in terms of vector bundles, is isomorphic to
		$$E = \mathcal O_C \oplus K_C^{-1}$$
		$$\varphi = \begin{pmatrix}
			0 & 0 \\
			s_c & 0
		\end{pmatrix}$$
		$$Q = \{E|_c \supseteq \mathcal O_C|_c \supseteq 0\},$$
		\noindent where $s_c \in H^0(\mathcal O_C(c))$ is the canonical section. Given a point $a \in H^0(K_C^2(c))$ in the Hitchin base, the upward flow is given by keeping $E$ and $Q$ and using the Higgs fields
		$$(\varphi^a)' = \begin{pmatrix}
			0 & a\\
			s_c & 0
		\end{pmatrix}.$$
	\end{example}
	
	\subsection{Hecke curves in the moduli space}
	
	Using Proposition \ref{heckeequivariance}, the curves inside the Hecke transformation spaces studied in Lemmas \ref{bruhatcurve}, \ref{bruhatcurvespringer}, \ref{affineflagcurvespringer}, \ref{affineflagcurvespringer2} and Propositions  \ref{affineflagcurve}, \ref{affineflagcurvespringershifted} induce curves in the moduli space $\mathcal M_{sp}(G,\alpha)^{\C^\times}$. The goal of this subsection is to verify this by showing that the $\C^\times$-action preserves these curves and that stability is also maintained.
	
	\begin{proposition}\label{affineflagcurvetype2}
		Let $w = \mu \cdot u \in \widetilde{W}$ and $\beta \in \Delta^+$ such that $u \le us_{\beta}$ in the Bruhat order of $\widetilde{W}$. There is an element $\sigma_{w,\beta} \in \Fl_G$ such that, for the $\C^\times$-action in Proposition \ref{heckeequivariance} conjugated by $u^{-1}$ (see Remark \ref{conjugatedheckeequivariance}), we have
		$$\lim\limits_{\lambda \to 0}\lambda \cdot \sigma_{w,\beta} = 1$$
		$$\lim\limits_{\lambda \to \infty}\lambda \cdot \sigma_{w,\beta} = s_\beta.$$
	\end{proposition}
	\begin{proof}
		The element is that of Lemma \ref{bruhatcurveshifted}, i.e. it is $u_{-\beta(\lambda)}(1)$ where $u_{-\beta} : \mathbb G_a \to G$ is the root subgroup for $-\beta$.
		
		Indeed,
		$$\lambda \cdot \sigma_{w,\beta} = u^{-1}\xi_\lambda^{-1}u \cdot u_{-\beta}(1) = (u^{-1}\xi)_{\lambda}^{-1}u_{-\beta}(1) = u_{-\beta}(\lambda^k)$$
		\noindent where $k = \beta(u^{-1}(\zeta)) = u(\beta)(\zeta) > 0$, since $u(\beta) \in \Delta^+$ because $u \le us_\beta$.
	\end{proof}
	
	\begin{proposition}\label{affineflagcurvetype1}
		Let $w=z^\mu \cdot u \in \widetilde{W}$ and $\beta \in \Delta^+$. There is an element $\sigma'_{w,\beta} \in \Fl_G$ such that, for the $\C^\times$-action in Proposition \ref{heckeequivariance} conjugated by $u^{-1}$ (see Remark \ref{conjugatedheckeequivariance}), we have
		$$\lim\limits_{\lambda \to 0}\lambda \cdot \sigma'_{w,\beta} = 1$$
		$$\lim\limits_{\lambda \to \infty}\lambda \cdot \sigma'_{w,\beta} = z^{u^{-1}(\beta^\vee)}u^{-1}s_\beta u.$$
	\end{proposition}
	\begin{proof}
		Let $m:=\beta(\mu)+1$. This element is
		$$\sigma'_{w,\beta} := u_{u^{-1}(\beta)}(z^{m-\beta(\mu)})z^{u^{-1}(\beta^\vee)}u^{-1}s_\beta u.$$
		
		Indeed,
		\begin{align*}
			\lambda \cdot \sigma'_{w,\beta} &= \left(u^{-1}\xi^{-1}_\lambda u\right) u_{u^{-1}(\beta)}(z^{m-\beta(\mu)})z^{u^{-1}(\beta^\vee)}u^{-1}s_\beta u \\
			&= u_{u^{-1}(\beta)}(\lambda^{-\text{ht}(\beta)}z^{m-\beta(\mu)})z^{u^{-1}(\beta^\vee)}u^{-1}s_\beta u \cdot u^{-1}(s_\beta^{-1}(\xi_\lambda))
		\end{align*}	
		
		The rightmost element belongs to $T \subseteq \mathcal I$. Therefore, we obtain (up to an orientation-preserving reparametrisation) the curve whose limits are described in Proposition \ref{affineflagcurvespringershifted}.
	\end{proof}
	
	Now we show that, unless $w \in \widetilde{W} \cap \Fl_{G_{ad}}^e$ is minimal under the Bruhat order (i.e. unless it has length zero, see Proposition \ref{lengthzero}), we can always find curves of either of the two previous types, contained within $\Fl^{\Ad_{w^{-1}}(e)}_{G}$ and limiting at $1$ when $\lambda \to 0$.
	
	\begin{proposition}\label{affineflagwobblycurves}
		Let $w = \mu \cdot u \in \widetilde{W} \cap \Fl_{G_{ad}}^e$ and suppose that $l(w) \neq 0$. Then, there exists $\tilde{\sigma}_{w} \in \Fl^{\Ad_{w^{-1}}(e)}_{G}$ such that $\lambda \cdot \tilde{\sigma}_w \in \Fl^{\Ad_{w^{-1}}(e)}_{G}$ for all $\lambda \in \C^\times$ (for the $\C^\times$-action in Proposition \ref{heckeequivariance} conjugated by $u^{-1}$, see Remark \ref{conjugatedheckeequivariance}) and
		$$\lim\limits_{\lambda \to 0}\lambda \cdot \tilde{\sigma}_w = 1.$$
	\end{proposition}
	\begin{proof}
		Recall from Lemma \ref{affineflagspringercondition} that $\mu$ is antidominant and $u \in {}^{J_\mu}W$ and from Proposition \ref{lengthzero} that either $\mu$ is not minuscule or $u$ is not the longest in ${}^{J_\mu}W$.
		
		If $\mu$ is not minuscule, by Proposition \ref{affineflagcurvespringershifted} we can find some positive root $\beta \in \Delta^+$ such that we can take $\tilde{\sigma}_w = \sigma_{w,\beta}'$ as in Proposition \ref{affineflagcurvetype1}.
		
		On the other hand, if $u$ is not maximal within ${}^{J_\mu}W$, by Remark \ref{chainproperty} there is another $u' = us_\beta \in {}^{J_\mu}W$ for some $\beta \in \Delta^+$ such that $u' \ge u$, and by Lemma \ref{bruhatcurveshifted} we can take $\tilde{\sigma}_w = \sigma_{w,\beta}$ as in Proposition \ref{affineflagcurvetype2}.
	\end{proof}
	
	Now we consider the stability of the points in these curves, so that they are fully contained in the moduli space. If we start with a smooth (in particular, stable) fixed point with twisted multiplicity $w_c(E,\varphi,Q)$ and consider its Hecke transformations by the curve $\lambda \cdot \tilde{\sigma}_w$ having it as limit when $\lambda \to 0$, by openness and $\C^\times$-invariance of stability the result is stable for all $\lambda \in \C^\times$, and we just need to consider the other endpoint (when $\lambda \to \infty$). We do so for both kinds of curves.
	
	\begin{proposition}\label{bruhatstable}
		Let $c \in D$ and let $(E,\varphi,Q)$ be a fixed point of Borel type. Let $w_c = (E,\varphi,Q)$, $\mu_c = \mu_c(E,\varphi)$ and $u_c = u_c(E,\varphi,Q)$, so that $w_c = z^{-\mu_c} \cdot u_c \in \widetilde{W} \cap \Fl_{G_{\ad}}^e$. Fix a local parabolic $T$-trivialisation around $c$ associated to a $(T,\lieg_1)$-trivialisation, identifying the space of Hecke transformations with $\Fl^{\Ad_{w_c^{-1}}(e)}_G$ by Proposition \ref{heckespaceparab}. Let $\beta \in \Delta^+$ be such that $\sigma_{w_c,\beta} \in \Fl^{\Ad_{w_c^{-1}}(e)}_G$. Let $(E',\varphi',Q')$ denote the Hecke transformation by
		$$\lim\limits_{\lambda \to \infty} \lambda \cdot \sigma_{w_c,\beta} = s_\beta.$$
		Then, if $(E,\varphi,Q)$ is (semi)stable, so is $(E',\varphi',Q')$.
	\end{proposition}
	
	\begin{proof}
		We apply Proposition \ref{boreltypestabilityparab}. Let $\chi \in X_+^*(T)$. Since $(E,\varphi) \simeq (E',\varphi')$ as meromorphic Higgs bundles, we have $d_E = d_{E'}$. Moreover, if $i \in \{1,\dots,s\}$ is the index such that $c=c_i$, we have by Proposition \ref{twistedmulthecke} that $\chi(u_{c_j}(E,\varphi,Q)(\alpha_j)) = \chi(u_{c_j}(E',\varphi',Q')(\alpha_j))$ for $j \neq i$, while 
		\begin{align*}
			\chi(u_{c_i}(E',\varphi',Q')(\alpha_i)) &= \chi(u_cs_\beta(\alpha_i)) = \chi(u_c(\alpha_i)) - \beta(\alpha_i)\chi(u_c(\beta^\vee)) \\
			&\le \chi(u_c(\alpha_i)) = \chi(u_{c_i}(E,\varphi,Q)(\alpha_i))
		\end{align*}
		\noindent since $\beta(\alpha_i) \ge 0$, $u_c(\beta^\vee) \in \Delta^{\vee}_+$ (because $u_cs_\beta \ge u_c$) and $\chi$ is dominant. 
	\end{proof}
	
	\begin{proposition}\label{bruhatstable2}
		Let $c \in D$ and let $(E,\varphi,Q)$ be a fixed point of Borel type. Let $w_c = (E,\varphi,Q)$, $\mu_c = \mu_c(E,\varphi)$ and $u_c = (E,\varphi,Q)$, so that $w_c = z^{-\mu_c} \cdot u_c \in \widetilde{W} \cap \Fl_{G_{\ad}}^e$. Fix a local parabolic $T$-trivialisation associated to a $(T,\lieg_1)$-trivialisation around $c$, identifying the space of Hecke transformations with $\Fl^{\Ad_{w_c^{-1}}(e)}_G$. Let $\beta \in \Delta^+$ be such that $\sigma'_{w_c,\beta} \in \Fl^{\Ad_{w_c^{-1}}(e)}_G$. Let $(E',\varphi',Q')$ denote the Hecke transformation by
		$$\lim\limits_{\lambda \to \infty} \lambda \cdot \sigma'_{w_c,\beta} = z^{u_c^{-1}(\beta^\vee)}u_c^{-1}s_\beta u_c.$$
		Then, if $(E,\varphi,Q)$ is (semi)stable, so is $(E',\varphi',Q')$.
	\end{proposition}
	
	\begin{proof}
		
		We apply Proposition \ref{boreltypestabilityparab}. Let $\chi \in X_+^*(T)$. By Proposition \ref{twistedmulthecke} we have $w_{c'}(E',\varphi',Q') = w_{c'}(E,\varphi,Q)$ for $c' \neq c$, while $w_c(E',\varphi',Q') = w_c(E,\varphi,Q) \cdot z^{u_c^{-1}(\beta^\vee)}u_c^{-1}s_\beta u_c$. Therefore, 
		$$\mu_c(E',\varphi') = \mu_c(E,\varphi) - \beta^\vee$$
		$$u_c(E',\varphi',Q') = s_\beta u_c(E,\varphi,Q).$$
		
		Thus,
		\begin{align*}
			\chi(d_{E'}) + \sum_{i=1}^s\chi(u_{c_i}(E',\varphi',Q')(\alpha_i)) 
			&= \chi(d_E) - \chi(\beta^\vee) + \sum_{c' \in D \setminus \{c\}}\chi(u_{c'}(E,\varphi,Q)(\alpha_{c'})) \\
			&\quad + \chi(s_\beta u_c(\alpha_c)) \\[0.5em]
			&= \chi(d_{E}) + \sum_{i=1}^s\chi(u_{c_i}(E,\varphi,Q)(\alpha_i)) \\
			&\quad + \bigl(-\chi(\beta^\vee) + \chi(s_\beta u_c(\alpha_c) - u_c(\alpha_c))\bigr).
		\end{align*}
		
		Since $(E,\varphi,Q)$ is stable, it suffices to prove that the last term $-\chi(\beta^\vee) + \chi(s_\beta u_c(\alpha_c) - u_c(\alpha_c))$ is not positive. Indeed:
		$$-\chi(\beta^\vee)+\chi\big(s_\beta u_c(\alpha_{c}) - u_c(\alpha_{c})\big) = -\chi(\beta^\vee) + \chi(-\beta(u_c(\alpha_{c}))\beta^\vee) = -(1+\beta(u_c(\alpha_{c})))\chi(\beta^\vee).$$
		
		Since $\chi$ is dominant and $\beta$ positive, $\chi(\beta^\vee) \ge 0$. Moreover, since $\alpha_{c}$ is in the interior of the standard Weyl alcove, we have $\beta(u_c(\alpha_{c})) \in (-1,1)$, so the term $1+\beta(u_c(\alpha_{c}))$ is also positive.
	\end{proof}
	
	\subsection{Classification theorem}
	
	We are in position to use all the analysis in the previous sections to establish the classification of very stable fixed points of Borel type. Recall the notion of a reduced twisted multiplicity divisor from Definition \ref{reduced}.
	
	\begin{theorem}\label{main}
		Let $(E,\varphi,Q) \in \mathcal M_{sp}(G,\alpha)$ be a smooth fixed point of Borel type. Then, $(E,\varphi,Q)$ is very stable if and only if $w(E,\varphi,Q)$ is reduced.
	\end{theorem}
	
	\begin{proof}		
		Suppose first that there is some $c \in C$ with $l(w_c(E,\varphi,Q)) \neq 0$. We may assume that $c \in D$, since the case with $c \notin D$ (which means that $\mu_c(E,\varphi)$ is not minuscule) was carried out in \cite[Theorem 5.9]{gonzalez_very_2025} using the projections of the curves $\sigma'_{w,\beta}$ to $\Gr_G$, and works here identically due to the local nature of Hecke transformations. 
		
		Now, fix a local parabolic $T$-trivialisation for $(E,\varphi,Q)$ around $c$ associated to a $(T,\lieg_1)$-trivialisation, identifying the space of Hecke transformations with $\Fl^{\Ad_{w_c(E,\varphi,Q)^{-1}}(e)}_G$ by Proposition \ref{heckespaceparab}. Let $(E_\lambda,\varphi_\lambda,Q_\lambda)$ for $\lambda \in \C^\times$ denote the Hecke transformations by $\lambda \cdot \tilde{\sigma}_{w_c(E,\varphi,Q)}$, where $\tilde{\sigma}_{w_c(E,\varphi,Q)}$ is that of Proposition \ref{affineflagwobblycurves}. This is well defined since $\lambda \cdot \tilde{\sigma}_{w_c(E,\varphi,Q)} \in \Fl^{\Ad_{w_c(E,\varphi,Q)^{-1}}(e)}_G$. Moreover, since $\lim\limits_{\lambda \to 0} \lambda \cdot \tilde{\sigma}_{w_c(E,\varphi,Q)} = 1$, we have
		$$\lim\limits_{\lambda \to 0}(E_\lambda,\varphi_\lambda, Q_\lambda) \simeq (E,\varphi,Q).$$
		By openness and $\C^\times$-invariance of stability, every $(E_\lambda,\varphi_\lambda,Q_\lambda)$ is stable. Moreover, the bundle
		$$(E_\infty,\varphi_\infty,Q_\infty) := \lim\limits_{\lambda \to \infty}(E_\lambda,\varphi_\lambda,Q_\lambda)$$
		\noindent is also stable by Propositions \ref{bruhatstable} and \ref{bruhatstable2}. The whole curve and its limits must lie in $(h_G)^{-1}(0)$ (since Hecke transformations preserve the Hitchin fibre) so they are nilpotent and, in particular, strongly parabolic (see Remark \ref{nilpotentissp}). This means that $(E,\varphi,Q)$ is wobbly, as the $(E_\lambda, \varphi_\lambda, Q_\lambda)$ for $\lambda \in \C^\times$ are nilpotent in its upward flow.
		
		For the converse, start by assuming that $G=G_{ad}$ is of adjoint type. Assume that all the $w_c(E,\varphi,Q)$ have zero length. The proof is by induction on the size of the support of $w(E,\varphi,Q)$, and it is very similar to that of \cite[Theorem 5.9]{gonzalez_very_2025} but using the techniques developed in previous sections. The base case corresponds to empty support, in which case the construction in Example \ref{paraboliccanonical} shows that it is very stable and its upward flow is a section of $h_G$. Otherwise, take $c \in C$ with $w_c(E,\varphi,Q) \neq 1$. Choose a $(T,\lieg_1)$-trivialisation for this $c \in C$. Assume that it is not very stable. Then, there is $(E',\varphi',Q') \in h_G^{-1}(0) \cap W^+_{(E,\varphi,Q)}$. By Proposition \ref{parabolicupwardflow} the Higgs field can be locally expressed as 
		$$\varphi'_1 = \varphi_1 + \sum_{\beta \in \Delta^{-}}f_\beta(z)X_{\beta} + \sum_{j=1}^rf_j(z)\beta^\vee_j,$$
		\noindent where the coefficients $f_\beta$, $f_j$ are in $\C[[z]]$. Therefore, if $c \notin D$, the proof of \cite[Theorem 5.9]{gonzalez_very_2025} continues identically to conclude that $(E,\varphi,Q)$ is very stable, by performing a Hecke transformation by the minuscule coweight $\omega_j^\vee$ such that $z^{\omega_j^\vee} \in \Gr_G^{\varphi_1}$, which therefore also satisfies $z^{\omega_j^\vee} \in \Gr_G^{\varphi_1'}$, to reduce the size of the support of the twisted multiplicity divisor and apply the induction hypothesis. We may then assume that $w(E,\varphi,Q)$ is supported at $D$. But then the explicit construction of the upward flow in Example \ref{lengthzeroverystable} shows that $(E,\varphi,Q)$ is very stable.
		
		Finally, if $G$ is not of adjoint type, the converse follows by using the map
		$$\mathcal M(G,\alpha) \to \mathcal M(G_{\ad},\alpha)$$
		\noindent given by $(E,\varphi,Q) \mapsto (E/Z(G),\varphi,Q)$. This preserves stability and the twisted multiplicity divisor, as well as the $\C^\times$-action and the image via the Hitchin map. Therefore, if $w(E,\varphi,Q)$ were reduced but 
		$$W^+_{(E,\varphi,Q)} \cap h_G^{-1}(0)$$
		\noindent had nonzero dimension, applying this map would yield that
		$$W^+_{(E/Z(G),\varphi,Q)} \cap h_{G_{\ad}}^{-1}(0)$$
		\noindent would also have nonzero dimension with the same twisted multiplicity divisor, a contradiction.
	\end{proof}
	
	\section{Mirror line bundle for parabolic very stable upward flows}\label{secmirror}
	Let $(E,\varphi,Q)$ be a smooth fixed point of Borel type with twisted multiplicity divisor $w:= w(E,\varphi,Q)$ supported at $D$. In this case, we know (cf. Example \ref{lengthzeroverystable}) that the corresponding upward flow $W^+_{(E,\varphi,Q)}$ is a section of the Hitchin map $h_G$. Therefore, the general philosophy of mirror symmetry being realised (generically over the Hitchin base) by a fibrewise Fourier--Mukai transform \cite[Theorem A]{donagi_langlands_2012} dictates that we should be able to produce a \textit{mirror} line bundle over $\mathcal M_{sp}(G^\vee,\alpha)$ recording how $W^+_{(E,\varphi,Q)}$ shifts from a given Hitchin section $W^+_{(E^0,\varphi^0,Q^0)}$.
	
	We propose a natural candidate for such a mirror line bundle. Suppose there exists a universal parabolic bundle
	$$\mathbb E \to \mathcal M_{sp}(G^\vee,\alpha) \times C,$$
	\noindent that is, a $G^\vee$-bundle endowed with universal parabolic structures
	$$\mathbb Q_i \in H^0(\mathcal M_{sp}(G^\vee,\alpha) \times \{c_i\}, \mathbb E|_{\mathcal M_{sp}(G^\vee,\alpha) \times \{c_i\}}/B^\vee)$$
	\noindent such that $(\mathbb E, \mathbb Q)|_{\{(E,\varphi, Q)\} \times C} \simeq (E,Q)$. Then, we have a reduction $\mathbb (E_{\mathcal M_{sp}(G^\vee,\alpha) \times \{c_i\}})_{B^\vee}$ to the Borel subgroup $B^\vee \subseteq G^\vee$ via $\mathbb Q_i$ and, given a cocharacter $\mu \in X_*(T) = X^*(T^\vee) = X^*(B^\vee)$, we can construct a line bundle $\mu((\mathbb E_{\mathcal M_{sp}(G^\vee,\alpha) \times \{c_i\}})_{B^\vee})$. Recall that our chosen twisted multiplicity divisor $w$ is supported over $D$, therefore at every $c_i$ we have a naturally associated antidominant minuscule cocharacter $\mu_i^{-1} := \mu_{c_i}(E,\varphi)^{-1}$.
	
	\begin{conjecture}\label{mirrorconjecture}
		The line bundle on $\mathcal M_{sp}(G^\vee,\alpha)$ mirror to the structure sheaf $\mathcal O_{W^+_{(E,\varphi,Q)}}$ on $\mathcal M_{sp}(G,\alpha)$ is given by
		$$\mathcal L_w := \bigotimes_{c_i \in D}\mu_{i}^{-1}((\mathbb E_{\mathcal M_{sp}(G^\vee,\alpha) \times \{c_i\}})_{B^\vee})$$
		\noindent up to tensoring by a fixed line bundle (independent of $w$) given by the chosen normalisation for $\mathbb E$.
	\end{conjecture}
	
	There are a few remarks to be taken into consideration. Firstly, we are implicitly assuming that every $\mu_i \in P_{\Delta^\vee}$ lifts to $X^*(T^\vee)$. This is automatic if the group $G=G_{\ad}$ is of adjoint type, but for other groups it imposes an additional restriction. 
	
	Secondly, this is dependent on the existence of $\mathbb E$ as a bundle over $\mathcal M_{sp}(G^\vee,\alpha) \times C$. Of course, the proposed bundles always exist at the level of the moduli stack (see \cite{laszloLineBundlesModuli1997}), but on the moduli space they may only descend with a twist by a $Z(G)$-gerbe corresponding to a class in $H^2(\mathcal M_{sp}(G^\vee,\alpha), Z(G))$ (see e.g. \cite[Section 6.2.1]{hausel_very_2022}). Even if $\mathbb E$ does not exist globally over the moduli space, the specific combination of $\mu_i$ may cause the induced twisting for $\mathcal L_w$ to become trivial, but this is not automatic.
	
	Lastly, in order to check this conjecture using the duality via Fourier--Mukai transforms along the fibres of the Hitchin map, it is convenient to have a \textit{good} description of said fibres (ideally, as honest abelian varieties or their torsors) generically over the Hitchin base. In other words, it is convenient to have smoothness of the corresponding spectral or cameral curve. However, even if we have chosen our parabolic weights $\alpha$ in the interior of the standard Weyl alcove, for a generic point over the Hitchin base the strongly parabolic condition makes $\varphi|_{c_i}$ into a regular \textit{nilpotent} element of $\lieg$. Therefore, the cameral curve is generally not smooth, complicating the approach for general $G$.
	
	Therefore, in the remainder of this section we verify Conjecture \ref{mirrorconjecture} in the most favourable setting of $G=\PGL_n(\C)$ and $G^\vee = \SL_n(\C)$ with weights $\alpha$ (always in the interior of the standard alcove) chosen generically so that $\mathcal M_{sp}(\SL_n(\C),\alpha)$ is smooth. In this way, all $\mu_i^{-1}$ make sense as characters of $T^\vee$, the universal bundle $\mathbb E$ exists globally, uniquely up to normalisation (see \cite[Theorem 1.1]{biswasBrauerGroupModuli2011} for the vanishing of the twisting class) and, as we review in the next subsection, the generic fibres of the Hitchin map are abelian varieties via the Beauville--Narasimhan--Ramanan (BNR) correspondence, as the generic spectral curves are smooth. 
	
	Having restricted to the case $G=\PGL_n(\C)$, in the sequel we work in the usual setting of vector bundles with full flags at the parabolic points, as described in Example \ref{typeafixed} (for $\PGL_n(\C)$, the only change needed for $\SL_n(\C)$ is that $E$ becomes an actual vector bundle, no longer up to line bundle tensoring, and is subject to the constraint $\det(E) \simeq \mathcal O_C$).
	
	\subsection{Spectral description of Hitchin fibres}\label{subsecspectral}
	
	Let us recall the aspects of the BNR correspondence that will be needed in order to check the duality. In our setting with full flags it turns out to be largely similar to the classical BNR correspondence in \cite{beauville_spectral_1989}. We follow \cite[Section 3]{gomezTorelliTheoremModuli2011}, see also \cite{suParabolicHitchinMaps2022} for a more general description involving any flag types.
	
	Let $\mathcal A:= \mathcal A(\SL_n(\C)) \simeq \mathcal A(\PGL_n(\C)) = \bigoplus_{j=2}^nH^0(C,K_C^j(j-1)D)$ denote the Hitchin base and choose a point $a = (a_2,\dots,a_n) \in \mathcal A$. Let $S:=\mathrm{Tot}(K_C(D)) = \underline{\Spec}\Sym^*(K_C^{-1}(-D))$ denote the total space of $K_C(D)$, equipped with the map $\pi_S: S \to C$. There is a tautological section $x \in H^0(S, \pi_S^*(K_C(D)))$, as well as pull-backs $a_j' := \pi_S^*(a_j) \in H^0(S, \pi_S^*(K_C^j((j-1)D)))$. We map these naturally into $H^0(S, \pi_S^*(K_C^j(jD)))$, resulting in a well defined section
	$$s_a := x^n + a_2'x^{n-2} + \dots + a_n'$$
	\noindent in $H^0(S, \pi_S^*(K_C^n(nD)))$.
	
	\begin{definition}
		The \textbf{spectral curve} associated to $a$ is the closed subscheme of $S$ defined by $s_a = 0$, and it is denoted by $C_a$. In other words, it is the cover:
		$$\pi_a : C_a := \underline{\Spec}(\Sym^*(K_C^{-1}(-D))/\mathcal I_a) \to C$$
		\noindent where $\mathcal I_a$ is the ideal sheaf in $\Sym^*(K_C^{-1}(-D))$ generated by the elements of the form $\nu + \nu a_2 + \dots + \nu a_n$ with $\nu$ a local section of $K_C^{-n}(-nD)$. 
	\end{definition}
	
	An important aspect of $C_a$ in our strongly parabolic setting is that it is totally ramified over each $c_i \in C$, as every $a_j$ vanishes at the parabolic punctures. We denote by $c_i^a \in C_a$ the only closed point with $\pi_a(c_i^a) = c_i$. Even then, there is an open dense subset $\mathcal A^{s} \subseteq \mathcal A$ such that $a \in \mathcal A^{s}$ implies that $C_a$ is smooth, see \cite[Lemma 3.1]{gomezTorelliTheoremModuli2011}.
	
	\begin{proposition}[{BNR correspondence, \cite[Lemma 3.2]{gomezTorelliTheoremModuli2011}}]\label{parabolicbnr}
		For $a \in \mathcal A^s$, the fibre $h_{\SL_n(\C)}^{-1}(a)$ is isomorphic to 
		$$P^\vee_a := \{L \in \Pic(C_a) : \det \left((\pi_a)_*L\right) \simeq \mathcal O_C\},$$
		\noindent and the fibre $h_{\PGL_n(\C)}^{-1}(a)$ is isomorphic to
		$$P_a := \Pic(C_a)/\pi_a^*\Pic(C).$$
	\end{proposition}
	
	\begin{remark}\label{torsorfibres}
		The notation $P_a$ and $P_a^\vee$ only seeks to convey that the former corresponds to $G$ and the latter to $G^\vee$. Strictly speaking, they are not dual abelian varieties: $P_a^\vee$ is a torsor for the Prym variety $\Prym(C_a/C)$, while $P_a$ decomposes in $n$ components which are torsors for the identity component, the latter being an honest abelian variety dual of $\Prym(C_a/C)$.  
	\end{remark}
	
	The corresponding point $(E_L,\varphi_L,Q_L)$ in the fibre is recovered from the line bundle $L \to C_a$ by taking $E_L := (\pi_a)_*L$ as the vector bundle and $\varphi_L := (\pi_a)_*x$ as the Higgs field. Since this Higgs field is regular nilpotent at each $c_i \in D$, there is already a unique parabolic structure $Q_{L,i}$ at each $c_i$ which is compatible with $\varphi_L$, so it is automatically determined without any extra data. We can describe it explicitly as
	$$Q_{L,i}^j := (\pi_a)_*\left(\frac{L(-(n-j)c_i^a)}{L(-nc_i^a)}\right).$$
	
	In particular, note for later use that
	$$E_L/Q_{L,i}^j = (\pi_a)_*\left(\frac{L}{L(-(n-j)c_i^a)}\right) = (\pi_a)_*L|_{(n-j)c_i^a}.$$
	
	\subsection{Universal line bundles over the moduli space and duality}\label{subsecuniv}
	
	Now we set up the appropriate universal bundles required to test the duality.
	
	Let
	$$\pi_{\mathscr C} : \mathscr C \to \mathcal A^s \times C$$
	\noindent denote the family of spectral curves, so that $\mathscr{C}|_{\{a\} \times C} \simeq C_a$. In this way, we can describe $h^{-1}_{\SL_n(\C)}(\mathcal A^s) \subseteq \mathcal M_{sp}(\SL_n(\C),\alpha)$ as 
	$$\mathscr P^\vee := \{\mathscr L \in \Pic(\mathscr C/\mathcal A^s) : [\det((\pi_{\mathscr C})_*\mathscr L)] = [\mathcal O_{\mathcal A^s \times C}] \text{ in } \Pic(\mathcal A^s \times C/\mathcal A^s)\},$$ which is a torsor for the relative Prym given by $\Prym(\mathscr{C}/\mathcal A^s \times C)$, and $h^{-1}_{\PGL_n(\C)}(\mathcal A^s)$ as the quotient of relative Picard group schemes $\mathscr P := \Pic(\mathscr C/\mathcal A^s)/\pi_{\mathscr{C}}^*(\mathcal A^s \times \Pic(C))$.
	
	As indicated in Remark \ref{torsorfibres}, we expect a duality of torsors over dual abelian varieties, therefore according to the general SYZ philosophy, certain gerbes (sometimes referred to as \textit{B-fields} following the physical terminology behind mirror symmetry) should enter the picture (see \cite[Section 1]{hausel_mirror_2003}), accounting for the glueing of the chosen local trivialisations of the torsors. However, in our current parabolic setting we can avoid this entirely, precisely because the existence of the very stable upward flows that we are currently considering: these are sections of the Hitchin map which provide a well-defined choice of origin, identifying $P_a$ and each component of $P^\vee_a$ with the abelian varieties themselves in a coherent way across the whole of $\mathcal A^s$. This absence of the necessity for B-fields was already observed in \cite[Theorem 3.5]{gothenTopologicalMirrorSymmetry2019}. We thus let
	$$\mathscr O : \mathcal A^s \to \mathscr P$$
	\noindent be the section given by $\mathscr O(a) := \mathcal O_{C_a}(c_1^a)$ (through the BNR correspondence, this gives precisely the very stable upward flow determined by $\mu(E,\varphi) = \omega_1^\vee \cdot c_1$). By taking tensor powers, this choice determines an origin for every component of the fibres of $h_{\PGL_n(\C)}$ across $\mathcal A^s$.
	
	Similarly, we need to choose such an origin 
	$$\mathscr O^\vee : \mathcal A^s \to \mathscr P^\vee$$
	\noindent for $h_{\SL_n(\C)}$. If $\deg(D)$ is even, this can be achieved straightforwardly by taking $M$ to be a square root of $K_C(D)$ and setting $\mathscr{O^\vee}(a) := \pi_a^*(M^{n-1})$: this merely corresponds to one of the upward flows with trivial multiplicity divisor $\mu(E,\varphi) = 1$. Otherwise, no such upward flows exist, so one needs to select a different very stable upward flow, or any section of $h_{\SL_n(\C)}$. This can always be done, as explained in the proof of \cite[Theorem 3.5]{gothenTopologicalMirrorSymmetry2019}.
	
	We now fix a universal line bundle
	$$\mathbb L \to \mathscr P^\vee \times_{\mathcal A^s} \mathscr C$$
	\noindent rigidified with the following normalisations: 
	\[
	\mathbb L|_{\mathscr O^\vee(\mathcal A^s) \times_{\mathcal A^s} \mathscr C} = \mathcal O_{\mathscr C}, \qquad \mathbb L|_{\mathscr P^\vee \times_{\mathcal A^s} \{c_1^a\}_{a \in \mathcal A_s}} = \mathcal O_{\mathscr P^\vee}.
	\]
	
	Being universal, it satisfies that for $M \in P_a^\vee$, we have $\mathbb L|_{\{M\} \times C_a} \simeq M$. Moreover, over $\mathcal A^s$ we can recover (one of the normalisations of) the universal parabolic bundle $\mathbb E$ via the BNR correspondence, that is
	$$\mathbb E|_{h^{-1}_{SL_n(\C)}(\mathcal A^s) \times C} \simeq (\Id \times \pi_{\mathscr C})_*\mathbb L.$$
	
	Similarly, the filtration $\mathbb Q_i$ of $\mathbb E|_{h^{-1}_{SL_n(\C)}(\mathcal A^s) \times \{c_i\}}$ can be recovered via the BNR correspondence as indicated at the end of Section \ref{subsecspectral}.
	
	The duality for the Hitchin systems over $\mathcal A^s$ is given by the Fourier--Mukai transform, which requires introducing a Poincaré line bundle on the product of our two abelian varieties. However, since we only care about the dual line bundles of sections
	$$\mathscr S : \mathcal A^s \to \mathscr P,$$
	\noindent we may bypass this and write it down merely in terms of $\mathbb L$.
	
	\begin{remark}\label{fouriermukai}
		The dual of the structure sheaf of the image of a section $\mathscr S : \mathcal A^s \to \mathscr P$ which is induced by a line bundle $\mathscr L \to \mathscr C$ can be computed as follows. If $\mathbb P \to \mathscr P \times_{\mathcal A^s} \mathscr P^\vee$ denotes the relative Poincaré line bundle, the desired dual is the line bundle over $\mathscr P^\vee$ given by the slice $\mathbb P|_{\mathscr S(A^s) \times_{A_s} \mathscr P^\vee}$. Said restriction can be related to $\mathbb L$ \cite[Equation 5-18]{meloFourierMukaiAutoduality2019} as follows:
		$$\mathbb P|_{\mathscr S(A^s) \times_{A_s} \mathscr P^\vee} = \det\left(R(p_{\mathscr P^\vee})_*\left(\mathbb L \otimes p_{\mathscr C}^*(\mathscr L\right)\right)\otimes \det(R(p_{\mathscr P^\vee})_* \mathbb L)^{-1}.$$
		
		The rightmost factor accounts for the correct normalisation in a way that computing the mirror of the structure sheaf of our basepoint section $\mathscr O$, which corresponds to the line bundle $\mathscr L$ with divisor given by the preimage of $\mathcal A^s \times \{c_1\}$ in $\mathscr C$, gives $\mathcal O_{\mathscr P^\vee}$.
	\end{remark}
	
	\subsection{Computation of the mirror line bundle}
	
	We now verify Conjecture \ref{mirrorconjecture} for the structure sheaf of $W^+_w := W^+_{(E,\varphi,Q)}$. The first step is to identify the line bundles corresponding to $W^+_w$ via the BNR correspondence. In order to do this, it is useful to have an explicit description of $W^+_w$ in terms of vector bundles, and we can do so by recalling Example \ref{lengthzeroverystable}. Let $j_i$ be the index such that $\mu_i = \omega^\vee_{j_i}$, denote by
	$$D_j := \sum_{c_i \in D : j_i = j}c_i$$
	\noindent and, for notational consistency, set $D_n := D - D_1 - \dots - D_{n-1}$ (similarly, set $j_i = n$ for those $c_i \in D_n$). 
	
	For $a \in \mathcal \mathcal A$, we denote by $(E,\varphi_a,Q)$ the point in $W^+_w \cap h^{-1}_{\PGL_n(\C)}(a)$. Then, we have
	$$E := \bigoplus_{j=0}^{n-1}K_C^{-j}(-jD + \sum_{k=1}^jD_k),$$ 
	$$
	\begin{aligned}
		\varphi_a|_{K_C^{-j}(-jD+\sum_{k=1}^jD_k)} &:= s_{D_{j+1}} \in H^0(\mathcal O_C(D_{j+1})) \\
		&= H^0\Bigg(\Hom\Bigg(K_C^{-j}\left(-jD + \sum_{k=1}^jD_k\right), \\
		&\hspace{2.5cm} K_C^{-(j+1)}\left(-(j+1)D + \sum_{k=1}^{j+1}D_k\right) \otimes K_C(D)\Bigg)\Bigg),
	\end{aligned}
	$$
	\noindent for $j<n-1$, where $s_{D_j}$ denotes the canonical section, and
	$$\varphi_a|_{K_C^{-(n-1)}(-(n-1)D+\sum_{k=1}^{n-1}D_k)} := \sum_{j=2}^na_j,$$
	\noindent where each summand is viewed as
	$$
	\begin{aligned}
		a_j &\in H^0\Bigg(\Hom\Bigg(K_C^{-(n-1)}\left(-(n-1)D+\sum_{k=1}^{n-1}D_k\right), \\
		&\hspace{2.5cm} K_C^{-(n-j)}\left(-(n-j)D+\sum_{k=1}^{n-j}D_k\right) \otimes K_C(D)\Bigg)\Bigg) \\
		&= H^0\left(K_C^j\left(jD-\sum_{k=n-j+1}^{n-1}D_k\right)\right)
	\end{aligned}
	$$
	\noindent under the natural map of $H^0(K^j_C(j-1)D)$ into the right hand side above (since the sum of all $D_j$ is an effective divisor lower or equal than $D$).
	
	The parabolic structure is given explicitly by the permutation indicated by $u(E,Q)$. Set $V_j := K_C^{-j}(-jD + \sum_{k=1}^jD_k)$. Then, for $c_i \in D_j$, the corresponding structure is
	\begin{align*}
		Q_{c_i} := 0 &\subsetneq V_{j-1}|_{c_i} \subsetneq V_{j-2}|_{c_i}\oplus V_{j-1}|_{c_i} \subsetneq \dots \subsetneq \bigoplus_{l=0}^{j-1}V_{l}|_{c_i} \\
		&\subsetneq \bigoplus_{l=0}^{j-1}V_{l}|_{c_i} \oplus V_{n-1}|_{c_i} \\
		&\subsetneq \bigoplus_{l=0}^{j-1}V_{l}|_{c_i} \oplus V_{n-2}|_{c_i} \oplus V_{n-1}|_{c_i} \\
		&\subsetneq \dots \subsetneq E|_{c_i}.
	\end{align*}

	Using this explicit description, we compute the corresponding spectral line bundle.
	
	\begin{proposition}\label{bnrcomputation}
		Let $a \in \mathcal A^s$ and let $(E,\varphi_a,Q) \in \mathcal M_{sp}(\PGL_n(\C),\alpha)$ denote the point $W^+_w \cap h^{-1}_{\PGL_n(\C)}(a)$. Then, under the identification of Proposition \ref{parabolicbnr}, $(E,\varphi_a,Q)$ corresponds to the line bundle
		$$L^w := \bigotimes_{c_i \in D}\mathcal O_{C_a}((n-j_i) \cdot c_i^a),$$
		\noindent where $j_i \in \{1,\dots,n-1\}$ is such that $\mu_i = \omega_{j_i}^\vee$.
	\end{proposition}
	\begin{proof}
		We compute $(\pi_a)_*(L^w)$ locally. Trivialising $K_C(D)$ around $c_i$ (thus $\pi_S^*K_C(D)$ around $c_i^a$), we find that the tautological section $x$ defines a uniformising parameter around $c_i^a$, since the smoothness of $C_a$ implies that $\ord_{c_i^a}(a_n') = n \cdot \ord_{c_i}(a_n) = n$.
		
		We may choose as well a uniformiser $z$ at $c_i$, such that we get identifications of the completed local rings $R = \widehat{\mathcal O_{C,c_i}} \simeq \C[[z]]$, $R' := \widehat{\mathcal O_{C_a,c_i^a}} \simeq \C[[x]] = \C[[z]][x]/(x^n+a_2(z)x^{n-2}+\dots+a_n(z))$ where $z$ equals $x^n$ up to multiplication by an invertible element. We choose $\{1,x,\dots,x^{n-1}\}$ as a basis for $\widehat{\mathcal O_{C_a,c_i^a}}$ over $\C[[z]]$. Around $c_i$, local sections of $\pi_*L^w$ correspond to meromorphic functions
		$$f = \sum_{j=0}^{n-1}f_j(z)x^{j} \in \widehat{\mathcal O_{C_a,c_i^a}}$$
		\noindent with $\ord_{c_i^a}(f) \ge j_i-n$. Since $\ord_{c_i^a}(f_j(z)x^j) = j+n\ord_{c_i}(f_j)$, which are all distinct for $0 \le j \le n-1$, the condition is equivalent to
		$$j+n\ord_{c_i}(f_j) \ge j_i-n,$$
		\noindent that is,
		$$\ord_{c_i}(f_j) \ge -\left(1+\frac{j-j_i}{n}\right).$$
		Therefore, the $f_j$ are allowed a simple pole at $c_i$ whenever $j \ge j_{i}$, so that globally
		$$\pi_* L^w = \bigoplus_{j=0}^{n-1}K_C^{-j}(-jD + \sum_{k=1}^jD_k) = E.$$
		
		It is clear from the above analysis that $(\pi_a)_*(x) = \varphi_a$, and there is a unique parabolic structure compatible with $\varphi_a$ since smoothness of $C_a$ implies that it is regular nilpotent at the parabolic punctures.
	\end{proof}
	
	We are now in position of computing the mirror line bundle.
	
	\begin{proposition}\label{conjecturecheck}
		For $G=\PGL_n(\C)$ and generic weights $\alpha$ in the interior of the standard alcove, the dual of the structure sheaf $\mathcal O_{W^+_w}$ where $w$ is a twisted multiplicity divisor supported at $D$ equals the line bundle $\mathcal L_w$ proposed in Conjecture \ref{mirrorconjecture} over $\mathcal A^s$.
	\end{proposition}
	
\begin{proof}
	We choose the universal bundle $\mathbb E$ normalised as before. Let $\mathbb E_i := \mathbb E|_{h_{\SL_n(\C)}^{-1}(\mathcal A^s) \times \{c_i\}}$, equipped with its filtration
	$$0 \subsetneq \mathbb Q_i^1 \subsetneq \dots \subsetneq \mathbb Q_i^n = \mathbb E_i.$$
	
	According to Proposition \ref{bnrcomputation}, it suffices to verify that the dual of $\mathscr L := \mathcal O_{\mathscr C}((n-j_i) \cdot c_i^\bullet)$, where $c_i^\bullet = \pi_{\mathscr C}^{-1}(\mathcal A^s \times \{c_i\})$, equals (up to fixed normalisation) 
	$$\omega_{j_i}^{-1}(\mathbb E_i) = \det(\mathbb E_i/\mathbb Q_{i}^{j_i}).$$
	
	Thus, by Remark \ref{fouriermukai} it suffices to show
	$$R(p_{\mathscr P^\vee})_*\mathbb L|_{\mathscr P^\vee \times_{\mathcal A^s} \{(n-j_i) \cdot c_i^\bullet\}} \simeq \mathbb E_i/\mathbb Q_{i}^{j_i},$$
	\noindent which is a direct consequence of the description of the filtration as a pushforward via the BNR correspondence, as explained at the end of Section \ref{subsecspectral}.
\end{proof}
	
	\begin{remark}
		Mirror symmetry predicts that, since $W^+_w$ is a holomorphic Lagrangian subvariety of $\mathcal M_{sp}(\PGL_n(\C),\alpha)$ (a \textit{BAA-brane}), the dual line bundle should carry a hyperholomorphic connection (a \textit{BBB-brane}). From the algebraic construction it is only apparent that the one proposed in Conjecture \ref{mirrorconjecture} is holomorphic in one of the structures, so an interesting question is to check their hyperholomorphicity. This could be achieved, as in \cite[Remark 7.1]{hausel_very_2022}, by considering the moduli space as a hyperkähler reduction and checking that the proposed line bundles can be obtained from unitary representations of the corresponding gauge group.
	\end{remark}
	
	\begin{remark}
		Due to the local behaviour of Hecke transformations, if the divisor $w$ is not supported at $D$, the dual can be obtained as in \cite[Section 6]{hausel_very_2022}. In other words, for the very stable upward flow with twisted multiplicity coweight $w$, letting $\mu_c := \mu_c(E,\varphi)$ and considering the representations $\rho^{\mu_c} : G^\vee \to \GL(V^{\mu_c})$ with highest weight $\mu_c$, the mirror bundle for $\mathcal O_{W^+_w}$ should be (up to normalisation) given by
		$$\bigotimes_{c \notin D}\rho^{\mu_c}(\mathbb E|_{\mathcal M_{sp}(G^\vee,\alpha) \times \{c\}}) \otimes \bigotimes_{c_i \in D}\mu_i^{-1}((\mathbb E|_{\mathcal M_{sp}(G^\vee,\alpha) \times \{c_i\}})_{B^\vee}).$$
	\end{remark}
	
	\bibliographystyle{plain}
	\bibliography{bibliography}
	
\end{document}